\documentclass[11pt]{article}

\usepackage{amsmath,amssymb,amsthm,bm,graphicx,color,epsfig}

\usepackage{algorithm}
\usepackage{algorithmic}

\usepackage{setspace}

\newtheorem{theorem}{Theorem}[section]

\newtheorem{algo}[theorem]{Algorithm}

\newcommand{\eps}{\varepsilon}
\newcommand{\p}{\partial}
\newcommand{\Lapl}{\Delta}
\newcommand{\hf}{\frac{1}{2}}
\renewcommand{\P}{\mathcal{P}}
\newcommand{\J}{\mathcal{J}}
\newcommand{\M}{\mathcal{M}}
\renewcommand{\H}{\mathcal{H}}
\newcommand{\ms}{\mathsf}
\newcommand{\wt}[1]{\widetilde{#1}}

\begin{document}

\title{Sweeping Preconditioner for the Helmholtz Equation:\\
  Hierarchical Matrix Representation}
\author{Bj\"orn Engquist and Lexing Ying\\
  Department of Mathematics and ICES, University of Texas, Austin, TX
  78712}

\date{July 2010}
\maketitle

\begin{abstract}
  The paper introduces the sweeping preconditioner, which is highly
  efficient for iterative solutions of the variable coefficient
  Helmholtz equation including very high frequency problems. The first
  central idea of this novel approach is to construct an approximate
  factorization of the discretized Helmholtz equation by sweeping the
  domain layer by layer, starting from an absorbing layer or boundary
  condition. Given this specific order of factorization, the second
  central idea of this approach is to represent the intermediate
  matrices in the hierarchical matrix framework. In two dimensions,
  both the construction and the application of the preconditioners are
  of linear complexity. The GMRES solver with the resulting
  preconditioner converges in an amazingly small number of iterations,
  which is essentially independent of the number of unknowns. This
  approach is also extended to the three dimensional case with some
  success. Numerical results are provided in both two and three
  dimensions to demonstrate the efficiency of this new approach.
\end{abstract}



\section{Introduction}
\label{sec:intro}

This is the first of a series of papers on developing efficient
preconditioners for the numerical solutions of the Helmholtz equation
in two and three dimensions. The efficiency of preconditioners for the
Helmholtz equation in the important high frequency range are at
present much lower than that of preconditioners of typical elliptic
problems. This paper develops efficient preconditioners of the
Helmholtz equation by exploiting the physical property of the wave
phenomena and certain low rank interaction properties of the Green's
function.

Let the domain of interest be the unit box $D = (0,1)^d$ with
$d=2,3$. The time-independent wave field $u(x)$ for $x\in D$ satisfies
\begin{equation}
\Lapl u(x) + \frac{\omega^2}{c^2(x)} u(x) = f(x),
\label{eq:helm}
\end{equation}
where $\omega$ is the angular frequency and $c(x)$ is the velocity
field and $f(x)$ is the external force. Commonly used boundary
conditions are approximations of the Sommerfeld condition which
guarantees that the wave field generated by $f(x)$ propagates out of
the domain. Other boundary condition for part of the boundary will
also be considered.  By appropriately rescaling the system, it is
convenient to assume that the mean of $c(x)$ is around $1$. Then
$\frac{\omega}{2\pi}$ is the (average) wave number of this problem and
$\lambda = \frac{2\pi}{\omega}$ is the (typical) wavelength.

The Helmholtz equation is ubiquitous since it is the root of almost
all linear wave phenomena. Applications of the Helmholtz equation are
abundant in acoustics, elasticity, electromagnetics, quantum
mechanics, and geophysics. As a result, efficient and accurate
numerical solution of the Helmholtz problem is one of the urgent
problems in computational mathematics. This is, however, a very
difficult problem due to two main reasons. Firstly, in a typical
engineering application, the Helmholtz equation is discretized with at
least $8$ to $16$ points per wavelength. Therefore, the number of
samples $n$ in each dimension is proportional to $\omega$, the total
number of samples $N$ is $n^d = O(\omega^d)$, and the discrete system
of the Helmholtz equation is of size $O(\omega^d) \times
O(\omega^d)$. In the high frequency range when $\omega$ is large, this
is an enormous system. Secondly, as the discrete system is highly
indefinite and has a very oscillatory Green's function due to the wave
nature of the Helmholtz equation, most of the modern multiscale
techniques developed for elliptic or parabolic problems are no longer
effective.

\subsection{Approach and contribution}

In this paper, we propose a {\em sweeping preconditioner} for the
iterative solution of the Helmholtz equation. In all examples, the
Helmholtz equation is discretized by centered finite differences,
i.e., the 5-point stencil in 2D and the 7-point stencil in 3D.

In the 2D case, this new preconditioner is based on a block $LDL^t$
factorization of the discrete Helmholtz operator. The overall process
is to eliminate the the unknowns layer by layer, starting from an
layer with Sommerfeld condition specified. The main observation is
that each intermediate $n\times n$ Schur complement matrix of this
block $LDL^t$ factorization roughly corresponds to the restriction of
a half-space Green's function to a line and these Schur complement
matrices are highly compressible with low-rank off-diagonal
blocks. Representing and manipulating these matrices in the
hierarchical matrix framework \cite{BormGrasedyckHackbusch:06}
requires only $O(n\log n)$ space and $O(n\log^2 n)$ steps. As a
result, the block $LDL^t$ factorization takes $O(n^2\log^2 n) = O(N
\log^2 N)$ steps. The resulting block $LDL^t$ factorization serves as
an excellent preconditioner for the discrete Helmholtz system and
applying it to any vector takes only $O(n^2 \log n) = O(N \log N)$
steps using the hierarchical matrix framework. By combining this
preconditioner with GMRES, we obtain iteration numbers that are almost
independent of $\omega$. In a typical example with a computational
domain of $256\times 256$ wavelengths and four million unknowns, only
3 to 4 GMRES iterations are required (see Section \ref{sec:2Dnum}).

We also extend this approach to the 3D case and construct an
approximate block $LDL^t$ factorization by eliminating the unknowns
face by face, starting from a face with Sommerfeld condition
specified. Though each intermediate $n^2 \times n^2$ Schur complement
matrix still corresponds to the restriction of a half-space Green's
function to a face, the off-diagonal parts may not be of numerically
low-rank. However, since the goal is to construct a preconditioner, we
still represent and manipulate these matrices under the hierarchical
matrix framework. Numerical results show that applying the resulting
preconditioner is highly efficient and the preconditioned GMRES solver
converges in a small number of iterations, weakly depending on
$\omega$.

The main observation of the sweeping preconditioner comes from the
analytic low-rank property of the Green's function of the {\em
  continuous} Helmholtz operator. On the other hand, the algorithms
construct the approximation to the Green's function of the {\em
  discrete} Helmholtz operator. It is important that this Green's
function is calculated from the discretized problem to be solved
numerically and is not an independent approximation of the continuous
analogue.

\subsection{Related work}

There has been a vast literature on developing efficient algorithms
for the Helmholtz equation. A wide class of methods for special sets
of solutions are based on asymptotic expansion of the solution
$u(x)$. These techniques of geometric optics type are efficient when
$\omega$ is very large. A review article on these methods can be found
in \cite{EngquistRunborg:03}. There is also a class of methods based
on boundary integral or volumetric integral representations. These
integral equation methods can be highly efficient for piecewise
constant velocity fields when combined with fast summation methods
such as the fast multipole methods and the fast Fourier transforms
\cite{BleszynskiBleszynskiJaroszewicz:96,BrunoMcKay:05,
  EngquistYing:07,EngquistYing:09,Rokhlin:90,Rokhlin:93}. Here we will
focus on the methods that discretize the Helmholtz equation directly.

The most efficient direct methods for solving the discretized
Helmholtz systems are the multifrontal methods or their pivoted
versions \cite{DuffReid:83,George:73,Liu:92}. The multifrontal methods
exploit the locality of the discrete operator and construct an $LDL^t$
factorization based on a hierarchical partitioning of the
domain. Their computational costs depend quite strongly on the
dimensionality. In 2D, for a problem with $N = n\times n$ unknowns, a
multifrontal method takes $O(N^{3/2})$ steps and $O(N \log N)$ storage
space. The prefactor is usually rather small, making the multifrontal
methods effectively the default choice for the 2D Helmholtz
problem. In 3D, for a problem with $N=n\times n \times n$ unknowns, a
multifrontal method takes $O(n^6)= O(N^2)$ steps and $O(n^4) =
O(N^{4/3})$ storage space. For large scale 3D problems, they can be
very costly.

In the setting of the elliptic operators, the intermediate matrices of
the multifrontal methods can be well approximated using hierarchical
matrix algebra and this allows one to bring the cost down to linear
complexity in both 2D and 3D
\cite{Martinsson:09,XiaChandrasekaranGuLi:09}. This is, however, not
true for the Helmholtz operator. As we pointed out, the sweeping
preconditioner introduced in this paper is also based on constructing
an $LDL^t$ factorization of the Helmholtz operator. However, due to
its specific sweeping (or elimination) order, which is very different
from the one of the multifrontal methods, we are able to represent the
intermediate matrices in a more effective way and obtain a highly
efficient preconditioner.

There has been a surge of developments in the category of iterative
methods for solving the Helmholtz equation. The following discussion
is by no means complete and more details can be found in
\cite{Erlangga:08}.

Standard multigrid methods do not work well for the Helmholtz equation
for several reasons. The most important one is that the oscillations
on the scale of the wavelength cannot be carried on the coarse
grids. Several methods have been proposed to remedy this
\cite{BrandtLivshits:97,ElmanErnstOLeary:01,FishQu:00,
  LeeManteuffelMcCormickRuge:00,LivshitsBrandt:06,VanekMandelBrezina:98}. For
example in \cite{BrandtLivshits:97,LivshitsBrandt:06}, Brandt and
Livshits proposed the wave-ray method. This method uses the standard
smoothers to remove the coarse and fine components of the residue. It
also decomposes the component that oscillates on the scale of the
wavelength into rays pointing at different directions. Each ray is
further represented with a phase and amplitude representation, and the
amplitude is relaxed with an anisotropic grid aligned with the ray
direction. A limitation of the wave-ray method is however that the
method is essentially restricted only to the case of constant velocity
field. We would like to point out that there is a connection between
the wave ray method and the sweeping preconditioner proposed in this
paper, as both methods exploit the analytic behavior of the Green's
function of the Helmholtz equation. The wave ray method relies on the
Green's function over the whole domain, while the sweeping
preconditioner uses its restriction on a single layer.

Several other methods
\cite{BenamouDespres:97,Despres:91,SusanResigaAtassi:98} leverage the
idea of domain decomposition. These methods are typically quite
suitable for parallel implementation, as the computation in each
subdomain can essentially be done independently. However, convergence
rates of the these methods are usually quite slow \cite{Erlangga:08}.

Another class of methods
\cite{BaylissGoldsteinTurkel:83,ErlanggaOosterleeVuik:06,ErlanggaVuikOosterlee:04,LairdGiles:02}
that attracts a lot of attention recently preconditions the Helmholtz
operator with a shifted Laplacian operator,
\[
\Lapl - \frac{\omega^2}{c^2(x)} ( \alpha + i \beta), \quad \alpha>0,
\]
to improve the spectrum property of the discrete Helmholtz
system. Since the shifted Laplacian operator is elliptic, standard
algorithms such as multigrid can be used for its inversion. These
methods offer quite significant improvements for the convergence rate,
but the reported number of iterations typically still grow linearly
with respect to $\omega$ and are much larger than the iteration
numbers produced by the sweeping preconditioner.

Several other constructions of preconditioners
\cite{BenziHawsTuma:00,GanderNataf:05,OseiKuffuorSaad:09} are based on
incomplete LU (ILU) decomposition, i.e., generating only a small
portion of the entries of the LU factorization of the discrete
Helmholtz operator and applying this ILU decomposition as a
preconditioner. Recent approaches based on ILUT (incomplete LU
factorization with thresholding) and ARMS (algebraic recursive
multilevel solver) have been reported in
\cite{OseiKuffuorSaad:09}. These ILU preconditioners bring down the
number of iterations quite significantly, however the number of
iterations still scale typically linearly in $\omega$. In connection
with the ILU preconditioners, the sweeping preconditioner can be
viewed as an approximate LU (ALU) preconditioner: instead of keeping
only a few selected entries, it approximates the whole inverse
operator more accurately in a more sophisticated and effective form,
thus resulting in substantially better convergence properties.

\subsection{Contents}

The rest of this paper is organized as follows. Section
\ref{sec:2Dpre} presents the sweeping preconditioner in the 2D case
and Section \ref{sec:2Dnum} reports the 2D numerical results. We
extend this approach to the 3D case in Section \ref{sec:3Dpre} and
report the 3D numerical results in Section \ref{sec:3Dnum}. Finally,
Section \ref{sec:conc} discusses some future directions of this work.

\section{Preconditioner in 2D}
\label{sec:2Dpre}

\subsection{Discretization}

Recall that the computational domain is $D = (0,1)^2$. Let us assume
for simplicity that the Sommerfeld condition is specified over the
whole boundary. One standard way of incorporating the Sommerfeld
boundary condition into \eqref{eq:helm} is to use the perfectly
matched layer (PML)
\cite{Berenger:94,ChewWeedon:94,Johnson:10}. Introduce
\begin{equation}
\sigma(t) = 
\begin{cases}
  \frac{C}{\eta}\cdot \left( \frac{t-\eta}{\eta} \right)^2 & t \in [0,\eta]\\
  0 & t\in [\eta, 1-\eta] \\
  \frac{C}{\eta}\cdot \left( \frac{t-1+\eta}{\eta} \right)^2 & t \in [1-\eta,1],
\end{cases}
\label{eq:sigma}
\end{equation}
and 
\[
s_1(x_1) = \left( 1+i\frac{\sigma(x_1)}{\omega} \right)^{-1},\quad
s_2(x_2) = \left( 1+i\frac{\sigma(x_2)}{\omega} \right)^{-1}.
\]
Here $\eta$ is typically about one wavelength and $C$ is an
appropriate positive constant independent of $\omega$. The PML
approach replaces $\p_1$ with $s_1(x_1) \p_1$ and $\p_2$ with
$s_2(x_2)\p_2$, which effectively provides a damping layer of width
$\eta$ near the boundary of the domain $[0,1]^2$. The resulting
equation is
\begin{eqnarray*}
\left( (s_1\p_1)(s_1\p_1) + (s_2\p_2)(s_2\p_2) + \frac{\omega^2}{c^2(x)} \right) u = f && x\in D=[0,1]^2,\\
u = 0 && x \in \p D.
\end{eqnarray*}
Without loss of generality, we assume that $f(x)$ is supported inside
$[\eta,1-\eta]^2$ (away from the PML). Dividing the above
equation by $s_1(x_1) s_2(x_2)$ results
\[
\left( \p_1\left(\frac{s_1}{s_2} \p_1\right) + \p_2\left(\frac{s_2}{s_1} \p_2\right) + 
  \frac{\omega^2}{s_1s_2 \cdot c^2(x)} \right) u = f.
\]
The advantage of working with this equation is that it is symmetric,
which offers some convenience from the algorithmic point of view.  We
discretize the domain with a Cartesian grid with spacing $h =
1/(n+1)$. In order to discretize each wavelength with a couple of
points, the number of points $n$ in each dimension needs to be
proportional to $\omega$. We assume that $n$ to be an integer power of
two for simplicity. The interior points of this grid are
\[
\P = \{ p_{i,j} = (ih,jh): 1\le i,j \le  n\}
\]
(see Figure \ref{fig:2Dgrid} (left)) and the total number of points
$N$ is equal to $n^2$.

\begin{figure}[h!]
  \begin{center}
    \includegraphics{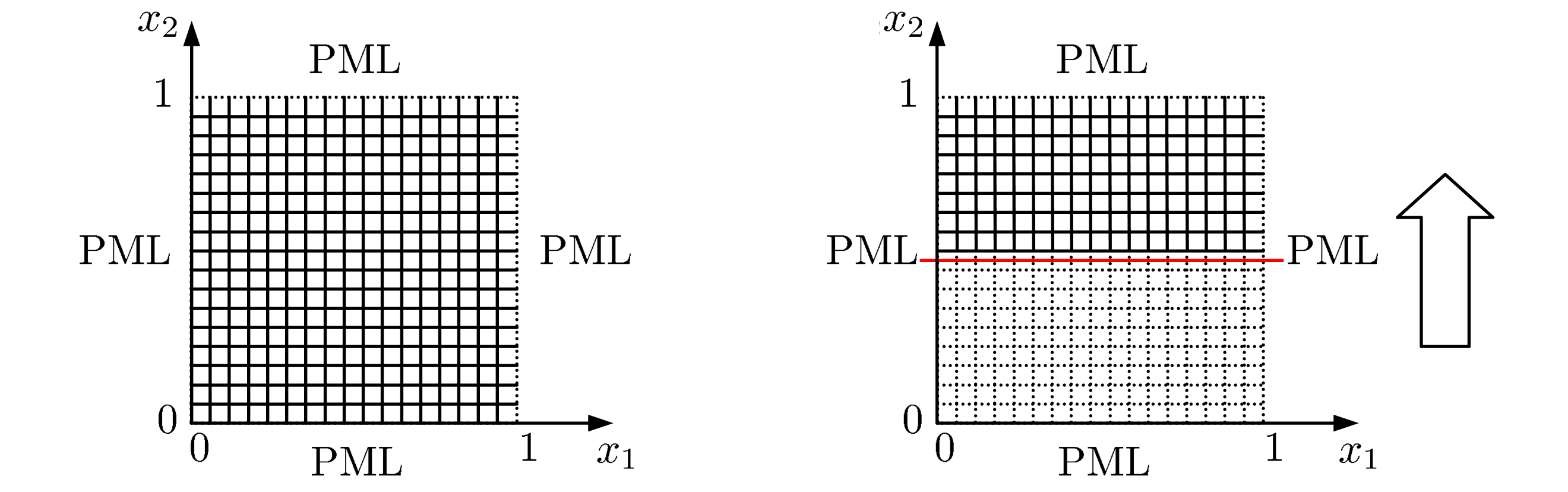}
  \end{center}
  \caption{Left: Discretization grid in 2D. Right: Sweeping order in
    2D. The dotted grid indicates the unknowns that have already been
    eliminated.}
  \label{fig:2Dgrid}
\end{figure}

We denote by $u_{i,j}$, $f_{i,j}$, and $c_{i,j}$ the values of $u(x)$,
$f(x)$, and $c(x)$ at point $p_{i,j}=(ih,jh)$. The standard 5-point
stencil finite difference method writes down the equation at points in
$\P$ using central difference. The resulting equation at $p_{i,j} =
(ih,jh)$ is
\begin{multline}
  \frac{1}{h^2} \left(\frac{s_1}{s_2}\right)_{i-\hf,j} u_{i-1,j} + 
  \frac{1}{h^2} \left(\frac{s_1}{s_2}\right)_{i+\hf,j} u_{i+1,j} + 
  \frac{1}{h^2} \left(\frac{s_2}{s_1}\right)_{i,j-\hf} u_{i,j-1} + 
  \frac{1}{h^2} \left(\frac{s_2}{s_1}\right)_{i,j+\hf} u_{i,j+1} \\
  + \left(
    \frac{\omega^2}{(s_1s_2)_{i,j}\cdot c_{i,j}^2}
    - \left(\cdots
    \right) 
  \right)
  u_{i,j} = f_{i,j}
  \label{eq:helmd}
\end{multline} 
with $u_{i',j'}$ equal to zero for $(i',j')$ that violates $1\le
i',j'\le n$. Here $(\cdots)$ stands for the sum of the four
coefficients appeared in the first line. We order $u_{i,j}$ row by row
starting from the first row $j=1$ and denote the vector containing all
unknowns by
\[
u = \left(u_{1,1}, u_{2,1},\ldots,u_{n,1}, \ldots,u_{1,n}, u_{2,n},\ldots,u_{n,n}\right)^t.
\]
Similarly, $f_{i,j}$ are ordered in the same way and the vector $f$ is
\[
f = \left(f_{1,1}, f_{2,1},\ldots,f_{n,1}, \ldots,f_{1,n}, f_{2,n},\ldots,f_{n,n}\right)^t.
\]
Then \eqref{eq:helmd} takes the form $A u = f$. We further define
$\P_m$ to be the unknowns in the $m$-th row
\[
\P_m = \{p_{1,m},\ldots,p_{n,m}\} 
\]
and introduce
\[
u_m = \left( u_{1,m}, u_{2,m},\ldots,u_{n,m} \right)^t\quad\text{and}\quad
f_m = \left( f_{1,m}, f_{2,m},\ldots,f_{n,m} \right)^t.
\]
Then
\[
u = (u_1^t, u_2^t, \ldots, u_n^t)^t\quad\text{and}\quad
f = (f_1^t, f_2^t, \ldots, f_n^t)^t.
\]
Using these notations, the system $Au=f$ takes the following
tridiagonal block form
\[
\begin{pmatrix}
  A_{1,1} & A_{1,2} & & \\
  A_{2,1} & A_{2,2} & \ddots & \\
  & \ddots & \ddots & A_{n-1,n}\\
  & & A_{n,n-1} & A_{n,n}
\end{pmatrix}
\begin{pmatrix}
  u_1\\
  u_2\\
  \vdots\\
  u_n
\end{pmatrix}
=
\begin{pmatrix}
  f_1\\
  f_2\\
  \vdots\\
  f_n
\end{pmatrix}
\]
where $A_{m,m}$ are tridiagonal matrices and $A_{m,m-1} = A_{m-1,m}^t$
are diagonal matrices.

We introduce the notion of the {\em sweeping factorization}, which is
essentially a block $LDL^t$ factorization of $A$ that eliminates the
unknowns layer by layer. Starting from the first row of unknowns
$\P_1$ gives
\[
A = L_1
\begin{pmatrix}
  S_1 &  & & \\
  & S_2 & A_{2,3} & \\
  & A_{3,2} & \ddots & \ddots \\
  & & \ddots & \ddots
\end{pmatrix}
L_1^t
\]
where $S_1=A_{1,1}$, $S_2 = A_{2,2} - A_{2,1} S_1^{-1} A_{1,2}$, and
the matrix $L_1$ is a block lower-triangular matrix given by
\[
L_1(\P_2,\P_1) = A_{2,1} S_1^{-1},\quad
L_1(\P_i,\P_i) = I\;\;(1\le i \le n),\quad
\text{and zero otherwise}.
\]
Repeating this process over all $\P_m$ for $m=2,\ldots,n-1$ gives
\begin{equation}
A = L_1 \cdots L_{n-1}
\begin{pmatrix}
  S_1 & & & \\
  & S_2 & & \\
  & & \ddots & \\
  & & & S_n\\
\end{pmatrix}
L_{n-1}^t \cdots L_1^t,
\label{eq:Afact}
\end{equation}
where $S_m = A_{m,m} - A_{m,m-1} S_{m-1}^{-1} A_{m-1,m}^t$ for $m=2,3,\ldots,
n$. The matrix $L_m$ is given by
\[
L_m(\P_{m+1},\P_m) = A_{m+1,m} S_m^{-1},\quad
L_m(\P_i,\P_i) = I\;\;(1\le i \le n),\quad
\text{and zero otherwise}.
\]
This process is illustrated graphically in Figure \ref{fig:2Dgrid}
(right). Inverting this factorization \eqref{eq:Afact} for $A$ gives
the following formula for $u$:
\[
u = (L_1^t)^{-1}\cdots (L_{n-1}^t)^{-1} 
\begin{pmatrix}
  S_1^{-1} & & & \\
  & S_2^{-1} & & \\
  & & \ddots & \\
  & & & S_n^{-1} \\
\end{pmatrix}
L_{n-1}^{-1} \cdots L_1^{-1} f.
\]

Algorithmically, the construction of the sweeping factorization of $A$
can be summarized as follows by introducing $T_m = S_m^{-1}$.
\begin{algo}
  Construction of the sweeping factorization of $H$.
  \label{alg:setupext}
\end{algo}
\begin{algorithmic}[1]
  \STATE $S_1 = A_{1,1}$ and $T_1 = S_1^{-1}$.
  \FOR{$m=2,\ldots,n$}
  \STATE $S_m = A_{m,m} - A_{m,m-1} T_{m-1} A_{m-1,m}$ and $T_m = S_m^{-1}$.
  \ENDFOR
\end{algorithmic}
Since $S_m$ and $T_m$ are in general dense matrices of size $n\times
n$, the cost of the construction algorithm is of order $O(n^4) =
O(N^2)$. The computation of $u = A^{-1}f$ is carried out in the
following algorithm once the sweeping factorization is ready.
\begin{algo}
  Computation of $u=A^{-1}f$ using the sweeping factorization of $A$.
  \label{alg:solveext}
\end{algo}
\begin{algorithmic}[1]
  \FOR{$m=1,\ldots,n$}
  \STATE $u_m = f_m$
  \ENDFOR
  \FOR{$m=1,\ldots,n-1$}
  \STATE $u_{m+1} = u_{m+1} - A_{m+1,m} (T_m u_m)$
  \ENDFOR
  \FOR{$m=1,\ldots,n$}
  \STATE $u_m = T_m u_m$
  \ENDFOR
  \FOR{$m=n-1,\ldots,1$}
  \STATE $u_m = u_m - T_m (A_{m,m+1} u_{m+1})$
  \ENDFOR
\end{algorithmic}
Obviously the computations of $T_m u_m$ in the second and the third
loops only need to be carried out once. However, we prefer to write
the algorithm this way for simplicity. The cost of computing $u$ is of
order $O(n^3) = O(N^{3/2})$. This is $O(N^{1/2})$ times more expensive
compared to the multifrontal methods, therefore Algorithms
\ref{alg:setupext} and \ref{alg:solveext} themselves are not very
useful.

\subsection{Main observation}

Let us consider the meaning of the matrix $T_m = S_m^{-1}$. Consider
only the top-left $m\times m$ blocks of the factorization
\eqref{eq:Afact}.
\begin{equation}
\begin{pmatrix}
  A_{1,1} & A_{1,2} & & \\
  A_{2,1} & A_{2,2} & \ddots & \\
  & \ddots & \ddots & A_{m-1,m}\\
  & & A_{m-1,m} & A_{m,m}
\end{pmatrix}
=
L_1 \cdots L_{m-1}
\begin{pmatrix}
  S_1 & & & \\
  & S_2 & & \\
  & & \ddots & \\
  & & & S_m\\
\end{pmatrix}
L_{m-1}^t \cdots L_1^t,
\label{eq:Afactm}
\end{equation}
where the $L_k$ matrices are redefined to their restrictions to the
top-left $m \times m$ blocks. The matrix on the left is in fact the
discrete Helmholtz operator of the half space problem below
$x_2=(m+1)h$ and with zero boundary condition on
$x_2=(m+1)h$. Inverting the above factorization gives
\begin{equation}
\begin{pmatrix}
  A_{1,1} & A_{1,2} & & \\
  A_{2,1} & A_{2,2} & \ddots & \\
  & \ddots & \ddots & A_{m-1,m}\\
  & & A_{m,m-1} & A_{m,m}
\end{pmatrix}^{-1}
=
(L_1^t)^{-1} \cdots (L_{m-1}^t)^{-1}
\begin{pmatrix}
  S_1^{-1} & & & \\
  & S_2^{-1} & & \\
  & & \ddots & \\
  & & & S_m^{-1}\\
\end{pmatrix}
L_{m-1}^{-1} \cdots L_1^{-1}.
\label{eq:Afacminv}
\end{equation}
The matrix on the left is the discrete half-space Green's function of
the Helmholtz operator with zero boundary condition. On the right
side, due to the definition of the matrices $L_1,\ldots,L_{m-1}$, the
$(m,m)$-th block of the whole product is exactly equal to
$S_m^{-1}$. Therefore,
\begin{center}
  $T_m = S_m^{-1}$ is the discrete half-space Green function of the
  Helmholtz operator \\
  with zero boundary at $x_2=(m+1)h$, restricted to the points on
  $x_2=mh$.
\end{center}
The main observation of our approach is that 
\begin{center}
  $T_m$ and $S_m$ are highly compressible with numerically low-rank
  off-diagonal blocks.
\end{center}
The following theorem shows that this is true for the continuous
half-space Green's function for the case of constant velocity field
$c(x)=1$.

\begin{theorem} \label{thm:lowrank}
  Let 
  \[
  Y = \left\{ p_{i,m} = (ih,mh), i=1,\ldots,\frac{n}{2} \right\} \quad\text{and}\quad
  X = \left\{ p_{i,m} = (ih,mh), i=\frac{n}{2}+1,\ldots,n \right\},
  \]
  and $G$ be the (continuous) half-space Green's function of the
  Helmholtz operator for the domain $(-\infty,\infty) \times
  (-\infty,(m+1)h)$ with zero boundary condition. Then $(G(x,y))_{x
    \in X, y\in Y}$ is numerically low-rank. More precisely, for any
  $\eps>0$, there exist a constant $R = O(\log\omega |\log \eps|^2)$,
  functions $\{\alpha_r(x)\}_{1\le r \le R}$ for $x\in X$ and
  functions $\{\beta_r(y)\}_{1 \le r \le R}$ for $y\in Y$ such that
  \[
  \left| G(x,y) - \sum_{r=1}^R \alpha_r(x) \beta_r(y) \right| \le \eps \quad\text{for}\quad x\in X, y\in Y.
  \]
\end{theorem}

The proof of this theorem relies on the following theorem from
\cite{MartinssonRokhlin:07}. Let $H_0(\cdot)$ be the $0$-th order
Hankel function of the first kind.
\begin{theorem} \label{thm:martinsson} Let $\omega$ be the angular
  frequency and $\lambda = 2\pi/\omega$.  Let $W>0$. There exists
  $C(W)$ such that, for $L>0$, $\eps>0$, and $S>C(W) |\log\eps| \cdot
  \frac{2\pi}\lambda$, there exist a constant $J \le \log(\omega L)
  |\log\eps|^2$, functions $\{\phi_j(x)\}_{1\le j \le J}$, and
  functions $\{\chi_j(y)\}_{1 \le j \le J}$ such that
  \[
  \left|H_0(\omega|x-y|) - \sum_{j=1}^J \phi_j(x) \chi_j(y) \right| \le \eps
  \]
  for 
  \[
  y\in [-L, -S/2]\times [-W/2,W/2] \quad\text{and}\quad  x\in [S/2, L]\times [-W/2,W/2].
  \]
\end{theorem}

\begin{figure}[h!]
  \begin{center}
    \includegraphics{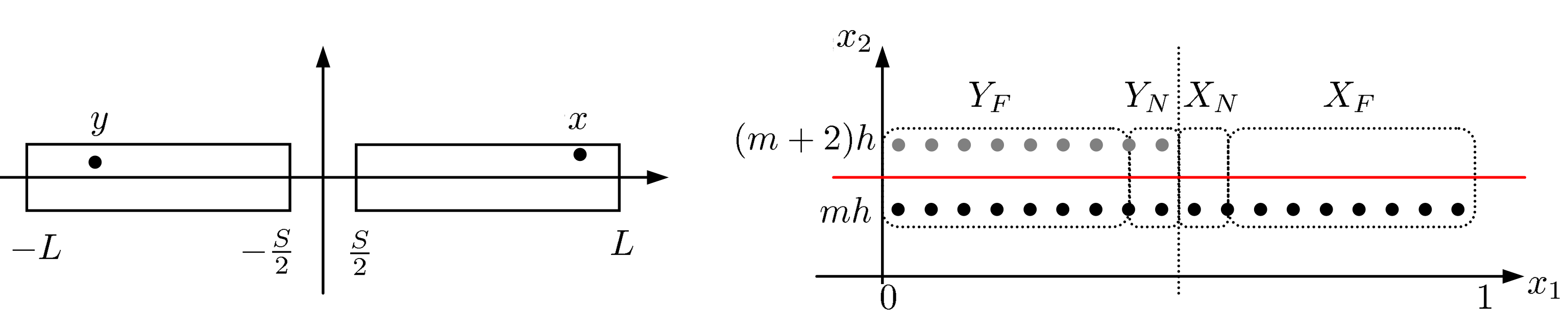}
  \end{center}
  \caption{Left: The setting of Theorem \ref{thm:martinsson}. Right:
    The setting of Theorem \ref{thm:lowrank}.  }
  \label{fig:thms}
\end{figure}

The setting of this theorem is illustrated in Figure \ref{fig:thms}
(left). Using Theorem \ref{thm:martinsson}, the proof of Theorem
\ref{thm:lowrank} goes as follows.
\begin{proof}[Proof of Theorem \ref{thm:lowrank}]
  Let $W=2h$. We partition the set $X$ into the union of the near set
  $X_N$ and the far set $X_F$ depending on the distance from $Y$.
  \begin{eqnarray*}
    &&X_N = \left\{ p=(p_1,p_2)\in X, p_1 \le \hf + \hf C(W)|\log(\eps/2)|\lambda \right\} \\
    &&X_F = \left\{ p=(p_1,p_2)\in X, p_1 >   \hf + \hf C(W)|\log(\eps/2)|\lambda \right\}.
  \end{eqnarray*}
  Similarly, $Y$ is partitioned into the union of $Y_N$ and $Y_F$
  \begin{eqnarray*}
    &&Y_N = \left\{ p=(p_1,p_2)\in Y, p_1 \ge \hf - \hf C(W)|\log(\eps/2)|\lambda \right\} \\
    &&Y_F = \left\{ p=(p_1,p_2)\in Y, p_1 <   \hf - \hf C(W)|\log(\eps/2)|\lambda \right\}.
  \end{eqnarray*}
  See Figure \ref{fig:thms} (right). These partitionings introduce a
  natural block structure for the matrix $(G(x,y))_{x\in X, y\in Y}$:
  \begin{equation}
  \begin{pmatrix}
    (G(x,y))_{x\in X_N, y\in Y_N} && (G(x,y))_{x\in X_N, y\in Y_F} \\
    (G(x,y))_{x\in X_F, y\in Y_N} && (G(x,y))_{x\in X_F, y\in Y_F}
  \end{pmatrix}
  \label{eq:GXY}
  \end{equation}
  Let $p = \lambda/h$ be the number of points per wavelength. It is
  clear from the definition of $X_N$ and $Y_N$ that each of them has
  at most $ \hf C(W) |\log \eps|\lambda /h = \hf C(2h) |\log \eps| p$
  points. Hence the ranks of the $(1,1)$, $(1,2)$, and
  $(2,1)$ blocks of \eqref{eq:GXY} are all bounded from above by
  $\hf C(2h) |\log \eps| p$.
  
  Let us consider the $(2,2)$ block. Define $\M(Y_F)$ to be the
  mirror image set of the set $Y_F$ with respect to the line
  $x_2=(m+1)h$. Due to the zero Dirichlet boundary condition at
  $x_2=(m+1)h$, for $x\in X_F$ and $y\in Y_F$
  \[
  G(x,y) = H_0(\omega|x-y|) - H_0(\omega|x-\M(y)|)
  \]
  where $\M(y)\in \M(Y_F)$ is the mirror image of $y$. $Y_F \bigcup
  \M(Y_F)$ is contained in the box
  \[
  \left[0, \hf - \hf C(W)|\log(\eps/2)|\lambda\right]    \times [mh,(m+2)h]
  \]
  and $X_F$ is in
  \[
  \left[\hf + \hf C(W)|\log(\eps/2)|\lambda, 1\right]    \times [mh,(m+2)h].
  \]
  Since the distance between these two boxes is
  $C(W)|\log(\eps/2)|\lambda$ and their widths are bounded by $1$,
  Theorem \ref{thm:martinsson} guarantees that there exist a constant
  $J \le \log(\omega) |\log(\eps/2)|^2$, functions
  $\{\phi_j(x)\}_{1\le j \le J}$, and functions $\{\chi_j(y)\}_{1\le j
    \le J}$ such that
  \[
  \left|H_0(\omega|x-y|) - \sum_{j=1}^J \phi_j(x) \chi_j(y) \right| \le \frac{\eps}{2}
  \]
  for $x\in X_F$ and $y\in Y_F \bigcup \M(Y_F)$. This implies that
  \[
  \left| G(x,y) - \sum_{j=1}^J \phi_j(x) (\chi_j(y)-\chi_j(\M(y))) \right| \le \eps.
  \]
  
  Combining this with the estimates for the other three blocks shows
  that there exists $R = \frac{3}{2} C(2h) |\log(\eps/2)| p +
  \log(\omega) |\log(\eps/2)|^2 = O(\log(\omega) |\log\eps|^2)$ and
  functions $\{\alpha_r(x)\}_{1 \le r \le R}$ for $x\in X$ and
  functions $\{\beta_r(y)\}_{1\le r \le R}$ for $y\in Y$ such that
  \[
  \left| G(x,y) = \sum_{r=1}^R \alpha_r(x) \beta_r(y) \right| \le \eps \quad\text{for}\quad x\in X, y\in Y.
  \]
\end{proof}


For a fixed $\eps$, Theorem \ref{thm:lowrank} shows that the rank $R$
grows logarithmically with respect to $\omega$ (and thus to
$n$). Though the theorem states the result under the case that $X$
contains the points on the left half and $Y$ contains the points on
the right half, it also applies to any disjoint intervals $X$ and $Y$
on $x_2=mh$ due to the translational invariance of the kernel $G(x,y)$
in the $x_1$ direction. It is also clear that, when $X$ and $Y$ are
well-separated from each other, the actual rank $R$ should be smaller.

Theorem \ref{thm:lowrank} can be extended to the case of smooth
layered media where the velocity variation only depends on $x_1$. In
this case, the restriction of the Green's function to $x_2=mh$ does
not develop caustics. Therefore, the geometric optics representation
$A(x,y) e^{i\omega \Phi(x,y)}$ of the Green's function for $x\in X$
and $y\in Y$ can be made sufficiently accurate as long as $X$ and $Y$
are well-separated. The amplitude $A(x,y)$ is numerically low-rank due
to its smoothness. The phase term is also numerically low-rank since
for the layered media $\Phi(x,y) = \tau(x) - \tau(y)$ where
$\tau(\cdot)$ is the travel time function from a fixed
point. Therefore, their product, the Green's function $G(x,y)$, is
also numerically low rank for well-separated $X$ and $Y$.

Numerical experiments confirm the result of Theorem
\ref{thm:lowrank}. For the constant coefficient case $c(x)=1$ with
$\frac{\omega}{2\pi}= 32$ ($n=256$), Figure \ref{fig:ranks} (left)
shows the numerical ranks of the off-diagonal blocks of $T_m$ for
$m=128$. For each off-diagonal block, the singular values of this
block are calculated and the value in each block indicates the number
of singular values that are greater than $10^{-6}$. For non-constant
velocity fields $c(x)$, the rank estimate would depend on the
variations in $c(x)$ and numerical results suggest that the
off-diagonal blocks of $T_m$ and $S_m$ still admit this low-rankness
property for a wide class of $c(x)$. An example for the non-constant
velocity field is given in Figure \ref{fig:ranks} (middle).

We would like to emphasize that both the Sommerfeld boundary condition
and the layer-by-layer sweeping order are essential. To illustrate
that, we perform the same test with the same threshold $10^{-6}$ but
with zero Dirichlet boundary condition. The result of $T_m$ for
$m=128$ is plotted in Figure \ref{fig:ranks} (right). It is clear that
the rank of a off-diagonal block is much higher and grows almost
linearly with respect to the size of the block. This clearly shows the
importance of the Sommerfeld boundary condition. A similar matrix
$T_m$ would also appear if one adopts different elimination orders
such as the one of multifrontal methods or the one proposed in
\cite{Martinsson:09}. Therefore, these elimination orders do not
result efficient solution methods for the Helmholtz equation.

\begin{figure}[h!]
  \begin{center}
    \includegraphics[width=1.8in]{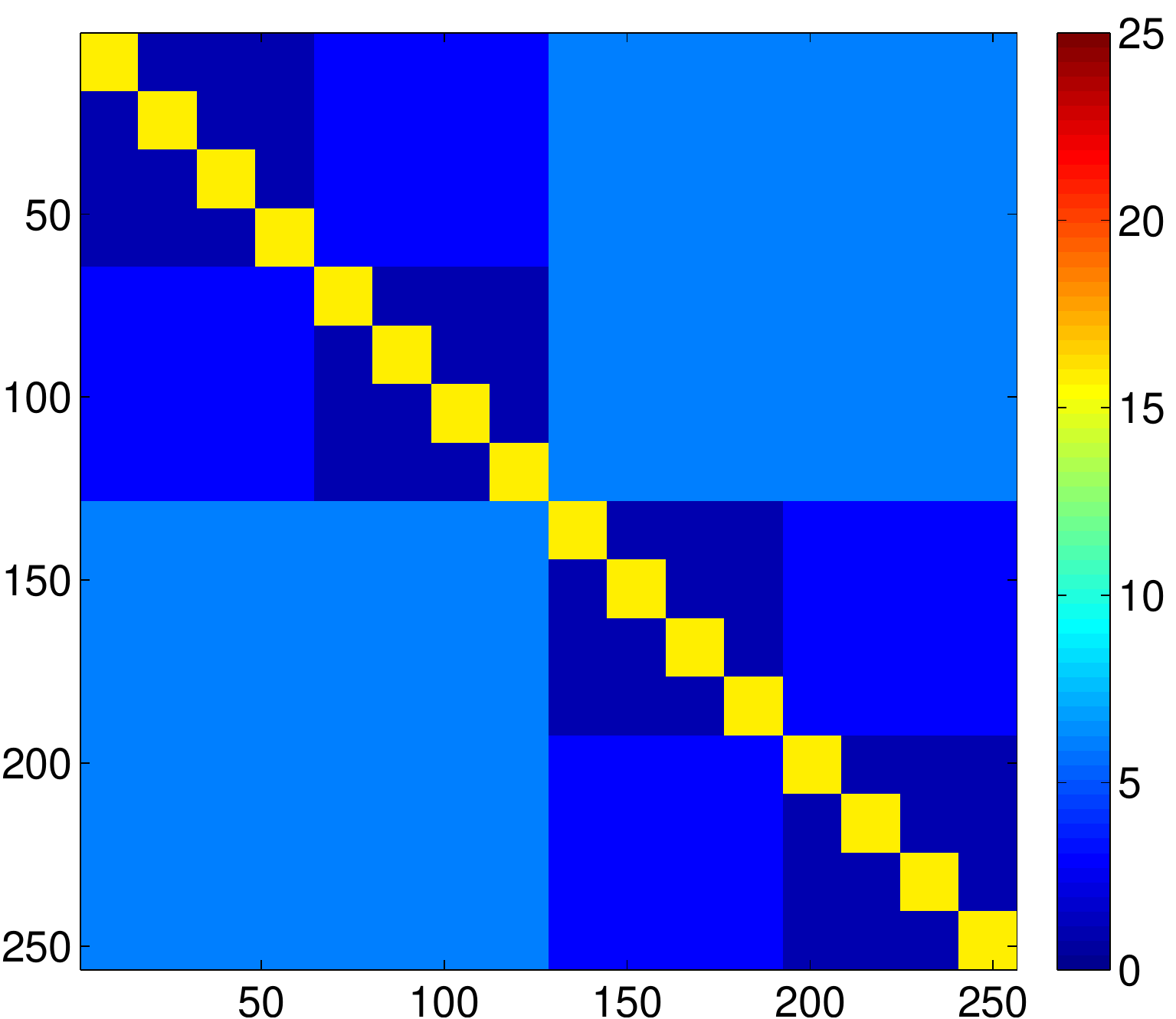}
    \includegraphics[width=1.8in]{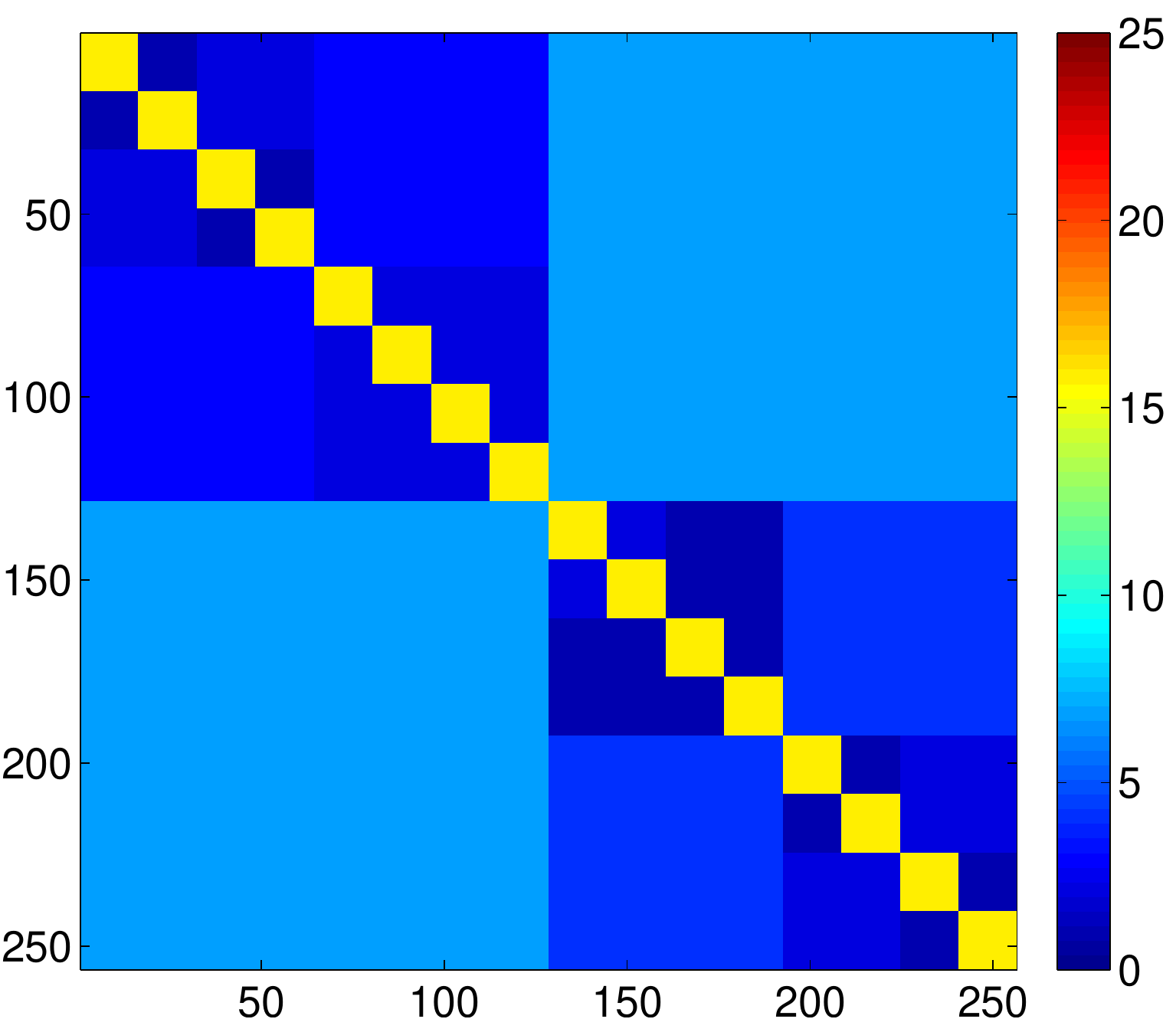}
    \includegraphics[width=1.8in]{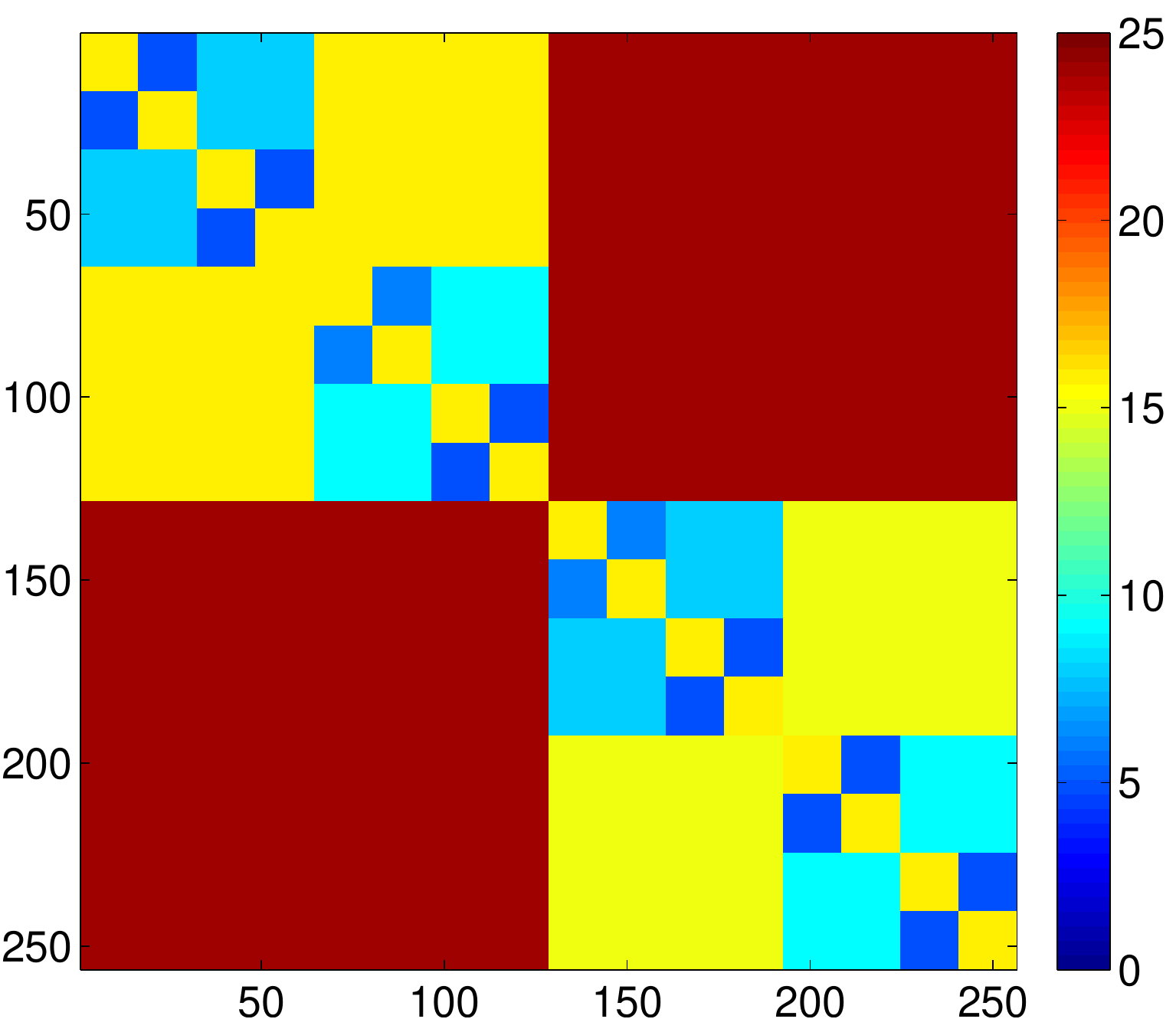}
  \end{center}
  \caption{Numerical ranks of off-diagonal blocks of $T_m$. Left:
    Constant coefficient case with PML boundary condition. Middle:
    Non-constant coefficient case with PML boundary condition.  Right:
    Constant coefficient case with zero Dirichlet boundary condition.
  }
  \label{fig:ranks}
\end{figure}

\subsection{Hierarchical matrix representation}

Since $T_m$ and $S_m$ are highly compressible with numerically
low-rank off-diagonal blocks, it is natural to represent these
matrices using the hierarchical matrix (or $\H$-matrix) framework
proposed by Hackbusch et al
\cite{BormGrasedyckHackbusch:06,GrasedyckHackbusch:03,Hackbusch:99},
where off-diagonal blocks are represented in low-rank factorized form.
The discussion below is by no means original and is included for the
sake of completeness.

At the $m$-th layer for any fixed $m$, we construct a hierarchical
decomposition of the grid points in $\P_m$ through bisection. At level
0 (the top level), the set
\[
J^0_1 = \P_m.
\] 
At level $\ell$, there are $2^\ell$ sets $\J^\ell_i$ for
$i=1,\ldots,2^\ell$ given by
\[
\J^\ell_i = \{ p_{t,m}: (i-1)\cdot n/2^\ell + 1 \le t \le i\cdot n/2^\ell \}.
\]
The bisection is stopped when each set $\J^\ell_i$ contains only a
small number of indices. Hence, the number of total levels $L$ is
equal to $\log_2 n - O(1)$ (see Figure \ref{fig:2Dpart} (left)). We
often write $G(\J^\ell_i, \J^\ell_{i'})$ (the restriction of a matrix
$G$ to $\J^\ell_i$ and $\J^\ell_{i'}$) as $G^\ell_{i,i'}$.

\begin{figure}[h!]
  \begin{center}
    \includegraphics{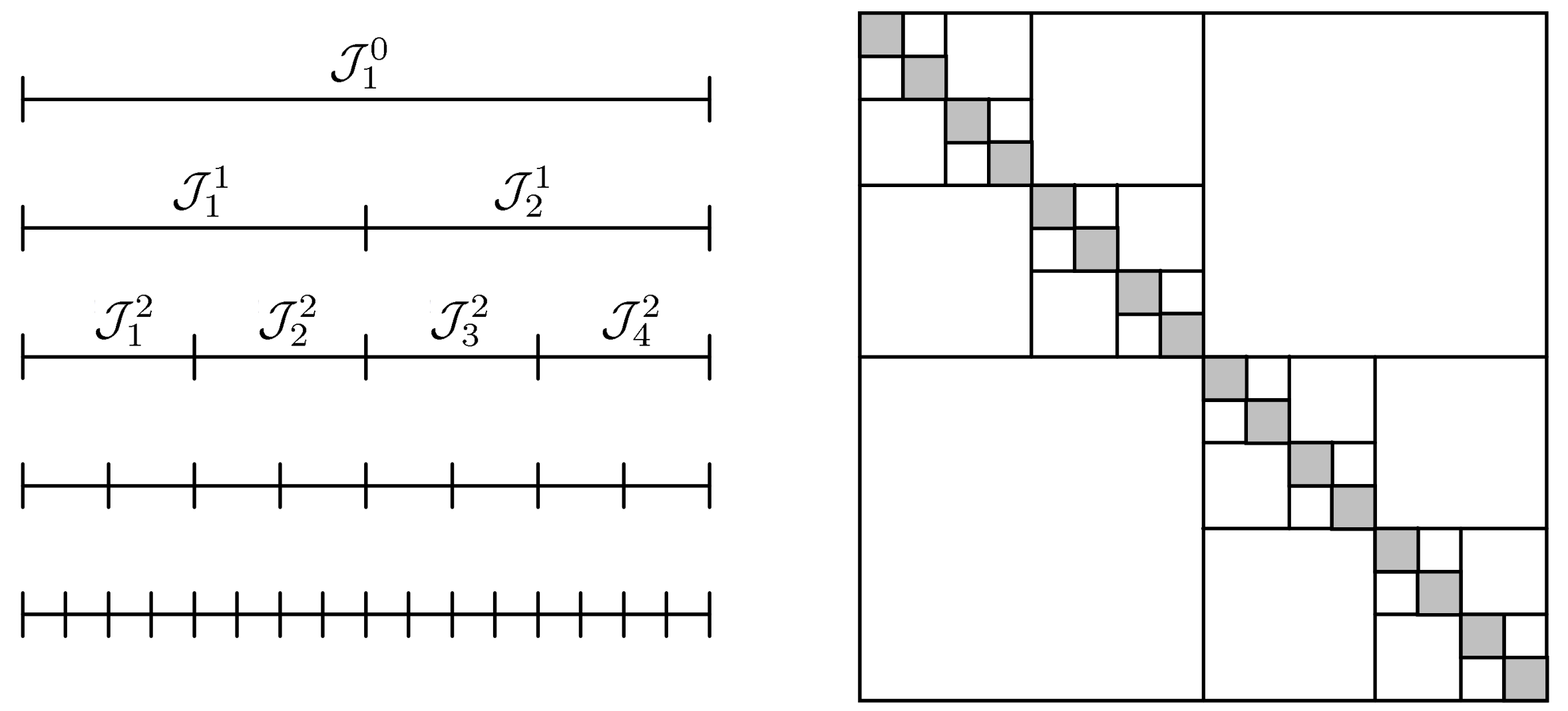}
  \end{center}
  \caption{Hierarchical matrix representation. Left: Hierarchical
    partitioning of the index set $\J$ for each layer. Right: Induced
    partitioning of the matrix $T_m$ in the weakly admissible
    case. Off-diagonal blocks (in white) are stored in low-rank
    factorized form. Diagonal blocks (in gray) are stored densely.  }
  \label{fig:2Dpart}
\end{figure}

The hierarchical matrix representation relies on the notion of {\em
  well-separatedness} between different sets. If $\J^\ell_i$ and
$\J^\ell_{i'}$ are well-separated from each other, then $G(\J^\ell_i,
\J^\ell_{i'})$ is allowed to be stored in a low-rank factorized
form. There are two different choices of the notion of
well-separatedness \cite{BormGrasedyckHackbusch:06}. In the {\em
  weakly admissible case}, $\J^\ell_i$ and $\J^\ell_{i'}$ are
well-separated if and only if they are disjoint. In the {\em strongly
  admissible case}, $\J^\ell_i$ and $\J^\ell_{i'}$ are well-separated
if and only if the distance between them is greater than or equal to
their width. Next, define the interaction list of $\J^\ell_{i}$ to be
the set of all index sets $\J^\ell_{i'}$ such that $\J^\ell_{i}$ is
well-separated from $\J^\ell_{i'}$ but $\J^\ell_{i}$'s parent is not
well-separated from $\J^\ell_{i'}$'s parent. It is clear from this
definition that being a member of another set's interaction list is a
symmetric relationship.

\subsubsection{Weakly admissible case}
In the weakly admissible case, the interaction list of $\J^\ell_{2i}$
contains only $\J^\ell_{2i-1}$ and vice versa.

\paragraph{Matrix representation.}  For a fixed $\eps$, let $R =
O(\log \omega) = O(\log n)$ be the maximum over the ranks of the
off-diagonal blocks on all levels. For a given matrix $G$, the
hierarchical matrix framework represents all blocks $G^\ell_{i,i'} =
G(\J^\ell_i, \J^\ell_{i'})$ with $\J^\ell_i$ and $\J^\ell_{i'}$ in
each other's interaction list in the factorized form with rank less
than or equal to $R$.  For example, at the first level, the two
off-diagonal blocks $G^1_{1,2} = G(\J^1_1,\J^1_2)$ and $G^1_{2,1} =
G(\J^1_2,\J^1_1)$ are represented with
\[
G^1_{1,2} \approx U^1_{1,2} (V^1_{1,2})^t  \quad\text{and}\quad
G^1_{2,1} \approx U^1_{2,1} (V^1_{2,1})^t,
\]
where each of $U^1_{1,2}$, $U^1_{2,1}$, $V^1_{1,2}$, $V^1_{2,1}$ has
at most $R$ columns. At the second level, the new off-diagonal blocks
are $G^2_{1,2}$, $G^2_{2,1}$, $G^2_{3,4}$, and $G^2_{4,3}$, each
represented in a similar way. Finally, at level $L-1$, all diagonal
blocks $G^{L-1}_{i,i}$ for $i=1,\ldots, 2^{L-1}$ are stored
densely. This representation is illustrated in Figure \ref{fig:2Dpart}
(right). The total storage cost is $O(R n\log n)$.

\paragraph{Matrix-vector multiplication.} 
Let us consider the product $G f$ where $f$ is a vector of size
$n$. Denote by $f^\ell_i$ the part of $f$ restricted to
$I^\ell_i$. Using the block matrix form, the product is
\[
\begin{pmatrix}
  G^1_{1,1} & G^1_{1,2} \\
  G^1_{2,1} & G^1_{2,2}
\end{pmatrix}
\begin{pmatrix}
  f^1_1 \\
  f^1_2
\end{pmatrix}
= 
\begin{pmatrix}
  G^1_{1,1} f^1_1 + G^1_{1,2} f^1_2 \\
  G^1_{2,1} f^1_1 + G^1_{2,2} f^1_2
\end{pmatrix}.
\]
First, the product $G^1_{1,2} f^1_2$ is computed with $G^1_{1,2} f^1_2
\approx U^1_{1,2} \left( (V^1_{1,2})^t f^1_2 \right)$. The same is
carried out for the product $G^1_{2,1} f^1_1$.  Second, the
computation of $G^1_{1,1} f^1_1$ and $G^1_{2,2} f^1_2$ is done
recursively since both $G^1_{1,1}$ and $G^1_{2,2}$ are in the
hierarchical matrix form. We denote this matrix-vector multiplication
procedure by $\ms{hmatvec}(G,f)$ and its computational cost is $O(R
n\log n)$.

\paragraph{Matrix addition and subtraction.}

Consider the sum of two matrices $G$ and $H$ with their off-diagonal
blocks represented in the factorized form by $G^\ell_{i,j} \approx
U^\ell_{i,j} (V^\ell_{i,j})^t$ and $H^\ell_{i,j} \approx X^\ell_{i,j}
(Y^\ell_{i,j})^t$. Under the block matrix notation, the sum is
\[
\begin{pmatrix}
G^1_{1,1} & G^1_{1,2}\\
G^1_{2,1} & G^1_{2,2}
\end{pmatrix}
+ 
\begin{pmatrix}
H^1_{1,1} & H^1_{1,2}\\
H^1_{2,1} & H^1_{2,2}
\end{pmatrix}
=
\begin{pmatrix}
G^1_{1,1}+H^1_{1,1} & G^1_{1,2}+H^1_{1,2}\\
G^1_{2,1}+H^1_{2,1} & G^1_{2,2}+H^1_{2,2}
\end{pmatrix}.
\]
First, $G^1_{1,2} + H^1_{1,2} \approx U^1_{1,1} (V^1_{1,2})^t +
X^1_{1,2} (Y^1_{1,2})^t = \left( U^1_{1,2}, X^1_{1,2}\right) \left(
  V^1_{1,2}, Y^1_{1,2}\right)^t$. One needs to recompress the last two
matrices in order to prevent the rank of the low rank factorization
from increasing indefinitely. This can be done by computing QR
decomposition of $(U^1_{1,2}, X^1_{1,2})$ and $(V^1_{1,2},
Y^1_{1,2})$, followed by a truncated SVD of a matrix of small
size. The same procedure is carried out for $G^1_{2,1} + H^1_{2,1}$ to
compute the necessary factorization. Second, let us consider the
diagonal blocks. $G^1_{1,1} + H^1_{1,1}$ and $G^1_{2,2} + H^1_{2,2}$
are done recursively since they are two sums of the same nature but
only half the size. This addition procedure is denoted by
$\ms{hadd}(G,H)$. The subtraction procedure is almost the same and is
denoted by $\ms{hsub}(G,H)$. Both of them take $O(R^2 n\log n)$ steps.

\paragraph{Matrix multiplication.}

Let us consider the sum of two matrices $G$ and $H$ with their
off-diagonal blocks represented by $G^\ell_{i,j} \approx U^\ell_{i,j}
(V^\ell_{i,j})^t$ and $H^\ell_{i,j} \approx X^\ell_{i,j}
(Y^\ell_{i,j})^t$.  Under the block matrix form, the product is
\[
\begin{pmatrix}
G^1_{1,1} & G^1_{1,2}\\
G^1_{2,1} & G^1_{2,2}
\end{pmatrix}
\cdot  
\begin{pmatrix}
H^1_{1,1} & H^1_{1,2}\\
H^1_{2,1} & H^1_{2,2}
\end{pmatrix}
=
\begin{pmatrix}
  G^1_{1,1} H^1_{1,1} + G^1_{1,2} H^1_{2,1} & G^1_{1,1} H^1_{1,2} + G^1_{1,2} H^1_{2,2} \\
  G^1_{2,1} H^1_{1,1} + G^1_{2,2} H^1_{2,1} & G^1_{2,1} H^1_{1,2} + G^1_{2,2} H^1_{2,2}
\end{pmatrix}.
\]
First, the off-diagonal block $G^1_{1,1} H^1_{1,2} + G^1_{1,2}
H^1_{2,2} \approx G^1_{1,1} X^1_{1,2} (Y^1_{1,2})^t + U^1_{1,2}
(V^1_{1,2})^t H^1_{2,2}$.  The computation $G^1_{1,1} X^1_{1,2}$ and
$(V^1_{1,2})^t H^1_{2,2}$ are essentially matrix-vector
multiplications. Once they are done, the remaining computation is then
similar to the off-diagonal part of the matrix addition algorithm. The
other off-diagonal block $G^1_{2,1} H^1_{1,1} + G^1_{2,2} H^1_{2,1}$
is done in the same way.  Next, consider the diagonal blocks. Take
$G^1_{1,1} H^1_{1,1} + G^1_{1,2} H^1_{2,1}$ as an example. The first
part $G^1_{1,1} H^1_{1,1}$ is done using recursion. The second part is
$ G^1_{1,2} H^1_{2,1} \approx U^1_{1,2} (V^1_{1,2})^t X^1_{2,1}
(Y^1_{2,1})^t$, where the middle product is carried out first in order
to minimize the computational cost.  The final sum $G^1_{1,1}
H^1_{1,1} + G^1_{1,2} H^1_{2,1}$ is done using the matrix addition
algorithm described above. The same procedure can be carried out for
$G^1_{2,1} H^1_{1,2} + G^1_{2,2} H^1_{2,2}$. This matrix
multiplication procedure is denoted by $\ms{hmul}(G,H)$ and its
computational cost is $O(R^2 n\log^2 n)$.

\paragraph{Matrix inversion.} The inverse of $G$ is done by performing
a $2\times 2$ block matrix inversion:
\[
\begin{pmatrix}
G^1_{1,1} & G^1_{1,2}\\
G^1_{2,1} & G^1_{2,2}
\end{pmatrix}^{-1}
=
\begin{pmatrix}
(G^1_{1,1})^{-1} + (G^1_{1,1})^{-1}G^1_{1,2}S^{-1}G^1_{2,1}(G^1_{1,1})^{-1} &  -(G^1_{1,1})^{-1} G^1_{1,2} S^{-1}\\
-S^{-1} G^1_{2,1} (G^1_{1,1})^{-1} & S^{-1}
\end{pmatrix}
\]
where $S = G^1_{2,2} - G^1_{2,1} (G^1_{1,1})^{-1} G^1_{1,2}$. The
computation of this formula requires matrix additions and
multiplications, along with the inversion of two matrices $S$ and
$G^1_{1,1}$, half of the original size. The matrix additions and
multiplications are carried out by the above procedures, while the
inversion are done recursively. This matrix inversion procedure is
denoted by $\ms{hinv}(G)$ and its cost is $O(R^2 n\log^2 n)$.

\paragraph{Multiplication with a diagonal matrix.} Finally, we
consider the multiplication of $G$ with a diagonal matrix $D$.  Denote
the two diagonal blocks of $D$ on the first level by $D^1_{1,1}$ and
$D^1_{2,2}$, both of which are diagonal matrices. In the block matrix
form, the product becomes
\[
\begin{pmatrix}
G^1_{1,1} & G^1_{1,2}\\
G^1_{2,1} & G^1_{2,2}
\end{pmatrix}
\cdot  
\begin{pmatrix}
D^1_{1,1} & \\
 & D^1_{2,2}
\end{pmatrix}
=
\begin{pmatrix}
  G^1_{1,1} D^1_{1,1} & G^1_{1,2} D^1_{2,2} \\
  G^1_{2,1} D^1_{1,1} & G^1_{2,2} D^1_{2,2}
\end{pmatrix}.
\]
Consider the off-diagonal blocks first. For example, $G^1_{1,2}
D^1_{2,2} \approx U^1_{1,2} (V^1_{1,2})^t D^1_{2,2}$ and this is done
by scaling each columns of $(V^1_{1,2})^t$ by the corresponding
diagonal entries of $D^1_{2,2}$. The same is true for $G^1_{2,1}
D^1_{1,1}$. For the diagonal blocks, say $G^1_{1,1} D^1_{1,1}$, we
simply apply recursion since $G^1_{1,1}$ is itself a hierarchical
matrix and $D^1_{1,1}$ is diagonal.  This special multiplication
procedure is denoted by $\ms{hdiagmul}(G,D)$ if $D$ is on the right or
$\ms{hdiagmul}(D,G)$ if $D$ is on the left. The cost of both
procedures is $O(R n\log n)$.

\subsubsection{Strongly admissible case}

The matrix representation and operations in the strongly admissible
case are similar to the ones in the weakly admissible case. The only
one that requires significant modification is the matrix
multiplication procedure $R = \ms{hmul}(G,H)$, where the most common
step is the calculation of
\begin{equation}
R^\ell_{i,i''} \leftarrow G^\ell_{i,i'} H^\ell_{i',i''}.
\label{eq:RGH}
\end{equation}
In order to simplify the discussion, we denote a matrix symbolically
by $\ms{H}$ if it is in hierarchical form and by $\ms{F}$ if it is
represented in a factorized form. The product \eqref{eq:RGH} can then
take one of the following eight forms
\[
\begin{matrix}
  \ms{H} = \ms{H} \cdot \ms{H}, &
  \ms{H} = \ms{H} \cdot \ms{F}, &
  \ms{H} = \ms{F} \cdot \ms{H}, &
  \ms{H} = \ms{F} \cdot \ms{F}, \\
  \ms{F} = \ms{H} \cdot \ms{H}, &
  \ms{F} = \ms{H} \cdot \ms{F}, &
  \ms{F} = \ms{F} \cdot \ms{H}, &
  \ms{F} = \ms{F} \cdot \ms{F}. \\
\end{matrix}
\]
All of them except one have already appeared in the matrix
multiplication procedure of the {\em weakly admissible case} and the
only one that is new is $\ms{F} = \ms{H} \cdot \ms{H}$. We implement
this using the randomized SVD algorithm proposed recently in
\cite{HalkoMartinssonTropp:09,LibertyWoolfeMartinssonRokhlinTygert07}
for numerically low-rank matrices. The main idea of this randomized
algorithm is to capture the column (or row) space of the matrix by
multiplying the matrix with a small number of Gaussian random test
vectors. Results from random matrix theory guarantee that the column
space of the product matrix approximates accurately the span of all
dominant singular vectors of the original (numerically low-rank)
matrix. Since the product matrix has much fewer columns, applying
singular value decompositions to it gives rise to an accurate and
efficient way to approximate the SVD of the original matrix. Notice
that this randomized approach only requires a routine to apply the
original matrix to an arbitrary vector and everything else is just
standard numerical linear algebra. In our setting, applying $\ms{H}
\cdot \ms{H}$ to a vector is simply equal to two $\ms{hmatvec}$
operations.

\subsection{Approximate inversion and preconditioner}

Let us denote the approximations of $S_m$ and $T_m$ in the
hierarchical matrix representation by $\wt{S}_m$ and $\wt{T}_m$,
respectively. The construction of the approximate $LDL^t$
factorization of $H$ takes the following steps.
\begin{algo}
  Construction of the approximate sweeping factorization of $H$ in the
  hierarchical matrix framework.
  \label{alg:setup}
\end{algo}
\begin{algorithmic}[1]
  \STATE $\wt{S}_1 = A_{1,1}$ and $\wt{T}_1 = \ms{hinv}(\wt{S}_1)$.
  \FOR{$m=2,\ldots,n$}
  \STATE $\wt{S}_m = \ms{hsub}(A_{m,m}, \ms{hdiagmul}(A_{m,m-1},\ms{hdiagmul}(\wt{T}_{m-1},A_{m-1,m}))$
  and $\wt{T}_m = \ms{hinv}(\wt{S}_m)$.
  \ENDFOR
\end{algorithmic}
The cost of Algorithm \ref{alg:setup} is $O(R^2 n^2 \log^2 n) = O(R^2
N \log^2 N)$. The computation of $u \approx A^{-1}f$ using the this
approximate factorization is summarized as follows.
\begin{algo}
  Computation of $u \approx A^{-1}f$ using the approximate
  sweeping factorization of $A$ in the hierarchical matrix framework.
  \label{alg:solve}
\end{algo}
\begin{algorithmic}[1]
  \FOR{$m=1,\ldots,n$}
  \STATE $u_m = f_m$
  \ENDFOR
  \FOR{$m=1,\ldots,n-1$}
  \STATE $u_{m+1} = u_{m+1} - A_{m+1,m} \cdot \ms{hmatvec}(\wt{T}_m,u_m)$
    \ENDFOR
  \FOR{$m=1,\ldots,n$}
  \STATE $u_m = \ms{hmatvec}(\wt{T}_m,u_m)$
  \ENDFOR
  \FOR{$m=n-1,\ldots,1$}
  \STATE $u_m = u_m - \ms{hmatvec}(\wt{T}_m, A_{m,m+1} u_{m+1})$
  \ENDFOR
\end{algorithmic}
The cost of Algorithm \ref{alg:solve} is $O(R n^2 \log n) = O(R N \log
N)$. Algorithm \ref{alg:solve} defines an operator
\[
M: 
f = (f_1^t, f_2^t, \ldots, f_n^t)^t \rightarrow
u = (u_1^t, u_2^t, \ldots, u_n^t)^t,
\]
which is an approximate inverse of the discrete Helmholtz operator
$A$. When the threshold $\eps$ is set to be sufficiently small, $M$
can be used directly as the inverse of $H$ and $u$ can be taken as the
solution. However, a small $\eps$ value means that the rank $R$ of the
low-rank factorized form needs to be fairly large, thus resulting
large storage and computation cost. On the other hand, when $R$ is
kept rather small, Algorithms \ref{alg:setup} and \ref{alg:solve}
become highly efficiently both in terms of storage and time. Though
the resulting $M$ is not accurate enough as the inverse of $A$, it
serves as an excellent preconditioner. Therefore, we solve the
preconditioner system
\[
MA u = Mf
\]
using iterative solvers such as GMRES and TFQMR
\cite{Saad:03,SaadSchultz:86}. Since the cost of applying $M$ to any
vector is $O(R N \log N)$, the total cost of the iterative solver is
$O(N_I R N \log N)$, where $N_I$ is the number of iterations. The
numerical results in Section \ref{sec:2Dnum} demonstrate that $N_I$ is
in practice very small, thus resulting an algorithm for almost linear
complexity.

Theorem \ref{thm:lowrank} shows that in the constant coefficient case
the hierarchical matrix representation of $T_m$ is
accurate. Therefore, the preconditioner $M$ well approximates the
inverse of $A$ and the number of iterations $N_I$ is expected to be
small. The numerical results in Section \ref{sec:2Dnum} demonstrates
that $N_I$ is also small for general velocity field such as converging
lens, wave guides, and random media.  Here we provide a heuristic
explanation for this phenomena. For the variable coefficient case, the
numerical rank of the off-diagonal blocks of $T_m$ can potentially
increase mainly due to the {\em turning rays}, i.e., the rays that
leave the $m$-th layer downward, travel horizontally in $x_1$
direction, and come upward back to the $m$-th layer. The interactions
related to turning rays are difficult to capture in the hierarchical
matrix representation of $T_m$ if $R$ is small. However, the iterative
solver addresses this interaction in several steps as follows: the
downward part of the ray is processed by a first few sweeps, the
horizontal part is then captured by the $T_m$ matrix of the next
sweep, and finally the upward part of the ray is processed by a couple
of extra sweeps.

In the presentation of the sweeping preconditioner, we choose the
sweeping direction to be in the positive direction of the $x_2$
axis. It is clear that sweeping along either one of the other three
directions also gives a slightly different sweeping preconditioner. Due
to the variations in the velocity field and, more precisely, the
existence of the turning rays, a carefully selected sweeping direction
can often result significantly fewer number of GMRES iterations than
the other directions do. We will give one numerical example to
demonstrate this phenomenon in Section \ref{sec:2Dnum}.

\subsection{Other boundary conditions}
\label{sec:2Dpreother}

So far, we discuss the case with Sommerfeld boundary condition
specified over the whole boundary. From the above discussion, it is
clear that the success of the preconditioner only relies on the fact
that $S_m$ and $T_m$ are compressible. For many other boundary
conditions, the matrices $S_m$ and $T_m$ also have this property, as
long as the Helmholtz problem is not close to resonance. Here, we
mention three representative examples.

\begin{figure}[h!]
  \begin{center}
    \includegraphics{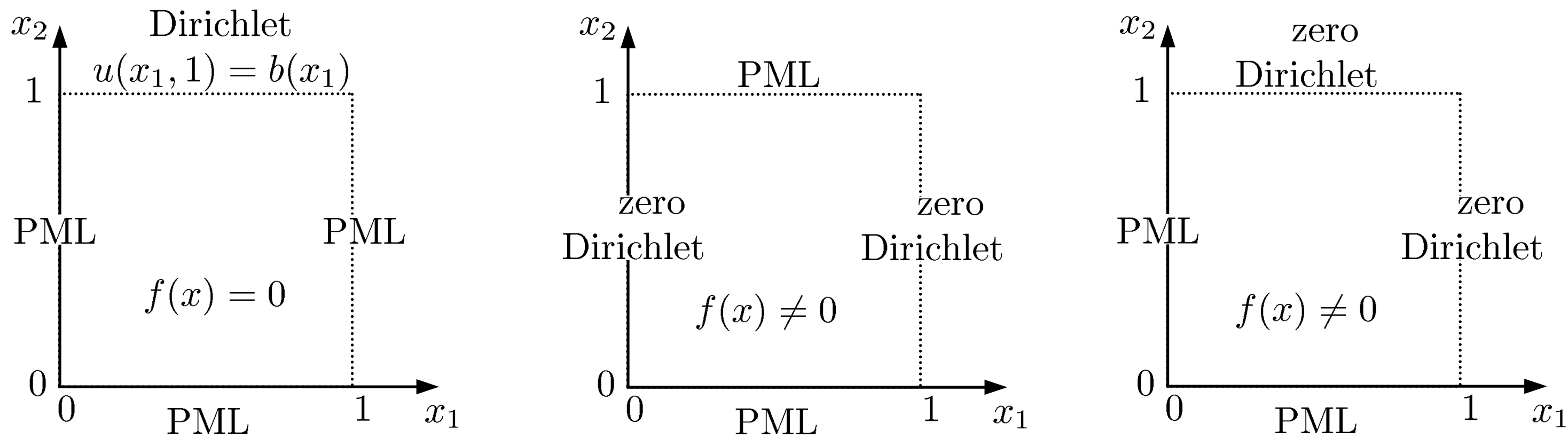}
  \end{center}
  \caption{Mixed boundary conditions. Left: Depth extrapolation
    problem in seismology. Middle and right: Problems with partly zero
    Dirichlet boundary condition and non-zero $f(x)$. }
  \label{fig:other}
\end{figure}

In the first example (see Figure \ref{fig:other} (left)), the PML
boundary condition at $x_2=1$ is replaced with non-trivial Dirichlet
boundary condition $u(x_1,1) = b(x_1)$ and $f$ is equal to zero. This
corresponds to the depth extrapolation problem
\cite{Biondi:06,Margrave:99} in reflection seismology. The proposed
algorithm proceeds exactly the same and the only modification is that
the boundary condition $b(x_1)$ is transformed into an appropriate
forcing term at last layer of unknowns (i.e., the index set $\P_n$).

In the second example, the zero boundary condition is mixed with the
PML condition. In Figure \ref{fig:other} (middle)) the zero Dirichlet
boundary condition is specified on $x_1=0$ and $x_1=1$. The matrix
$T_m$ then corresponds the restriction (to an edge) of the Green's
function of the discrete Helmholtz operator in a half strip. By using
the imaging method also in the $x_1$ direction, one can show that the
rank of the off-diagonal blocks is bounded by $O(\log\omega
|\log\eps|^2)$ with a slightly larger constant due to the mirror
images. In Figure \ref{fig:other} (right)), the zero Dirichlet
boundary condition is specified on $x_1=1$ and $x_2=1$. $T_m$
corresponds to the restriction of the Green's function of the discrete
Helmholtz operator in a quadrant in this case.

Finally, the PML boundary condition is by no means the only
approximation to the Sommerfeld condition. As the essential
requirement is that the problem should not be close to resonance
(i.e., a wave packet escapes the domain without spending too much time
inside), the sweeping preconditioner should work with any reasonable
approximations to the Sommerfeld boundary condition such as absorbing
boundary conditions (ABCs) \cite{EngquistMajda:77,EngquistMajda:79}
and damping/sponge layers. We focus on the PML due to its simplicity,
its low non-physical reflections, and the symmetry of its discrete
system.

\section{Numerical Results in 2D}
\label{sec:2Dnum}

In this section, we present several numerical results to illustrate
the properties of the sweeping preconditioner described in Section
\ref{sec:2Dpre}. The implementation is done in C++ and the results in
this section are obtained on a computer with a 2.6GHz CPU. The GMRES
method is used as the iterative solver with relative residue tolerance
set to be $10^{-3}$.



\subsection{PML}
\label{sec:2DnumPML}

The examples in this section have the PML boundary condition specified
at all sides.

\paragraph{Dependence on $\omega$.}
First, we study how the sweeping preconditioner behaves when $\omega$
varies. Consider three velocity fields in the domain $(0,1)^2$:
\begin{enumerate}
\item The first velocity field is a converging lens with a Gaussian profile
  at the center of the domain (see Figure \ref{fig:2Dnumspeed}(a)).
\item The second velocity field is a vertical waveguide with Gaussian cross
  section (see Figure \ref{fig:2Dnumspeed}(b)).
\item The third velocity field has a random velocity field (see Figure
  \ref{fig:2Dnumspeed}(c)).
\end{enumerate}
For each velocity field, we test with two external forces $f(x)$.
\begin{enumerate}
\item The first external force $f(x)$ is a Gaussian point source
  located at $(x_1,x_2) = (0.5, 0.125)$. The response of this forcing
  term generates circular waves propagating at all directions. Due to
  the variations of the velocity field, the circular waves would bend,
  form caustics, and intersect.
\item The second external force $f(x)$ is a Gaussian wave packet with
  a wavelength comparable to the typical wavelength of the Helmholtz
  equation. This packet centers at $(x_1,x_2) = (0.125, 0.125)$ and
  points to the $(1,1)$ direction. The response of this forcing term
  generates a Gaussian beam initially pointing towards the $(1,1)$
  direction. Due to the variations of the velocity field, this
  Gaussian beam should bend and scatter.
\end{enumerate}

\begin{figure}[h!]
  \begin{center}
    \begin{tabular}{ccc}
      \includegraphics[height=1.6in]{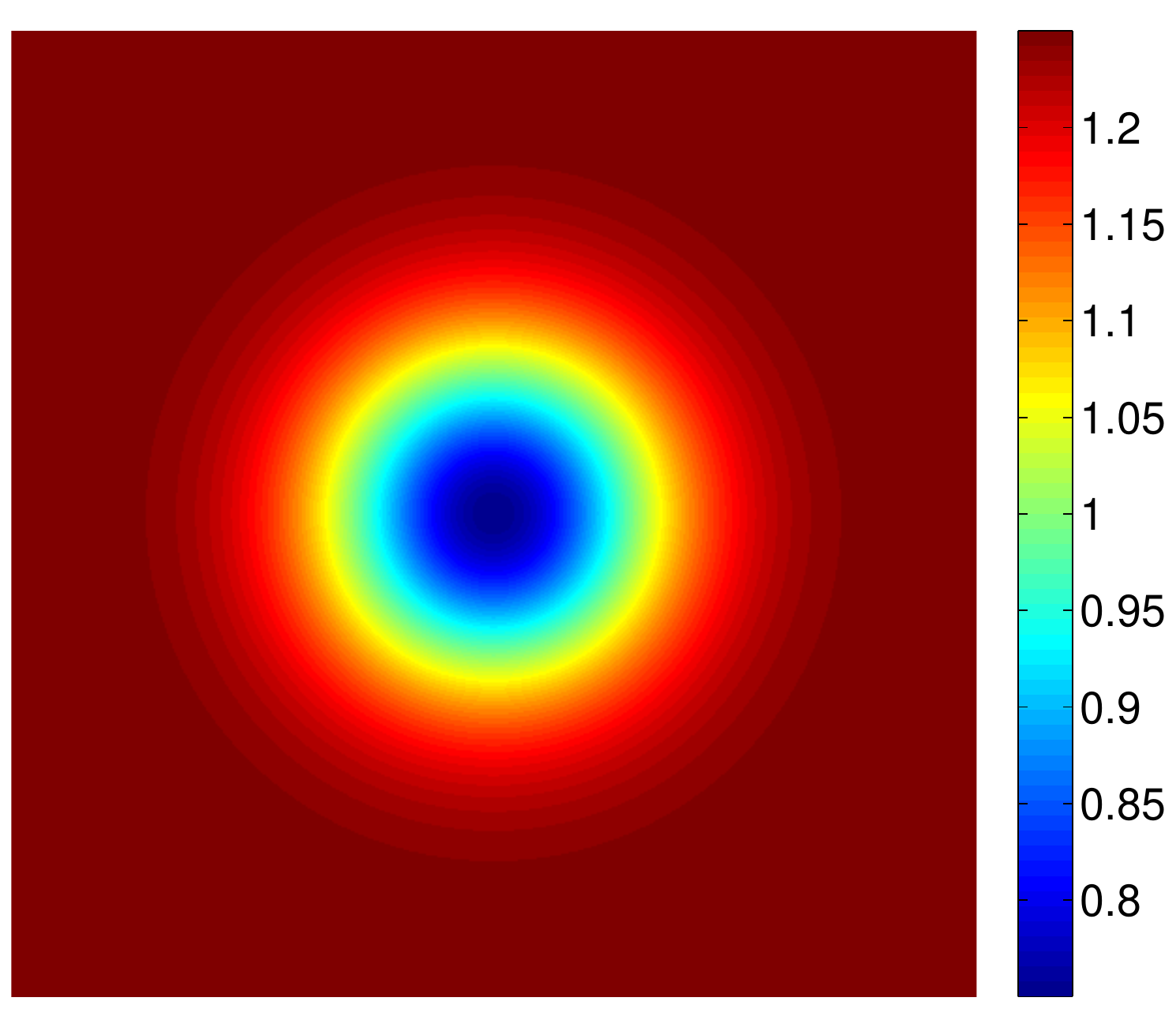}&\includegraphics[height=1.6in]{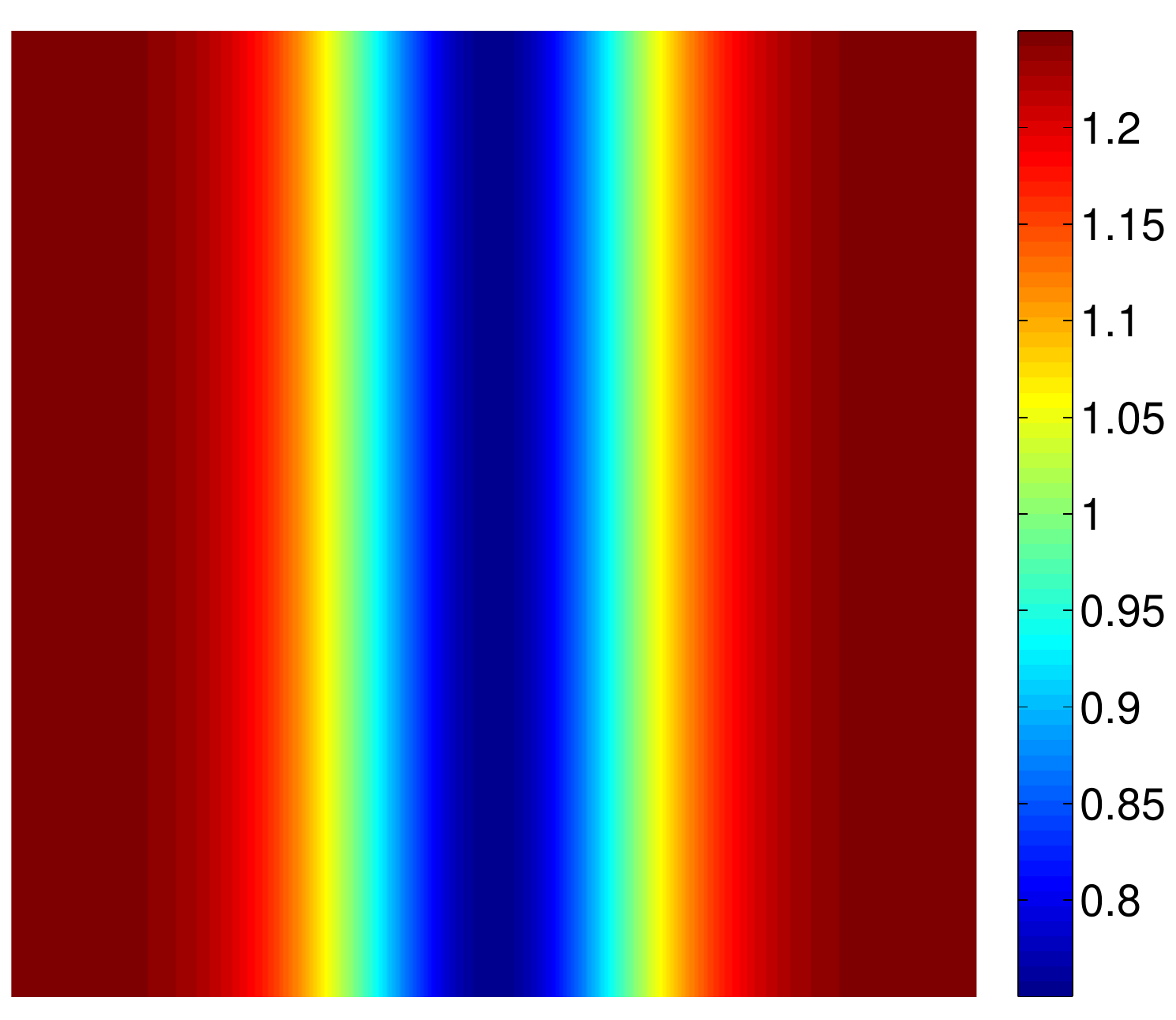}&\includegraphics[height=1.6in]{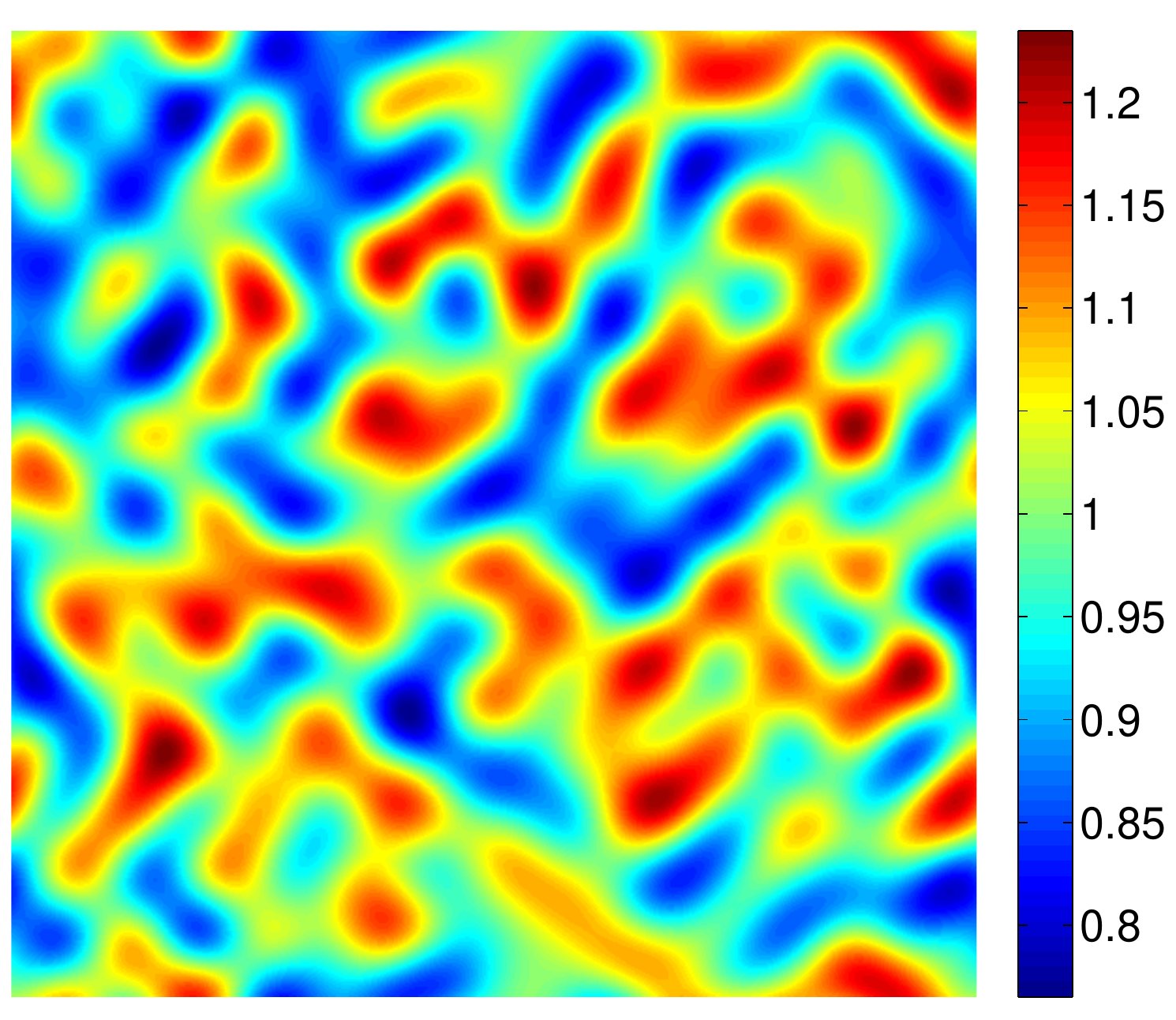}\\
      (a) & (b) & (c)
    \end{tabular}
  \end{center}
  \caption{Test velocity fields.}
  \label{fig:2Dnumspeed}
\end{figure}

For each velocity field, we perform tests for $\frac{\omega}{2\pi} =
16,32,\ldots,256$. In these tests, we discretize with $q=8$ points per
wavelength. Therefore, the number of points for each dimension is $n=8
\times \frac{\omega}{2\pi} = 128, 256, \ldots, 2048$. The strongly
admissible case is used in the implementation of the hierarchical
matrix representation. Recall that $R$ is the rank of the off-diagonal
blocks in the hierarchical matrix and we fix it to be a uniform
constant $2$. In all tests, the sweeping direction is bottom-up from
$x_2=0$ to $x_2=1$.

The results of the first velocity field are summarized in Table
\ref{tbl:2DPML1}. $T_{\text{setup}}$ denotes the time used to
construct the preconditioner in seconds. For each external force,
$N_{\text{iter}}$ is the number of iterations of the preconditioned
GMRES solver and $T_{\text{solve}}$ is the overall solution time. When
$n$ doubles and $N$ quadruples, the setup cost $T_{\text{setup}}$
increases by a factor of $5$ or $6$, which is consistent with the $O(N
\log^2 N)$ complexity of Algorithm \ref{alg:setup}. A remarkable
feature of the sweeping preconditioner is that the number of
iterations is extremely small. In fact, in all cases, the
preconditioned GMRES solver converges in less than 3 iterations. As a
result of the constant iteration number, the solution time increase by
a factor of 4 or 5 when $N$ quadruples, which is consistent with the
$O(N \log N)$ complexity of Algorithm \ref{alg:solve}. Finally, we
would like to point out that our algorithm is extremely efficient: for
a problem with $N = n^2 = 2048^2$ unknowns, the solution time is only
about $30$ seconds.

The results of the second and third velocity fields are summarized in
Tables \ref{tbl:2DPML2} and \ref{tbl:2DPML4}, respectively. The
behaviors of these tests are similar to the one of the first velocity
field. In all cases, the GMRES solver converges in less than 5
iterations when combined with the sweeping preconditioner.

\begin{table}[h!]
  \begin{center}
    \includegraphics[height=2.3in]{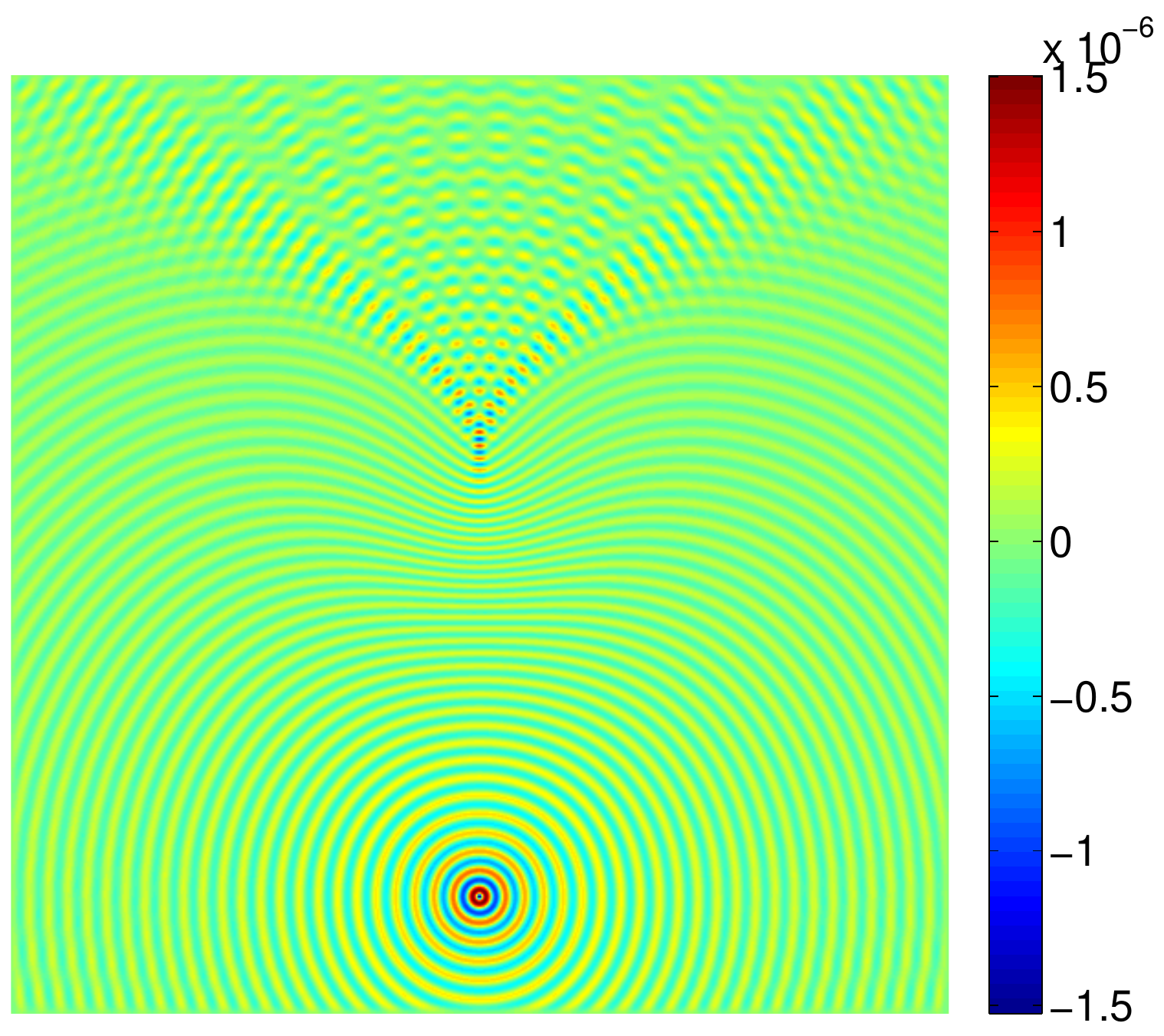}
    \includegraphics[height=2.3in]{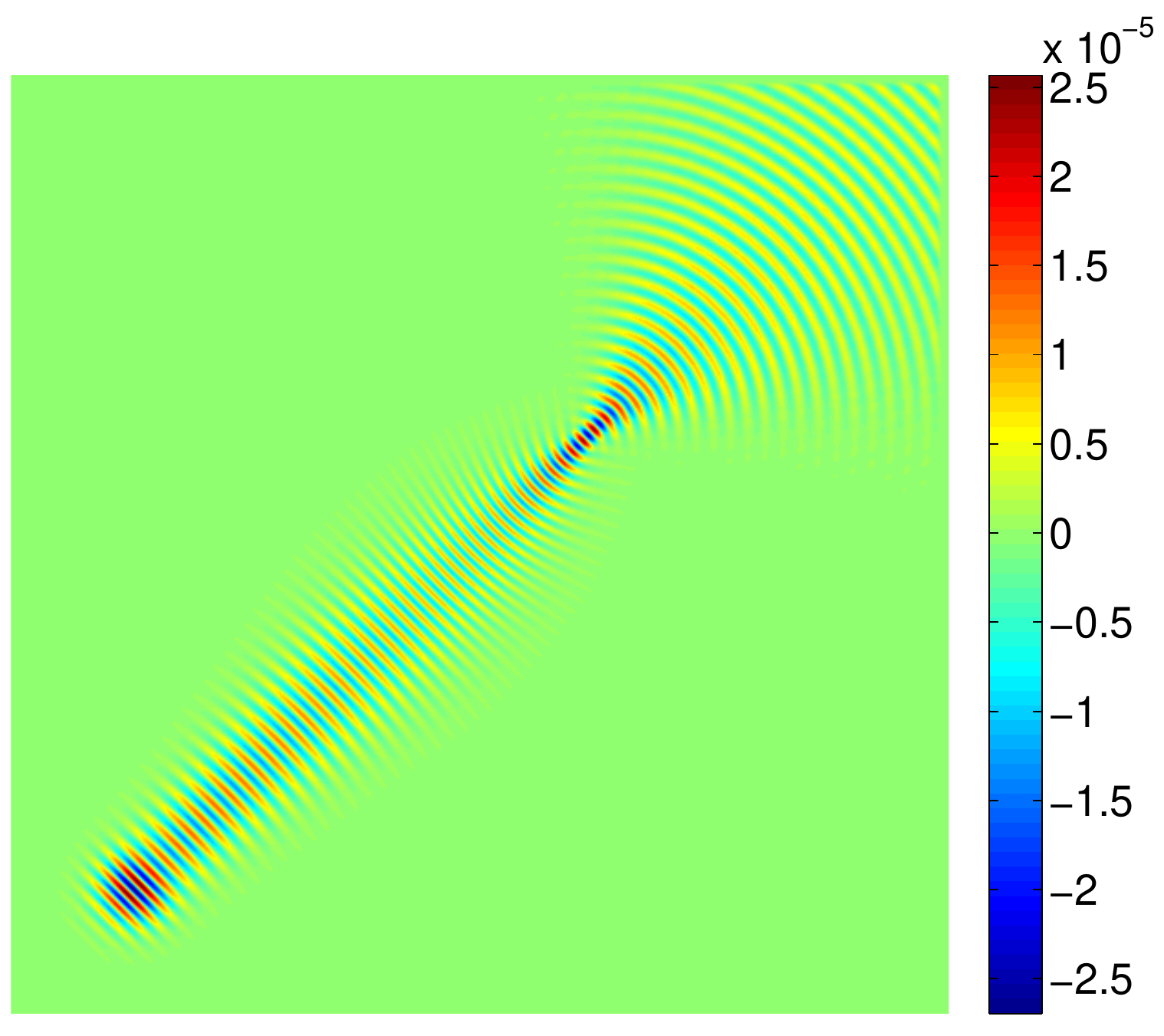}\\
    \begin{tabular}{|ccccc|cc|cc|}
      \hline
      \multicolumn{5}{|c|}{} & \multicolumn{2}{|c|}{Test 1} & \multicolumn{2}{|c|}{Test 2}\\
      \hline
      $\omega/(2\pi)$ & $q$ & $N=n^2$ & $R$ & $T_{\text{setup}}$ & $N_{\text{iter}}$ & $T_{\text{solve}}$ & $N_{\text{iter}}$ & $T_{\text{solve}}$\\
      \hline
      16  & 8 & $128^2$  & 2 & 6.50e-01 & 2 & 5.00e-02 & 2 & 5.00e-02 \\
      32  & 8 & $256^2$  & 2 & 5.05e+00 & 2 & 2.50e-01 & 2 & 2.50e-01 \\
      64  & 8 & $512^2$  & 2 & 3.44e+01 & 3 & 1.45e+00 & 3 & 1.42e+00 \\
      128 & 8 & $1024^2$ & 2 & 2.16e+02 & 3 & 7.37e+00 & 3 & 7.36e+00 \\
      256 & 8 & $2048^2$ & 2 & 1.24e+03 & 3 & 3.31e+01 & 3 & 3.28e+01 \\
      \hline
    \end{tabular}
  \end{center}
  \caption{Results of velocity field 1 for different $\omega$. Top: Solutions for two external forces with $\omega/(2\pi)=64$.
    Bottom: Results for different $\omega$.  }
  \label{tbl:2DPML1}
\end{table}

\begin{table}[h!]
  \begin{center}
    \includegraphics[height=2.3in]{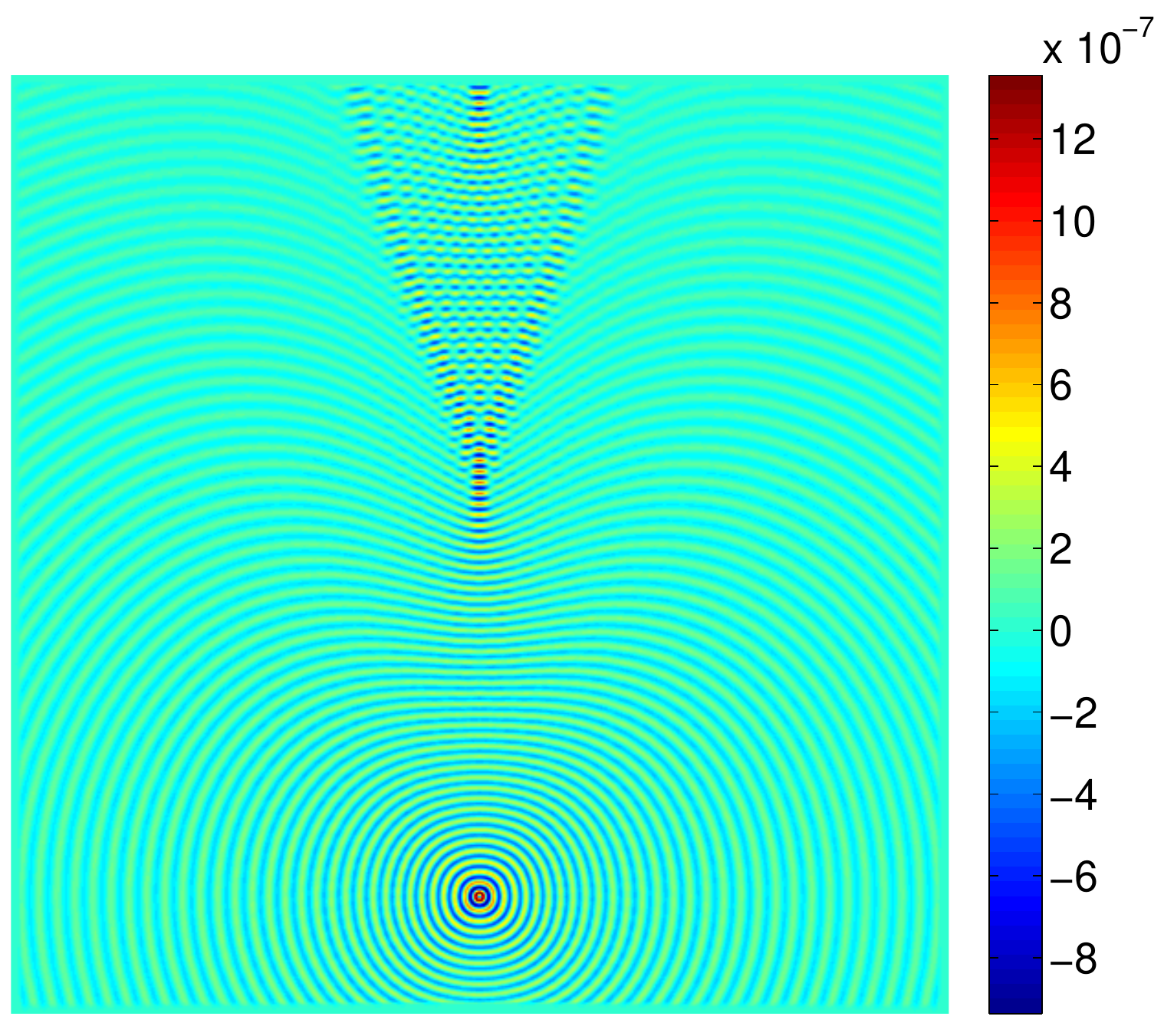}
    \includegraphics[height=2.3in]{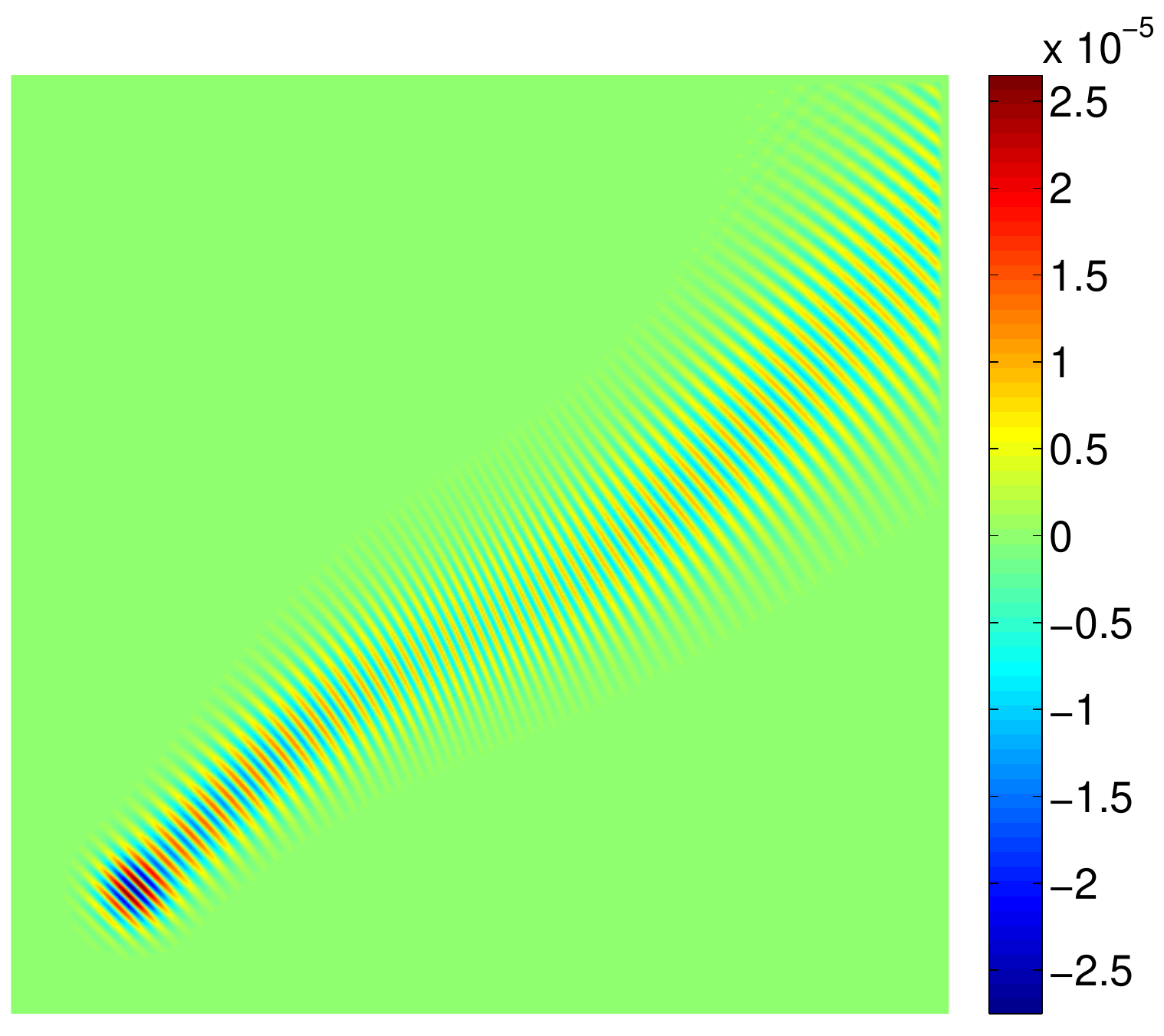}\\
    \begin{tabular}{|ccccc|cc|cc|}
      \hline
      \multicolumn{5}{|c|}{} & \multicolumn{2}{|c|}{Test 1} & \multicolumn{2}{|c|}{Test 2}\\
      \hline
      $\omega/(2\pi)$ & $q$ & $N=n^2$ & $R$ & $T_{\text{setup}}$ & $N_{\text{iter}}$ & $T_{\text{solve}}$ & $N_{\text{iter}}$ & $T_{\text{solve}}$\\
      \hline
      16  & 8 & $128^2$  & 2 & 6.70e-01 & 2 & 5.00e-02 & 2 & 6.00e-02 \\
      32  & 8 & $256^2$  & 2 & 4.97e+00 & 2 & 2.30e-01 & 2 & 2.30e-01 \\
      64  & 8 & $512^2$  & 2 & 3.43e+01 & 3 & 1.39e+00 & 3 & 1.39e+00 \\
      128 & 8 & $1024^2$ & 2 & 2.13e+02 & 4 & 8.43e+00 & 4 & 8.38e+00 \\
      256 & 8 & $2048^2$ & 2 & 1.25e+03 & 5 & 4.65e+01 & 4 & 3.93e+01 \\
      \hline
    \end{tabular}
  \end{center}
  \caption{Results of velocity field 2 for different $\omega$. Top: Solutions for two external forces with $\omega/(2\pi)=64$.
    Bottom: Results for different $\omega$.  }
  \label{tbl:2DPML2}
\end{table}

\begin{table}[h!]
  \begin{center}
    \includegraphics[height=2.3in]{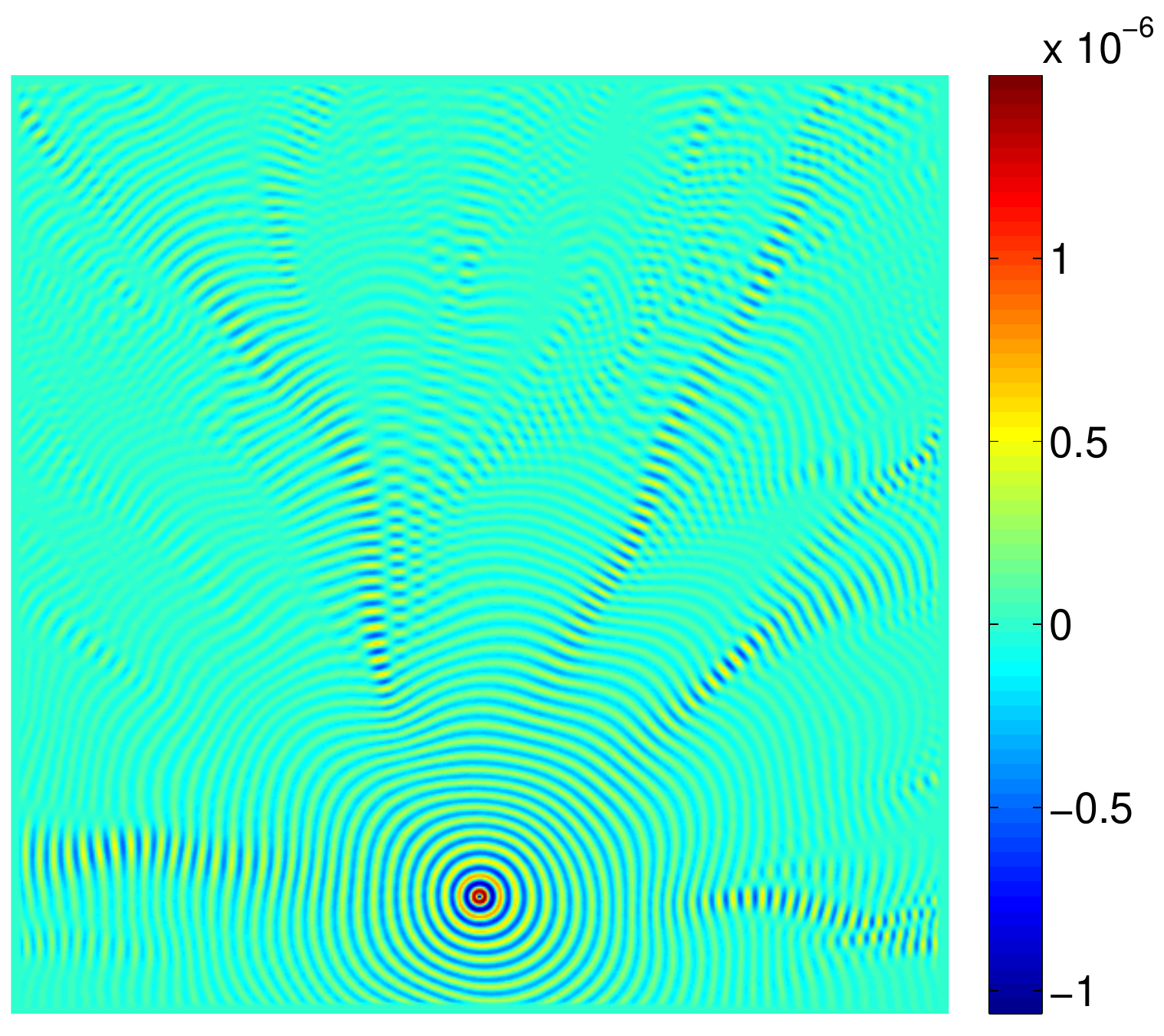}
    \includegraphics[height=2.3in]{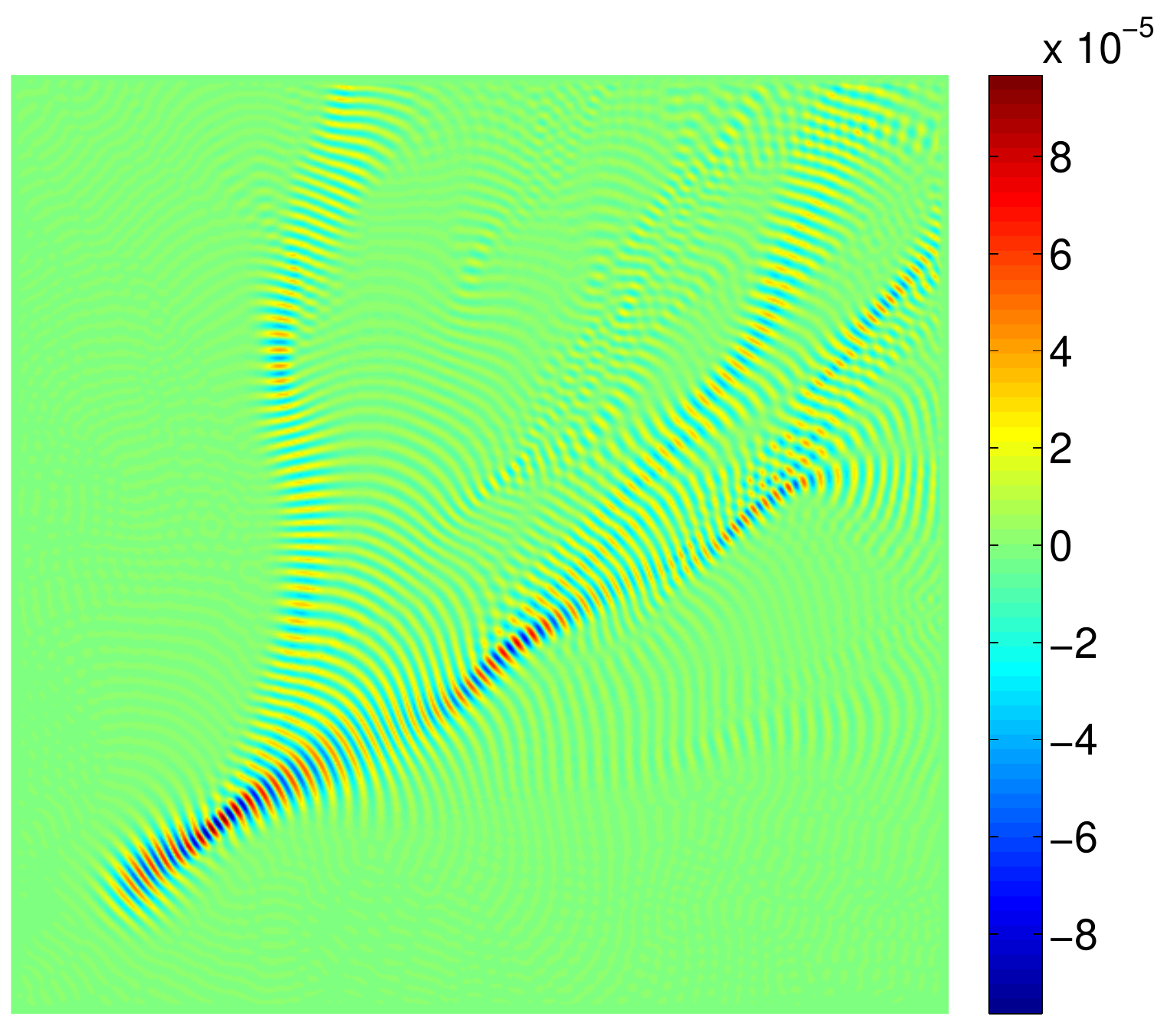}\\
    \begin{tabular}{|ccccc|cc|cc|}
      \hline
      \multicolumn{5}{|c|}{} & \multicolumn{2}{|c|}{Test 1} & \multicolumn{2}{|c|}{Test 2}\\
      \hline
      $\omega/(2\pi)$ & $q$ & $N=n^2$ & $R$ & $T_{\text{setup}}$ & $N_{\text{iter}}$ & $T_{\text{solve}}$ & $N_{\text{iter}}$ & $T_{\text{solve}}$\\
      \hline
      16  & 8 & $128^2$  & 2 & 6.50e-01 & 2 & 5.00e-02 & 2 & 5.00e-02 \\
      32  & 8 & $256^2$  & 2 & 5.10e+00 & 2 & 2.50e-01 & 3 & 3.00e-01 \\
      64  & 8 & $512^2$  & 2 & 3.48e+01 & 3 & 1.49e+00 & 3 & 1.48e+00 \\
      128 & 8 & $1024^2$ & 2 & 2.16e+02 & 4 & 8.99e+00 & 3 & 7.37e+00 \\
      256 & 8 & $2048^2$ & 2 & 1.26e+03 & 5 & 4.64e+01 & 3 & 3.25e+01 \\
      \hline
    \end{tabular}
  \end{center}
  \caption{Results of velocity field 3 for different $\omega$. Top: Solutions for two external forces with $\omega/(2\pi)=64$.
    Bottom: Results for different $\omega$.  }
  \label{tbl:2DPML4}
\end{table}

\paragraph{Dependence on $q$.} Secondly, we study how the sweeping
preconditioner behaves when the number of discretization points per
wavelength $q$ varies. Fix $\frac{\omega}{2\pi}$ at $32$ and let $q$
be $8,16,\ldots,64$. In the following tests, $R$ is again equal to
2. The sweeping direction is bottom-up from $x_2=0$ to $x_2=1$. The
test results for the three velocity fields are summarized in Tables
\ref{tbl:2DSPL1}, \ref{tbl:2DSPL2}, and \ref{tbl:2DSPL4},
respectively. These results show that the number of iterations remain
to be extremely small and the overall solution time scales roughly
linearly with respect to the number of unknowns.

\begin{table}[h!]
  \begin{center}
    \begin{tabular}{|ccccc|cc|cc|}
      \hline
      \multicolumn{5}{|c|}{} & \multicolumn{2}{|c|}{Test 1} & \multicolumn{2}{|c|}{Test 2}\\
      \hline
      $\omega/(2\pi)$ & $q$ & $N=n^2$ & $R$ & $T_{\text{setup}}$ & $N_{\text{iter}}$ & $T_{\text{solve}}$ & $N_{\text{iter}}$ & $T_{\text{solve}}$\\
      \hline
      32 & 8  & $256^2 $ & 2 & 4.93e+00 & 2 & 2.30e-01 & 2 & 2.30e-01 \\
      32 & 16 & $512^2 $ & 2 & 3.42e+01 & 2 & 1.11e+00 & 2 & 1.09e+00 \\
      32 & 32 & $1024^2$ & 2 & 2.13e+02 & 2 & 5.45e+00 & 2 & 5.45e+00 \\
      32 & 64 & $2048^2$ & 2 & 1.23e+03 & 2 & 2.50e+01 & 2 & 2.49e+01 \\
      \hline
    \end{tabular}
  \end{center}
  \caption{Results of velocity field 1 for different $q$.}
  \label{tbl:2DSPL1}
\end{table}

\begin{table}[h!]
  \begin{center}
    \begin{tabular}{|ccccc|cc|cc|}
      \hline
      \multicolumn{5}{|c|}{} & \multicolumn{2}{|c|}{Test 1} & \multicolumn{2}{|c|}{Test 2}\\
      \hline
      $\omega/(2\pi)$ & $q$ & $N=n^2$ & $R$ & $T_{\text{setup}}$ & $N_{\text{iter}}$ & $T_{\text{solve}}$ & $N_{\text{iter}}$ & $T_{\text{solve}}$\\
      \hline
      32 & 8  & $256^2$  & 2 & 4.93e+00 & 2 & 2.30e-01 & 2 & 2.30e-01 \\
      32 & 16 & $512^2$  & 2 & 3.42e+01 & 2 & 1.11e+00 & 2 & 1.09e+00 \\
      32 & 32 & $1024^2$ & 2 & 2.13e+02 & 2 & 5.45e+00 & 2 & 5.37e+00 \\
      32 & 64 & $2048^2$ & 2 & 1.23e+03 & 2 & 2.50e+01 & 2 & 2.49e+01 \\
      \hline
    \end{tabular}
  \end{center}
  \caption{Results of velocity field 2 for different $q$.}
  \label{tbl:2DSPL2}
\end{table}

\begin{table}[h!]
  \begin{center}
    \begin{tabular}{|ccccc|cc|cc|}
      \hline
      \multicolumn{5}{|c|}{} & \multicolumn{2}{|c|}{Test 1} & \multicolumn{2}{|c|}{Test 2}\\
      \hline
      $\omega/(2\pi)$ & $q$ & $N=n^2$ & $R$ & $T_{\text{setup}}$ & $N_{\text{iter}}$ & $T_{\text{solve}}$ & $N_{\text{iter}}$ & $T_{\text{solve}}$\\
      \hline
      32 & 8  & $256^2 $ & 2 & 5.13e+00 & 2 & 2.40e-01 & 3 & 3.10e-01 \\
      32 & 16 & $512^2 $ & 2 & 3.47e+01 & 2 & 1.21e+00 & 2 & 1.20e+00 \\
      32 & 32 & $1024^2$ & 2 & 2.14e+02 & 2 & 5.87e+00 & 2 & 5.84e+00 \\
      32 & 64 & $2048^2$ & 2 & 1.23e+03 & 2 & 2.52e+01 & 2 & 2.51e+01 \\
      \hline
    \end{tabular}
  \end{center}
  \caption{Results of velocity field 3 for different $q$.}
  \label{tbl:2DSPL4}
\end{table}

\paragraph{Dependence on sweeping direction.} Next, we study how the
sweeping directions affect the convergence rate of the GMRES
algorithm. In this example, the velocity field is given by $c(x_1,x_2)
= 1/2 + x_2$. The external force is a a narrow Gaussian point source
centered at $(x_1,x_2)=(0.125,0.5)$. Two sweeping directions are
tested here: the first one sweeps in the positive $x_2$ direction
while the second sweeps in the negative $x_2$ direction. For the first
sweeping direction, the matrix $T_m$ approximates the Green's function
of the lower half-space $(-\infty,\infty)\times (-\infty,mh)$. Since
the velocity field decreases in the negative $x_2$ direction, in a
geometric optics argument the rays emanating from $x_2=mh$ plane do
not travel back to the same plane. Therefore, we expect that the
numerical rank of the off-diagonal blocks of $T_m$ to be low, the
preconditioner to be quite accurate, and the number of iterations to
be small. The geometric theory of diffraction indicates that the
coupling between points on the plane $x_2=mh$ is via exponentially
decaying creeping rays and thus very weak.

For the second sweeping direction, the matrix $T_m$
approximates the Green's function of the upper half-space
$(-\infty,\infty)\times(1-mh,\infty)$. Since the velocity field
increases in the positive $x_2$ direction, the rays emanating from
$x_2=mh$ can shoot back to the same plane. As a result, the
hierarchical matrix representation of $T_m$ would incur larger error
for the same $R$ value and the number of iterations would become
larger.

Table \ref{tbl:2DDIR} reports the results of these two sweeping
directions for different $\omega$ values. As expected by the above
argument, the number of iterations for the first sweeping
preconditioner (in the positive $x_2$ direction) remains very small
while the number of iterations for the second one (in the negative
$x_2$ direction) increases slightly with $N$.

\begin{table}[h!]
  \begin{center}
    \includegraphics[height=1.6in]{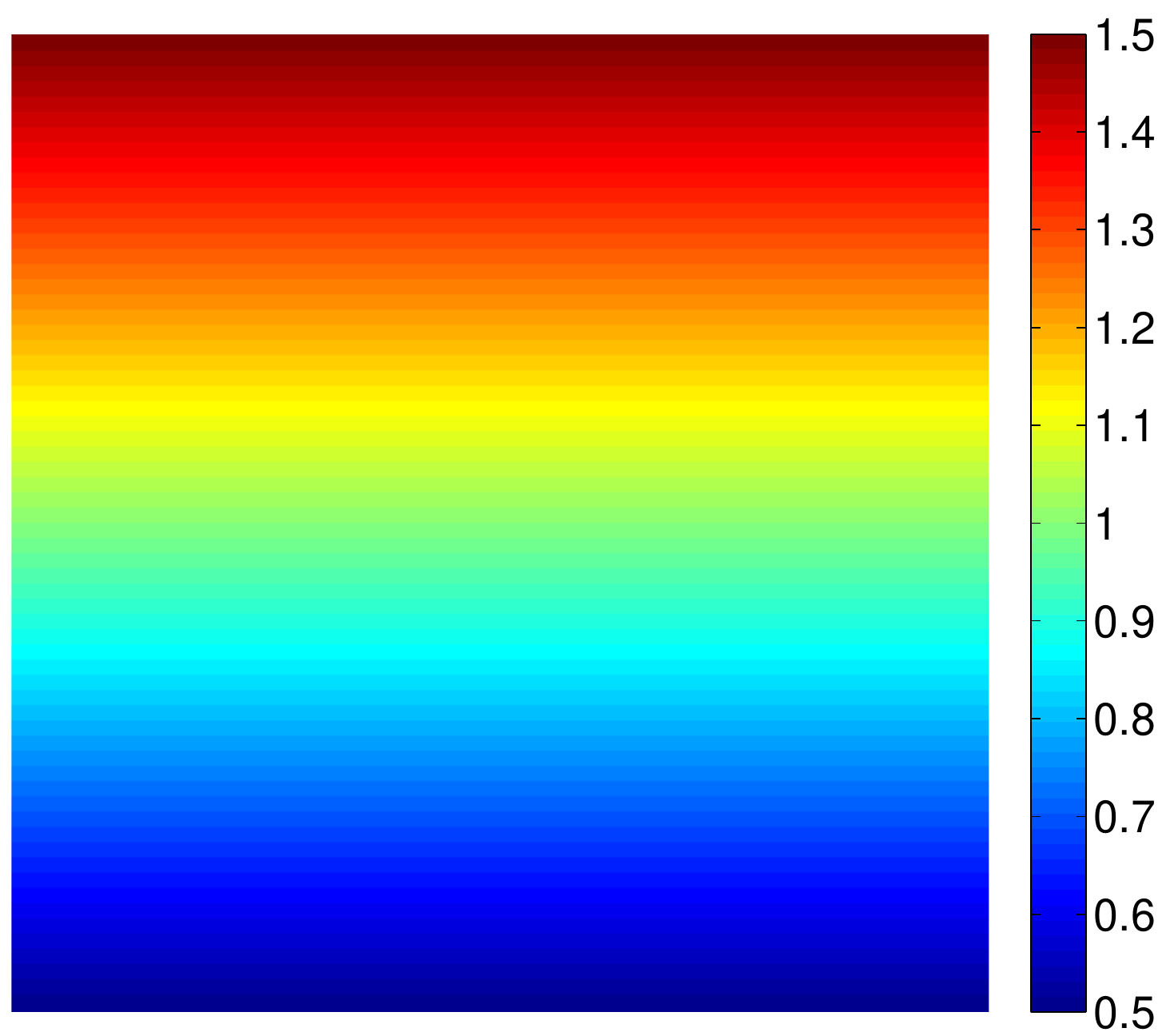}
    \includegraphics[height=2.3in]{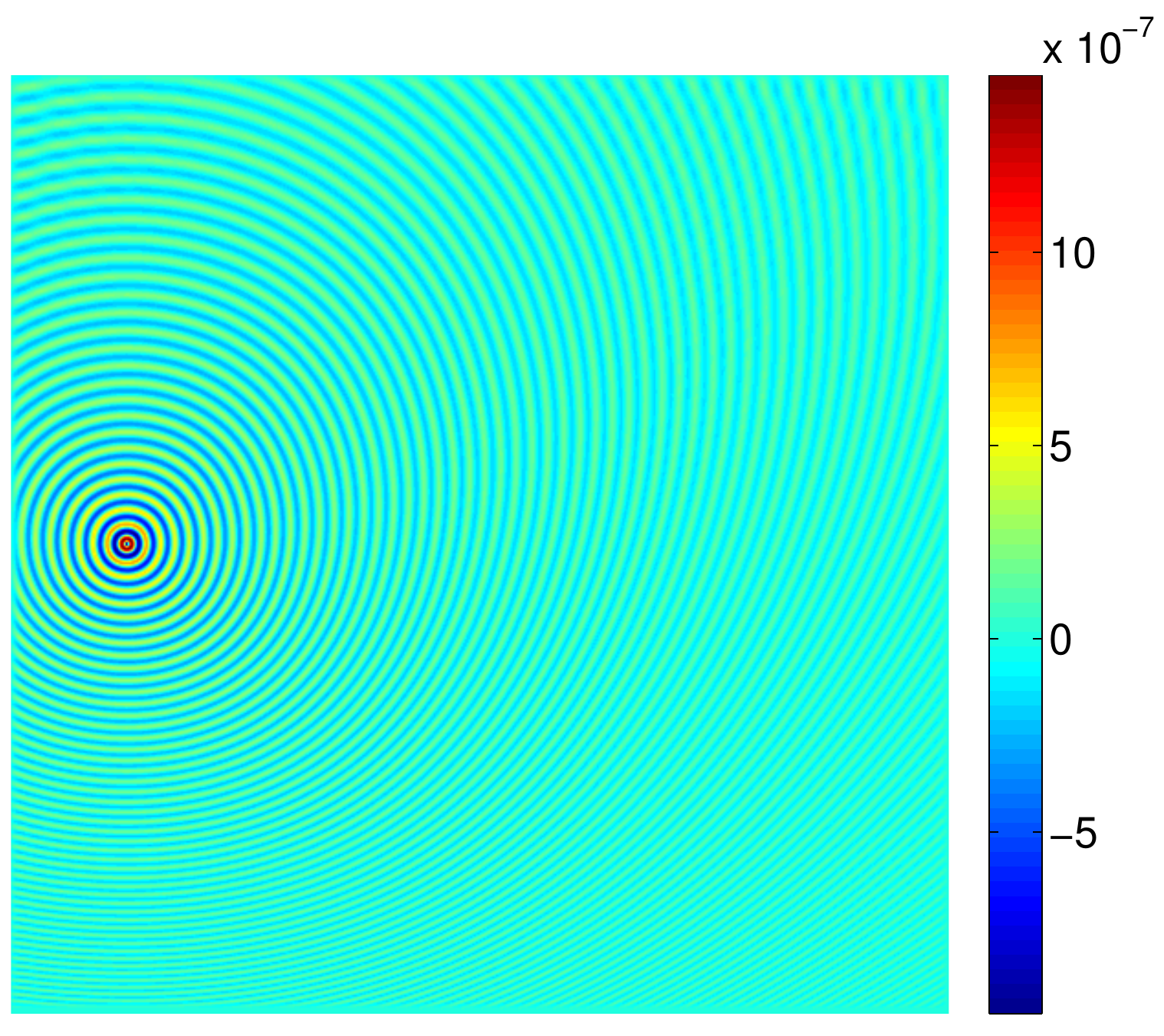}\\
    \begin{tabular}{|ccccc|cc|cc|}
      \hline
      \multicolumn{5}{|c|}{} & \multicolumn{2}{|c|}{Positive $x_2$} & \multicolumn{2}{|c|}{Negative $x_2$}\\
      \hline
      $\omega/(2\pi)$ & $q$ & $N=n^2$ & $R$ & $T_{\text{setup}}$ & $N_{\text{iter}}$ & $T_{\text{solve}}$ & $N_{\text{iter}}$ & $T_{\text{solve}}$\\
      \hline
      16  & 8 & $128^2$  & 2 & 6.80e-01 & 1 & 4.00e-02 & 2 & 5.00e-02 \\
      32  & 8 & $256^2$  & 2 & 4.95e+00 & 2 & 2.50e-01 & 3 & 3.00e-01 \\
      64  & 8 & $512^2$  & 2 & 3.40e+01 & 2 & 1.13e+00 & 4 & 1.86e+00 \\
      128 & 8 & $1024^2$ & 2 & 2.14e+02 & 2 & 5.82e+00 & 6 & 1.21e+01 \\
      256 & 8 & $2048^2$ & 2 & 1.25e+03 & 2 & 2.64e+01 & 6 & 5.49e+01 \\
      \hline
    \end{tabular}
  \end{center}
  \caption{Results of the positive and negative $x_2$ sweeping directions.
    Top row: the velocity field (left) and the solution for the external force (right).
    Bottom row: results for different $\omega$.
  }
  \label{tbl:2DDIR}
\end{table}

\subsection{Other boundary conditions}
\label{sec:2DnumOTH}

Here we report three examples with different boundary conditions. 

\paragraph{Depth extrapolation.} In the first example (see
Figure \ref{fig:other} (left)), the velocity field is a vertical wave
guide. We specify the Dirichlet boundary condition $u(x_1,1) = b(x_1)$
at the top edge $x_2=1$ and the PML at the other three edges. This is
the depth extrapolation problem in reflection seismology and we report
the results of two test cases:
\begin{enumerate}
\item $b(x_1)=1$. This corresponds to a plane wave entering the wave
  guide. The center part of the plane wave should start to bend and
  eventually form multiple caustics.
\item $b(x_1)=\exp\left(i\frac{\omega}{2}x_1\right)$. This corresponds
  to a slant wave entering the wave guide.
\end{enumerate}
The sweeping direction is bottom-up from $x_2=0$ to $x_2=1$. The
results are summarized in Table \ref{tbl:2DOTH1}. The running time
again follows closely the analytical estimate and the number of GMRES
iterations are bounded by $4$.

\begin{table}[h!]
  \begin{center}
    \includegraphics[height=2.3in]{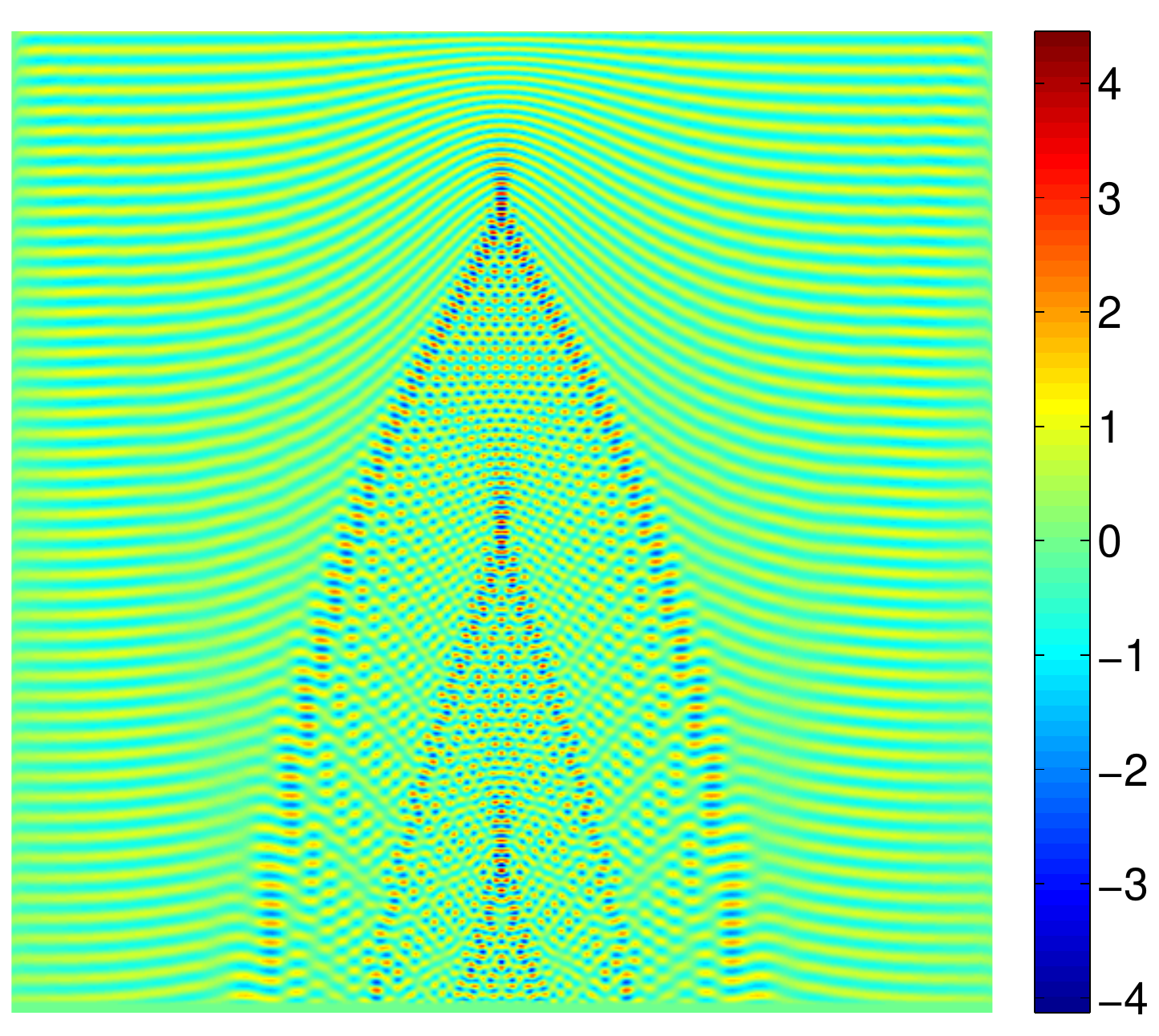}
    \includegraphics[height=2.3in]{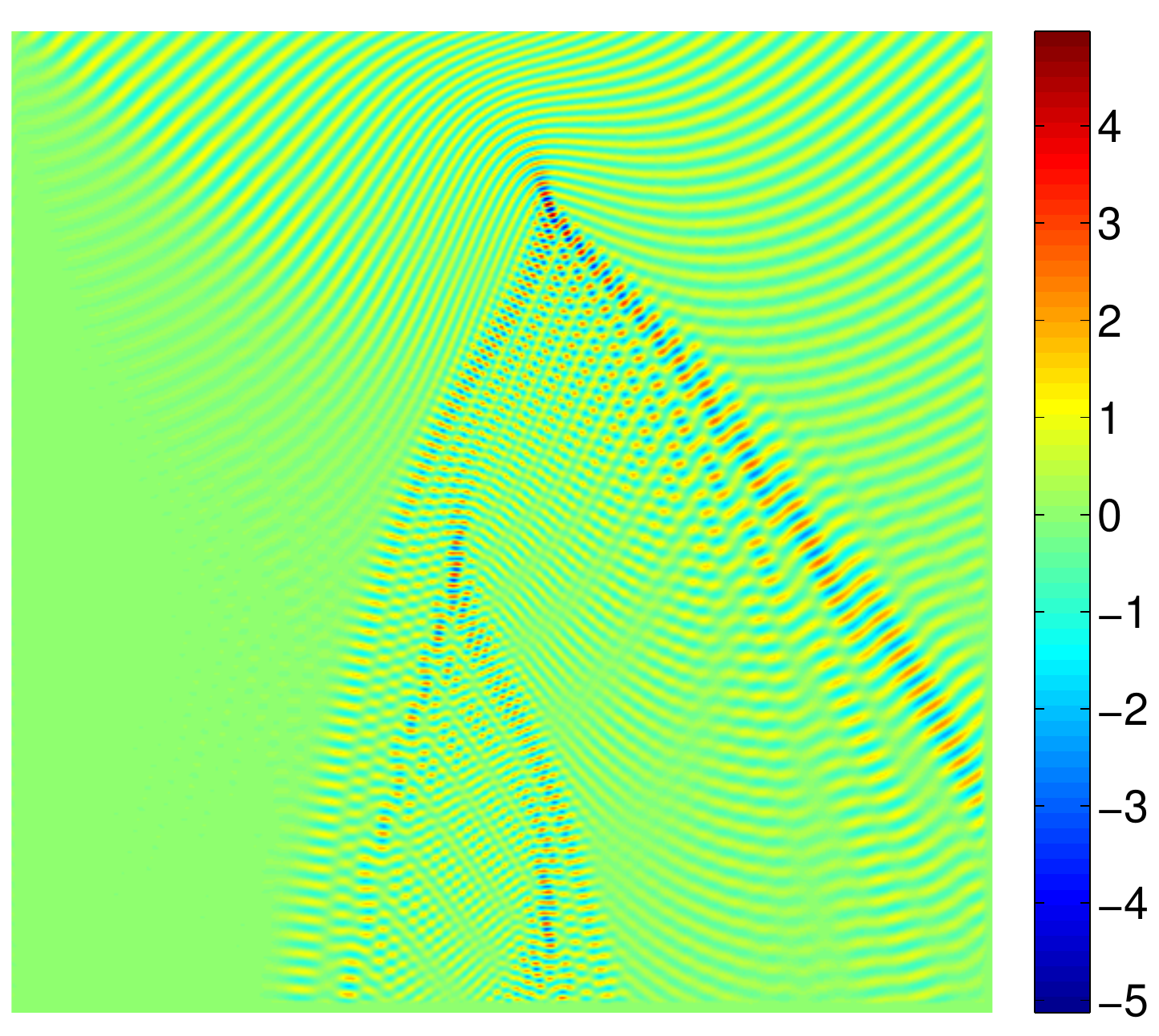}\\
    \begin{tabular}{|ccccc|cc|cc|}
      \hline
      \multicolumn{5}{|c|}{} & \multicolumn{2}{|c|}{Test 1} & \multicolumn{2}{|c|}{Test 2}\\
      \hline
      $\omega/(2\pi)$ & $q$ & $N=n^2$ & $R$ & $T_{\text{setup}}$ & $N_{\text{iter}}$ & $T_{\text{solve}}$ & $N_{\text{iter}}$ & $T_{\text{solve}}$\\
      \hline
      16  & 8 & $128^2$  & 2 & 6.60e-01 & 2 & 4.00e-02 & 2 & 4.00e-02 \\
      32  & 8 & $256^2$  & 2 & 5.07e+00 & 3 & 3.20e-01 & 3 & 3.00e-01 \\
      64  & 8 & $512^2$  & 2 & 3.45e+01 & 3 & 1.48e+00 & 3 & 1.46e+00 \\
      128 & 8 & $1024^2$ & 2 & 2.15e+02 & 3 & 7.29e+00 & 3 & 7.30e+00 \\
      256 & 8 & $2048^2$ & 2 & 1.25e+03 & 4 & 3.92e+01 & 4 & 3.94e+01 \\
      \hline
    \end{tabular}
  \end{center}
  \caption{Results of the depth stepping example for different $\omega$. 
    Top: Solutions for two test cases with $\omega/(2\pi)=64$.
    Bottom: Results for different $\omega$.  }
  \label{tbl:2DOTH1}
\end{table}

\paragraph{Mixed PML-Dirichlet boundary condition.}
In the second example, the velocity field $c(x)$ is equal to constant
one and perform two tests with mixed boundary conditions.
\begin{enumerate}
\item In the first test (see Figure \ref{fig:other} (middle)), we
  specify the zero Dirichlet boundary condition at $x_1=0$ and $x_1=$
  and the PML condition at the other two sides.  The external force
  $f(x)$ is a Gaussian wave packet with a wavelength comparable to the
  typical wavelength of the Helmholtz equation. This packet centers at
  $(x_1,x_2) = (0.5, 0.125)$ and points to the
  $(\cos(\pi/8),\sin(\pi/8))$ direction. The Gaussian beam generated
  by this forcing term should bounce back from the edge $x_1=1$ and
  then from the edge $x_1=0$.
\item In the second test (see Figure \ref{fig:other} (right)), we
  specify the zero Dirichlet boundary condition at $x_1=1$ and $x_2=1$
  and the PML condition at the other two sides. The external force
  $f(x)$ is a Gaussian wave packet with a wavelength comparable to the
  typical wavelength of the Helmholtz equation. This packet centers at
  $(x_1,x_2) = (0.5, 0.125)$ and points to the $(1,1)$ direction. The
  Gaussian beam generated by this forcing term should bounce back from
  the edge $x_1=1$ and then from the edge $x_2=1$.
\end{enumerate}
The sweeping direction is bottom-up from $x_2=0$ to $x_2=1$ and the
results of these tests are summarized in Table \ref{tbl:2DOTH2}.  The
running time again follows the analytical estimate. For the first test
case with zero Dirichlet boundary condition at $x_1=0$ and $x_1=1$,
due to the reason mentioned in Section \ref{sec:2Dpreother}, the rank
of the off-diagonal blocks of the Schur complement matrices are
slightly higher. Hence, with the same $R$ value the number of
iterations is expected to increase slightly. In all cases, the number
of GMRES iterations is bounded by 10.

\begin{table}[h!]
  \begin{center}
    \includegraphics[height=2.3in]{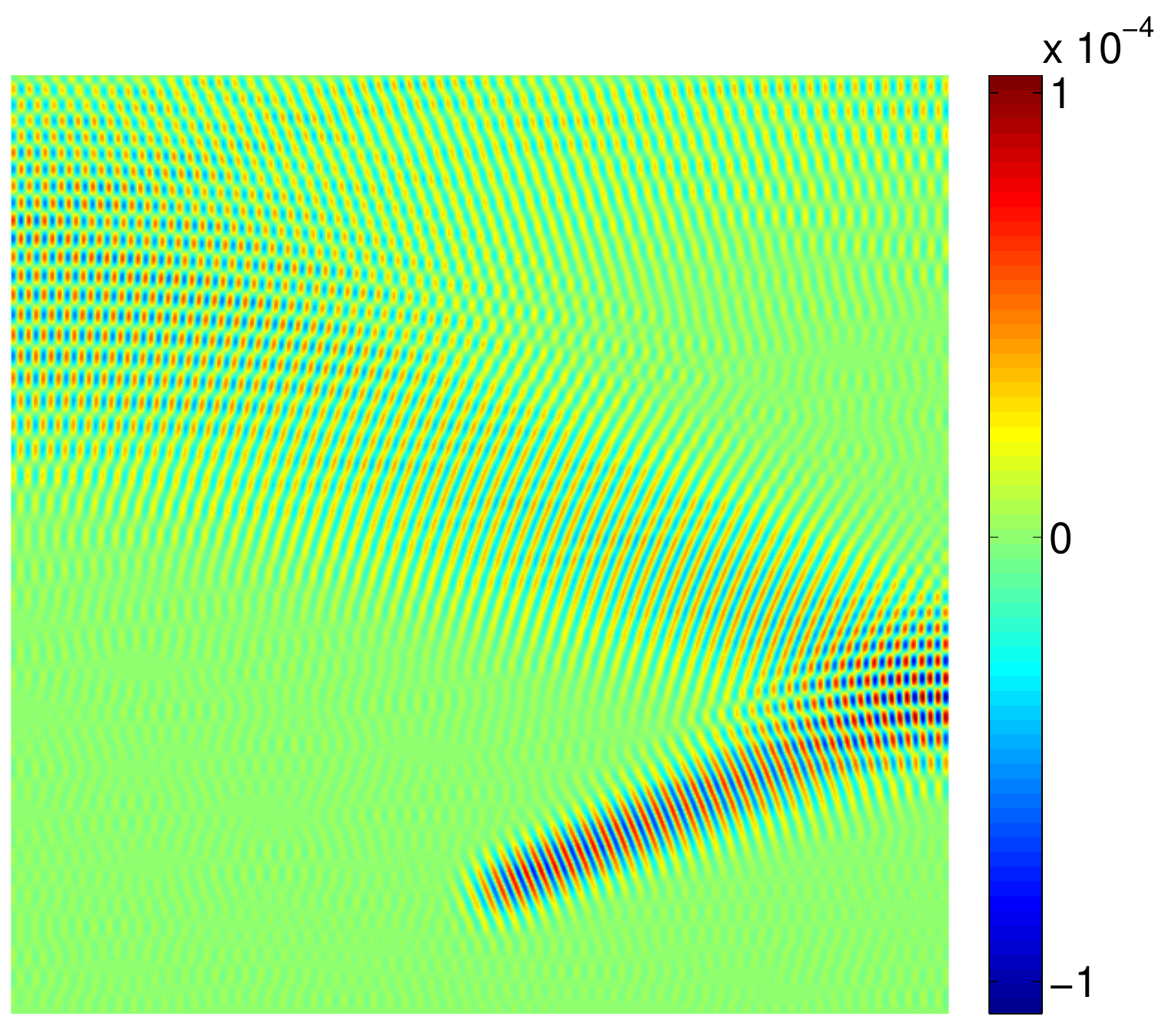}
    \includegraphics[height=2.3in]{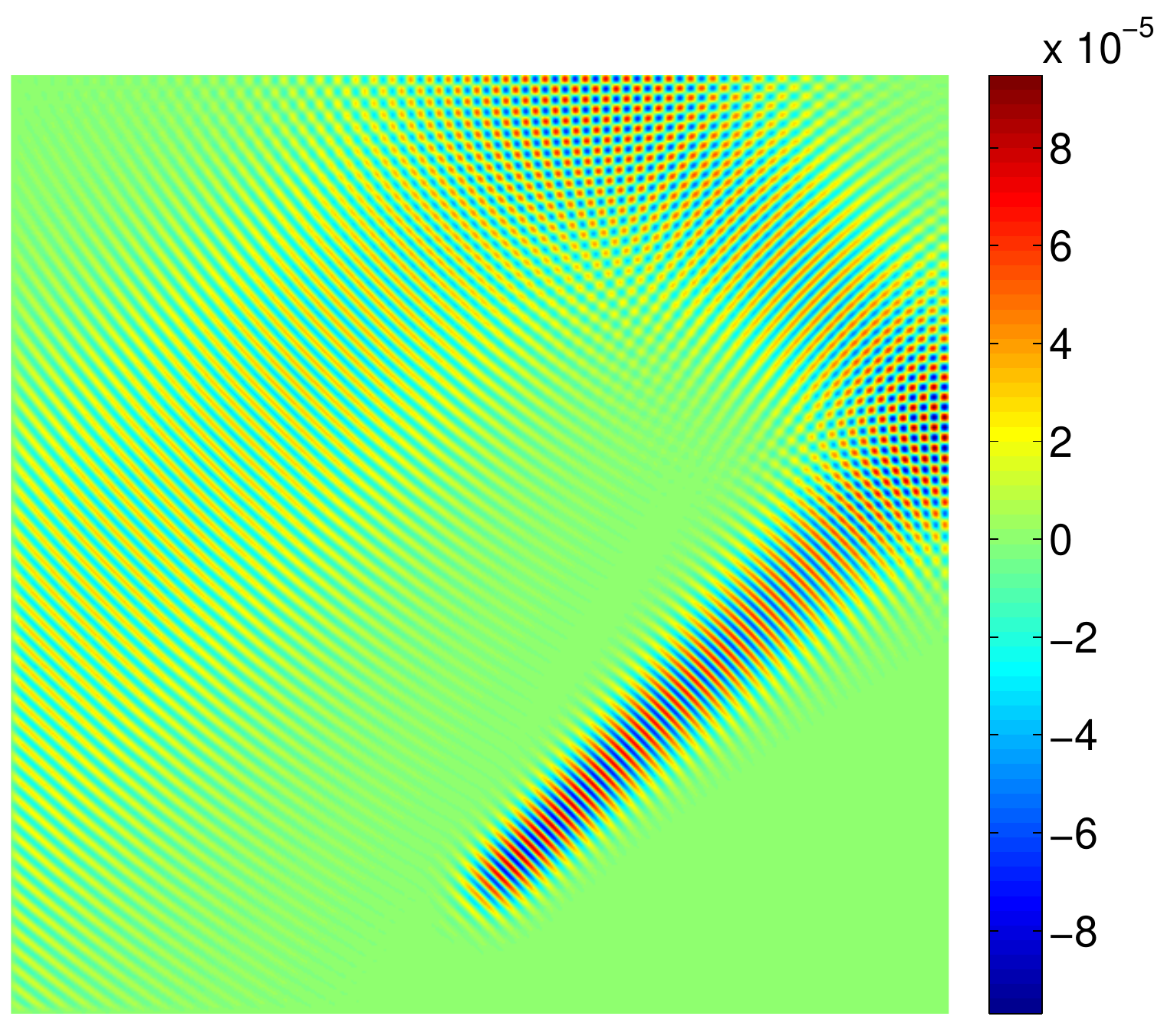}\\
    \begin{tabular}{|ccccc|cc|cc|}
      \hline
      \multicolumn{5}{|c|}{} & \multicolumn{2}{|c|}{Test 1} & \multicolumn{2}{|c|}{Test 2}\\
      \hline
      $\omega/(2\pi)$ & $q$ & $N=n^2$ & $R$ & $T_{\text{setup}}$ & $N_{\text{iter}}$ & $T_{\text{solve}}$ & $N_{\text{iter}}$ & $T_{\text{solve}}$\\
      \hline
      16  & 8 & $128^2$  & 2 & 6.80e-01 & 2 & 5.00e-02 & 2 & 5.00e-02 \\
      32  & 8 & $256^2$  & 2 & 5.00e+00 & 3 & 3.10e-01 & 2 & 2.50e-01 \\
      64  & 8 & $512^2$  & 2 & 3.47e+01 & 6 & 2.70e+00 & 2 & 1.27e+00 \\
      128 & 8 & $1024^2$ & 2 & 2.16e+02 & 9 & 1.80e+01 & 2 & 6.19e+00 \\
      256 & 8 & $2048^2$ & 2 & 1.26e+03 & 10& 8.52e+01 & 2 & 2.69e+01 \\
      \hline
    \end{tabular}
  \end{center}
  \caption{Results of the mixed boundary condition example for different $\omega$. 
    Top: Solutions for two test cases with $\omega/(2\pi)=64$.
    Bottom: Results for different $\omega$. }
  \label{tbl:2DOTH2}
\end{table}

\paragraph{Absorbing boundary condition.} In the last example, we
replace the PML with the second order absorbing boundary condition
(ABC). The velocity field $c(x)$ is taken to be one and we perform
tests with two different external forces, which are similar to the
ones given at the beginning of Section \ref{sec:2DnumPML}. However,
since the low order ABCs generate more non-physical reflections at the
domain boundaries, we move the support of these external forces closer
to the center of the computational domain.
\begin{enumerate}
\item The first external force $f(x)$ is a Gaussian point source
  located at $(x_1,x_2) = (0.5, 0.25)$.
\item The second external force $f(x)$ is a Gaussian wave packet with
  a wavelength comparable to the typical wavelength of the Helmholtz
  equation. This packet centers at $(x_1,x_2) = (0.25, 0.25)$ and
  points to the $(1,1)$ direction.
\end{enumerate}
Due to the same non-physical reflections, the discrete Green's
function associated with a low order ABC often has off-diagonal blocks
with higher numerical ranks compared to the discrete Green's function
associated with the PML. As a result, we let $R$ increase slightly
with $\omega$. The sweeping direction is bottom-up from $x_2=0$ to
$x_2=1$ and the results are summarized in Table \ref{tbl:2DABC0}. The
setup time grows slightly higher than linear complexity due to the
increase of $R$. The number of iteration increases roughly
logarithmically with respect to $\omega$. In all cases, the number of
GMRES iterations is bounded by 13. Overall the results for the ABC
compares slightly worse than the ones of the PML, suggesting that, in
order for the sweeping preconditioner to work well, it is essential to
minimize non-physical reflections at the domain boundary.

\begin{table}[h!]
  \begin{center}
    \includegraphics[height=2.3in]{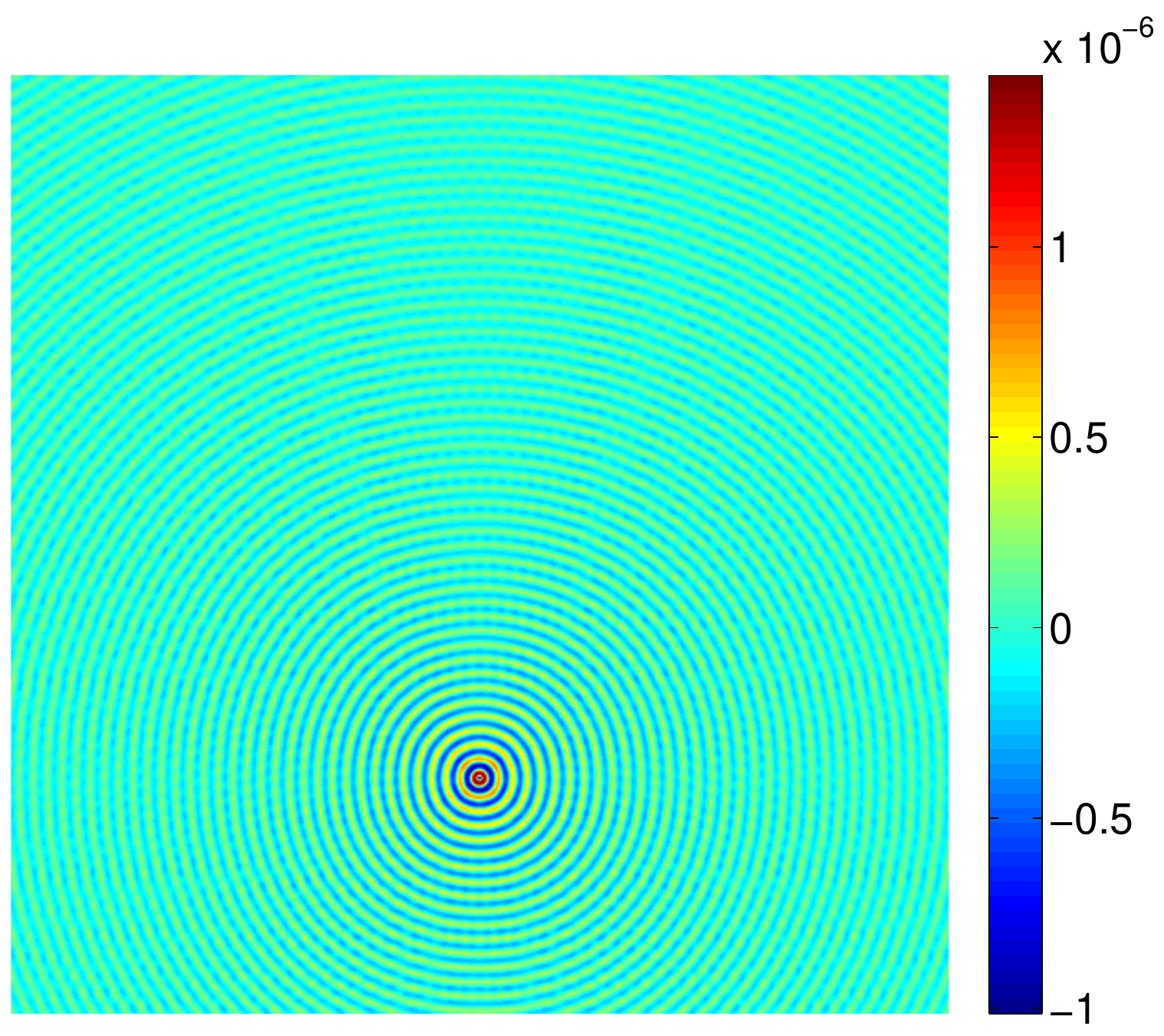}
    \includegraphics[height=2.3in]{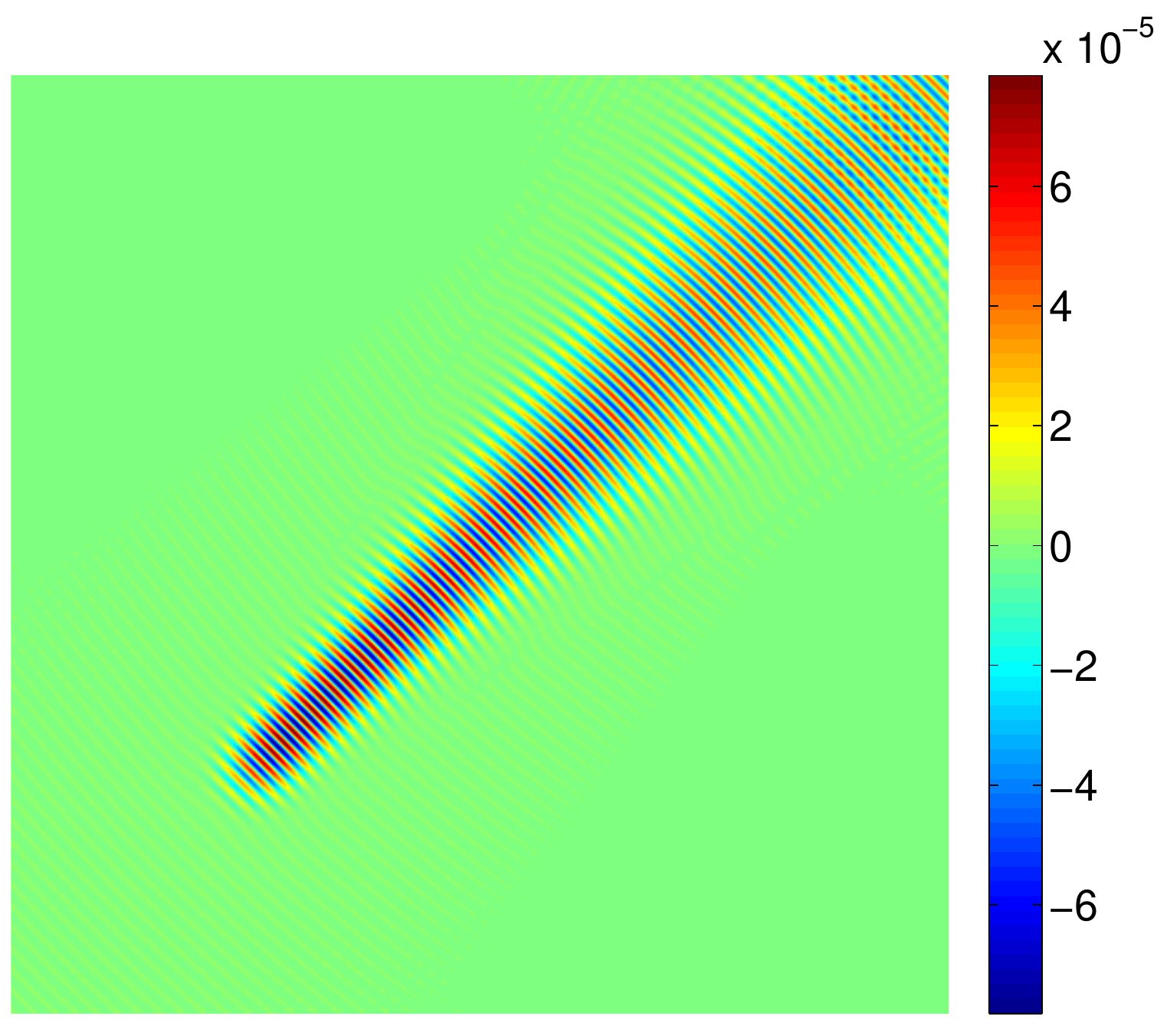}\\
    \begin{tabular}{|ccccc|cc|cc|}
      \hline
      \multicolumn{5}{|c|}{} & \multicolumn{2}{|c|}{Test 1} & \multicolumn{2}{|c|}{Test 2}\\
      \hline
      $\omega/(2\pi)$ & $q$ & $N=n^2$ & $R$ & $T_{\text{setup}}$ & $N_{\text{iter}}$ & $T_{\text{solve}}$ & $N_{\text{iter}}$ & $T_{\text{solve}}$\\
      \hline
      16  & 8 & $128^2$  & 2 & 6.70e-01 & 7 & 1.30e-01 & 6 & 1.50e-01\\
      32  & 8 & $256^2$  & 2 & 4.98e+00 & 7 & 4.80e-01 & 6 & 4.30e-01\\
      64  & 8 & $512^2$  & 3 & 5.10e+01 & 8 & 3.16e+00 & 6 & 2.42e+00\\
      128 & 8 & $1024^2$ & 4 & 4.65e+02 & 10& 2.06e+01 & 6 & 1.33e+01\\
      256 & 8 & $2048^2$ & 5 & 3.84e+03 & 13& 1.59e+01 & 6 & 8.49e+01\\
      \hline
    \end{tabular}
  \end{center}
  \caption{Results of the absorbing boundary condition (ABC) test for different $\omega$. 
    Top: Solutions for two test cases with $\omega/(2\pi)=64$.
    Bottom: Results for different $\omega$. }
  \label{tbl:2DABC0}
\end{table}

\section{Preconditioner in 3D}
\label{sec:3Dpre}

\subsection{Discretization}
\label{sec:3Dpredisc}

The computational domain is $D=(0,1)^3$. Using the same $\sigma(t)$
defined in \eqref{eq:sigma}, we define
\[
s_1(x_1) = \left( 1+i \frac{\sigma(x_1)}{\omega} \right)^{-1},\quad
s_2(x_2) = \left( 1+i \frac{\sigma(x_2)}{\omega} \right)^{-1},\quad
s_3(x_3) = \left( 1+i \frac{\sigma(x_3)}{\omega} \right)^{-1}.
\]
The PML replaces $\p_1$ with $s_1(x_1)\p_1$, $\p_2$ with
$s_2(x_2)\p_2$, and $\p_3$ with $s_3(x_3)\p_3$. This effectively
provides a damping layer of width $\eta$ near the boundary of
$D=(0,1)^3$. The resulting equation is
\begin{eqnarray*}
  \left( (s_1\p_1)(s_1\p_1) + (s_2\p_2)(s_2\p_2) + (s_3\p_3)(s_3\p_3) + \frac{\omega^2}{c^2(x)} \right) u = f && x\in D=[0,1]^3,\\
  u = 0 && x \in \p D.
\end{eqnarray*}
Without loss of generality, we assume that $f(x)$ is supported inside
$[\eta,1-\eta]^3$ (away from the PML). Dividing the above
equation by $s_1 s_2 s_3$ results
\[
\left( \p_1\left(\frac{s_1}{s_2s_3} \p_1\right) + \p_2\left(\frac{s_2}{s_1s_3} \p_2\right) + 
  \p_3\left(\frac{s_3}{s_1s_2} \p_3\right) + 
  \frac{\omega^2}{s_1s_2s_3c^2(x)} \right) u = f.
\]
The domain $[0,1]^3$ is discretized with a Cartesian grid with spacing
$h = 1/(n+1)$. As we discretize the equation with a couple number of
points per wavelength, the number $n$ of samples in each dimension is
proportional to $\omega$. The interior points of this grid are
\[
\P = \{ p_{i,j,k} = \left(ih,jh,kh\right) : 1\le i,j,k \le  n\}
\]
(see Figure \ref{fig:3Dgrid} (left)) and the total number of points is
equal to $N=n^3$.

\begin{figure}[h!]
  \begin{center}
    \includegraphics{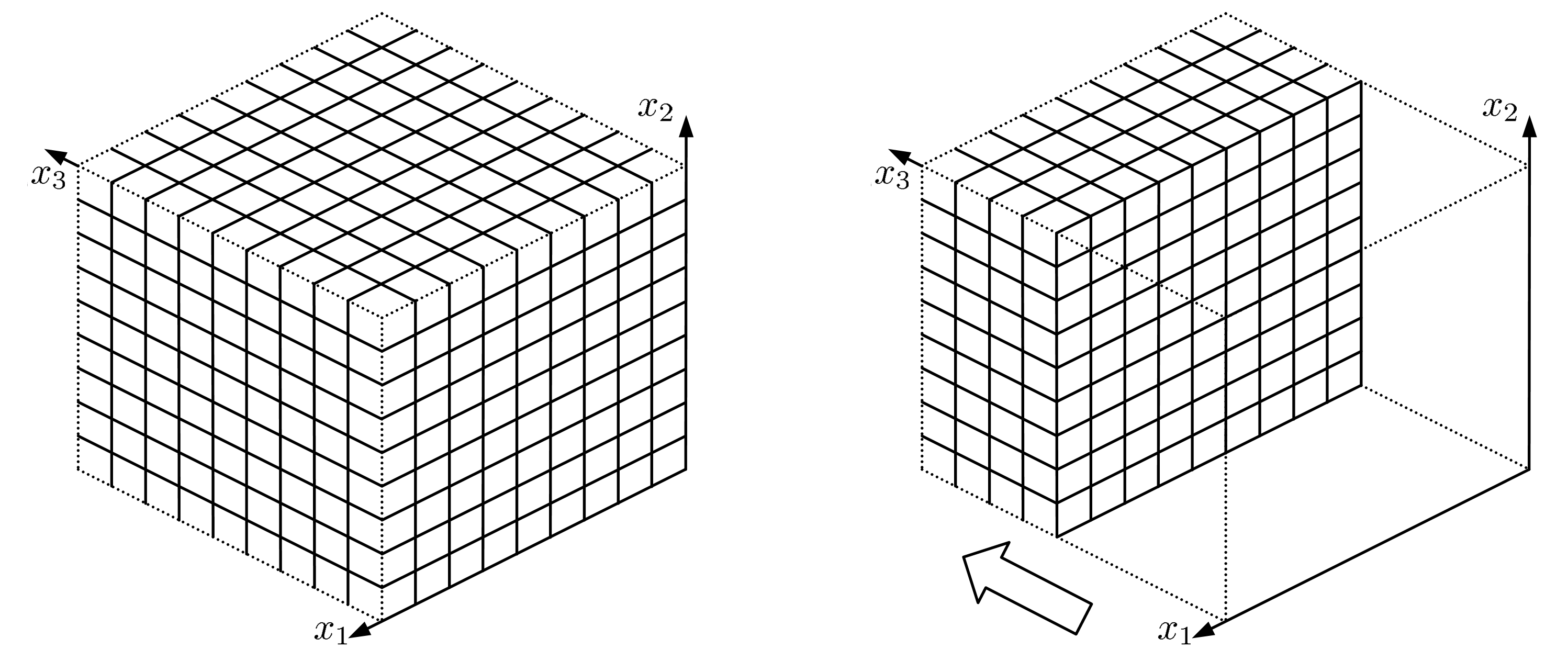}
  \end{center}
  \caption{Left: Discretization grid in 3D. Right: Sweeping order in
    3D. The remaining grid shows the unknowns yet to be processed.}
  \label{fig:3Dgrid}
\end{figure}

We denote by $u_{i,j,k}$, $f_{i,j,k}$, and $c_{i,j,k}$ the values of
$u(x)$, $f(x)$, and $c(x)$ at point $p_{i,j,k} = (ih,jh,kh)$. The
7-point stencil finite difference method writes down the equation at
points in $\P$ using central difference. The resulting equation at
$p_{i,j,k} = (ih,jh,kh)$ is
\begin{multline*}
  \frac{1}{h^2} \left(\frac{s_1}{s_2s_3}\right)_{i-\hf,j,k} u_{i-1,j,k}
  + \frac{1}{h^2} \left(\frac{s_1}{s_2s_3}\right)_{i+\hf,j,k} u_{i+1,j,k}
  + \frac{1}{h^2} \left(\frac{s_2}{s_1s_3}\right)_{i,j-\hf,k} u_{i,j-1,k} \\
  + \frac{1}{h^2} \left(\frac{s_2}{s_1s_3}\right)_{i,j+\hf,k} u_{i,j+1,k}
  + \frac{1}{h^2} \left(\frac{s_3}{s_1s_2}\right)_{i,j,k-\hf} u_{i,j,k-1}
  + \frac{1}{h^2} \left(\frac{s_3}{s_1s_2}\right)_{i,j,j+\hf} u_{i,j,k+1}\\
  + \left(
    \frac{\omega^2}{(s_1s_2s_3)_{i,j,k}\cdot c_{i,j,k}^2 }
    - \left(\cdots
    \right) 
  \right)
  u_{i,j,k} = f_{i,j,k}
\end{multline*} 
with $u_{i',j',k'}$ equal to zero for $(i',j',k')$ that violates $1\le
i',j',k'\le n$. Here $(\cdots)$ stands for the sum of the six
coefficients appeared in the first two lines. We order $u_{i,j,k}$ by
going through the dimensions in order and denote the vector containing
all unknowns by
\[
u = \left(u_{1,1,1}, u_{2,1,1},\ldots,u_{n,1,1}, \ldots,u_{1,n,n}, u_{2,n,n},\ldots,u_{n,n,n}\right)^t.
\]
Similarly, $f_{i,j,k}$ are ordered in the same way and the vector $f$ is
\[
f = \left(f_{1,1,1}, f_{2,1,1},\ldots,f_{n,1,1}, \ldots,f_{1,n,n}, f_{2,n,n},\ldots,f_{n,n,n}\right)^t.
\]
The whole system takes the form $A u = f$. We further introduce a
block version.  Define $\P_m$ to be the indices in the $m$-th row
\[
\P_m = \{p_{1,1,m}, p_{2,1,m}, \ldots, p_{n,n,m}\}  
\]
and introduce
\[
u_m = \left( u_{1,1,m}, u_{2,1,m},\ldots,u_{n,n,m} \right)^t,\quad
f_m = \left( f_{1,1,m}, f_{2,1,m},\ldots,f_{n,n,m} \right)^t.
\]
Then
\[
u = (u_1^t, u_2^t, \ldots, u_n^t)^t,\quad
f = (f_1^t, f_2^t, \ldots, f_n^t)^t.
\]
Using these notations, the system $Au=f$ takes the following block
tridiagonal form
\[
\begin{pmatrix}
  A_{1,1} & A_{1,2} & & \\
  A_{2,1} & A_{2,2} & \ddots & \\
  & \ddots & \ddots & A_{n-1,n}\\
  & & A_{n,n-1} & A_{n,n}
\end{pmatrix}
\begin{pmatrix}
  u_1\\
  u_2\\
  \vdots\\
  u_n
\end{pmatrix}
=
\begin{pmatrix}
  f_1\\
  f_2\\
  \vdots\\
  f_n
\end{pmatrix}
\]
where each block $A_{i,j}$ is of size $n^2 \times n^2$ and $A_{m,m-1}
= A_{m-1,m}^t$ are diagonal matrices. Similar to the 2D case, the
sweeping factorization eliminates the unknowns face by face, starting
from the face next to $x_3=0$ (illustrated in Figure \ref{fig:3Dgrid}
(right)). The algorithms for constructing and applying the sweeping
factorization are exactly the same as Algorithms \ref{alg:setupext}
and \ref{alg:solveext}). The matrix $T_m = S_m^{-1}$ is now the
discrete half-space Green's function with zero boundary condition at
$x_3 = (m+1)h$, restricted to the points on $x_3=mh$. Recall that in
the 2D case the off-diagonal blocks of $T_m$ is numerically
low-rank. In the 3D case, this is no longer exactly true. On the other
hand, since we only aim at constructing a preconditioner for the
Helmholtz problem, it is still reasonable to introduce a hierarchical
structure on the unknowns on the face $x_3=mh$ and use the
hierarchical matrix framework to approximate $T_m$ and $S_m$.

\subsection{Hierarchical matrix representation}
\label{sec:3Dprehier}

At the $m$-th layer for any fixed $m$, we build a hierarchical
structure for the grid points in $\P_m$ through bisections in both
$x_1$ and $x_2$ directions. At the top level (level 0), the set
\[
\J^0_{11} = \P_m.
\]
At level $\ell$, there are $2^\ell \times 2^\ell$ index sets
$\J^\ell_{ij}, i,j=1,\ldots,2^\ell$
\[
\J^\ell_{ij} = 
\{p_{s,t,m}:
(i-1)\cdot n/2^\ell + 1 \le s \le i\cdot n/2^\ell,
(j-1)\cdot n/2^\ell + 1 \le t \le j\cdot n/2^\ell
\}.
\]
The bisection is stopped when each set $\J^\ell_{ij}$ contains only a
small number of indices. Hence, the number of total levels $L$ is
equal to $\log_2 n - O(1)$. This hierarchical partition is illustrated
in Figure \ref{fig:3Dpart} (left)).
  
\begin{figure}[h!]
  \begin{center}
    \includegraphics{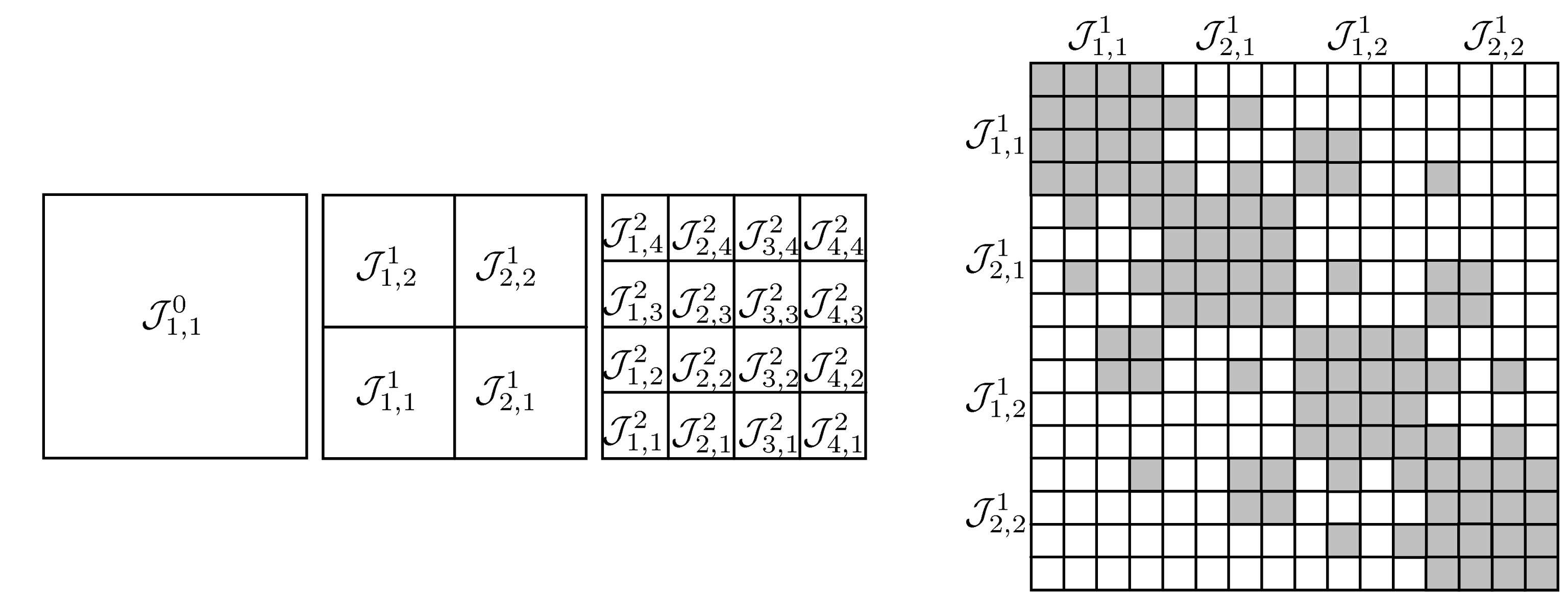}
  \end{center}
  \caption{Hierarchical matrix representation. Left: hierarchical
    decomposition of the index set $\J$ for each layer. Right: Induced
    partitioning of the matrix $T_m$ in the strongly admissible
    case. Blocks in white are stored in low-rank factorized
    form. Blocks in gray are stored densely.  }
  \label{fig:3Dpart}
\end{figure}

We write $G(\J^\ell_{ij}, \J^\ell_{i'j'})$ (the restriction of a
matrix $G$ to $\J^\ell_{ij}$ and $\J^\ell_{i'j'}$) as
$G^\ell_{ij,i'j'}$.  The strongly admissible case is used here and two
index sets $\J^\ell_{ij}$ and $\J^\ell_{i'j'}$ on the same level
$\ell$ are considered well-separated from each other if
$\max(|i-i'|,|j-j'|)>1$. Recall that the interaction list of
$\J^\ell_{ij}$ is defined to be the set of all index sets
$\J^\ell_{i'j'}$ such that $\J^\ell_{ij}$ is well-separated from
$\J^\ell_{i'j'}$ but $\J^\ell_{ij}$'s parent is not well-separated
from $\J^\ell_{i'j'}$'s parent. When $\J^\ell_{ij}$ and
$\J^\ell_{i'j'}$ are well-separated from each other, the numerical
rank of their interaction $G^\ell_{ij,i'j'}$ is of order
$O(n/2^\ell)$. As the number of indices in $\J^\ell_{ij}$ and
$\J^\ell_{i'j'}$ is equal to $(n/2^\ell)^2$, the numerical rank scales
like the square root of the number of indices in each set. Therefore,
it is still favorable to store the interaction $G^\ell_{ij,i'j'}$ in a
factorized form. In principle, the rank $R$ of the factorized form
should scale like $O(n/2^\ell)$. As the construction cost of the
approximate sweeping factorization scales like $O(R^2 n^3 \log^2 n) =
O(R^2 N \log^2 N)$, following this scaling can be rather costly in
practice. Instead, we choose $R$ to be a rather small constant as the
goal is only to construct a preconditioner. An illustration of this
hierarchical representation is given in Figure \ref{fig:3Dpart}
(right).

Once the details of the hierarchical matrix representation are
determined, the construction of the approximate $LDL^t$ factorization
and the application of its inverse take the same form as Algorithms
\ref{alg:setup} and \ref{alg:solve}, respectively. The operator
\[
M: 
f = (f_1^t, f_2^t, \ldots, f_n^t)^t \rightarrow
u = (u_1^t, u_2^t, \ldots, u_n^t)^t
\]
defined by Algorithm \ref{alg:solve} is an approximate inverse and a
good preconditioner of the discrete Helmholtz operator $A$.
Therefore, we solve the preconditioner system
\[
MA u = Mf
\]
using the GMRES algorithm. As the cost of applying $M$ to any vector
is $O(R n^3 \log n) = O(R N \log N)$, the total cost is $O(N_I R
n^3\log n) = O(N_I R N \log N)$, where $N_I$ is the number of
iterations. The numerical results in Section \ref{sec:3Dnum}
demonstrate that $N_I$ and $R$ are in practice rather small.

\section{Numerical Results in 3D}
\label{sec:3Dnum}

In this section, we present several numerical results to illustrate
the properties of the sweeping preconditioner described in Section
\ref{sec:3Dpre}.  We use the GMRES method as the iterative solver with
relative residue tolerance equal to $10^{-3}$.  The examples in this
seciton have the PML boundary condition specified at all sides.


We consider three velocity fields in the domain $[0,1]^3$:
\begin{enumerate}
\item The first velocity field is a converging lens with a Gaussian
  profile at the center of the domain (see Figure
  \ref{fig:3Dnumspeed}(a)).
\item The second velocity field is a vertical waveguide with Gaussian
  cross section (see Figure \ref{fig:3Dnumspeed}(b)).
\item The third velocity field is a random velocity field  (see
  Figure \ref{fig:3Dnumspeed}(c)).
\end{enumerate}

\begin{figure}[h!]
  \begin{center}
    \begin{tabular}{ccc}
      \includegraphics[height=1.6in]{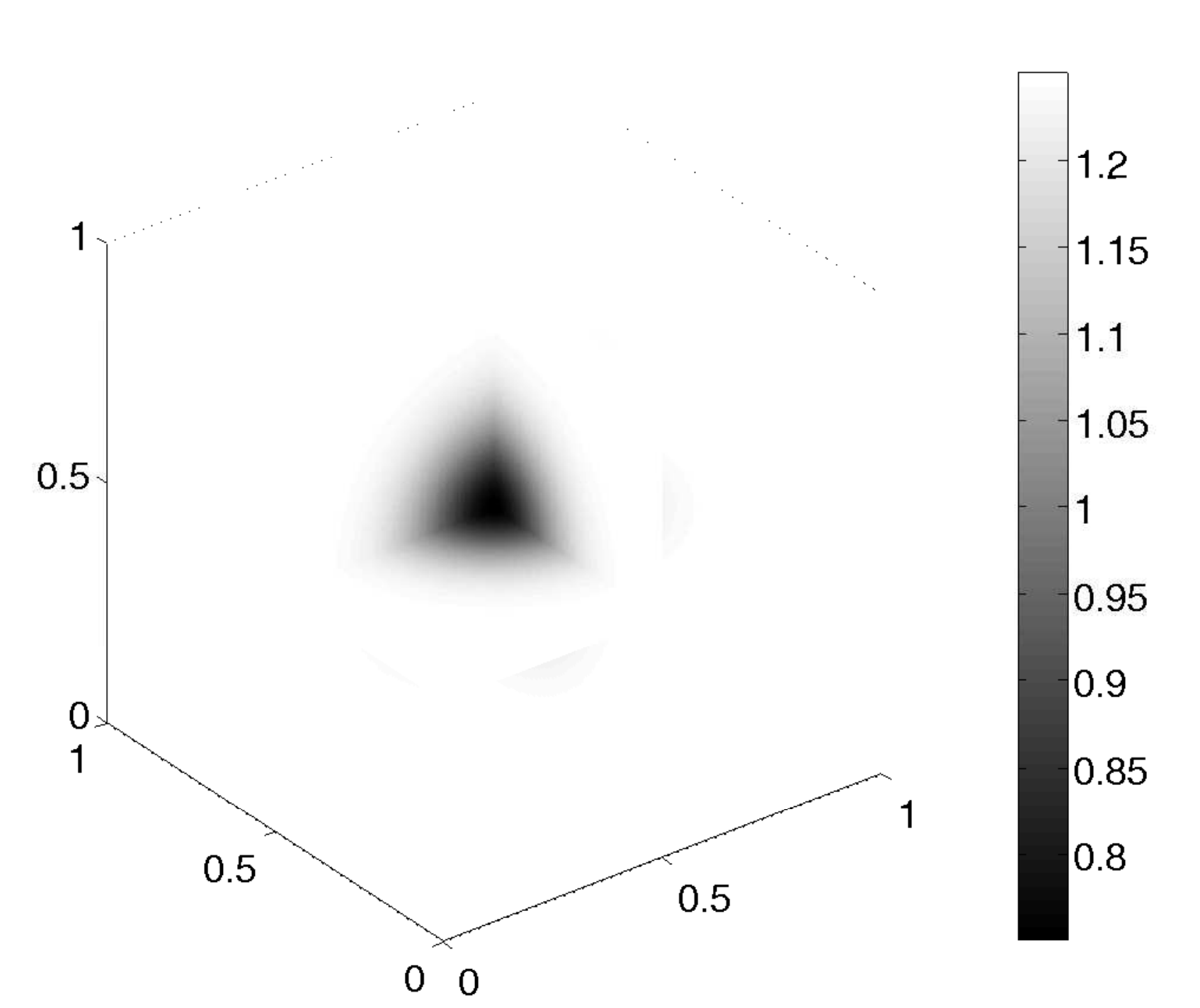}&\includegraphics[height=1.6in]{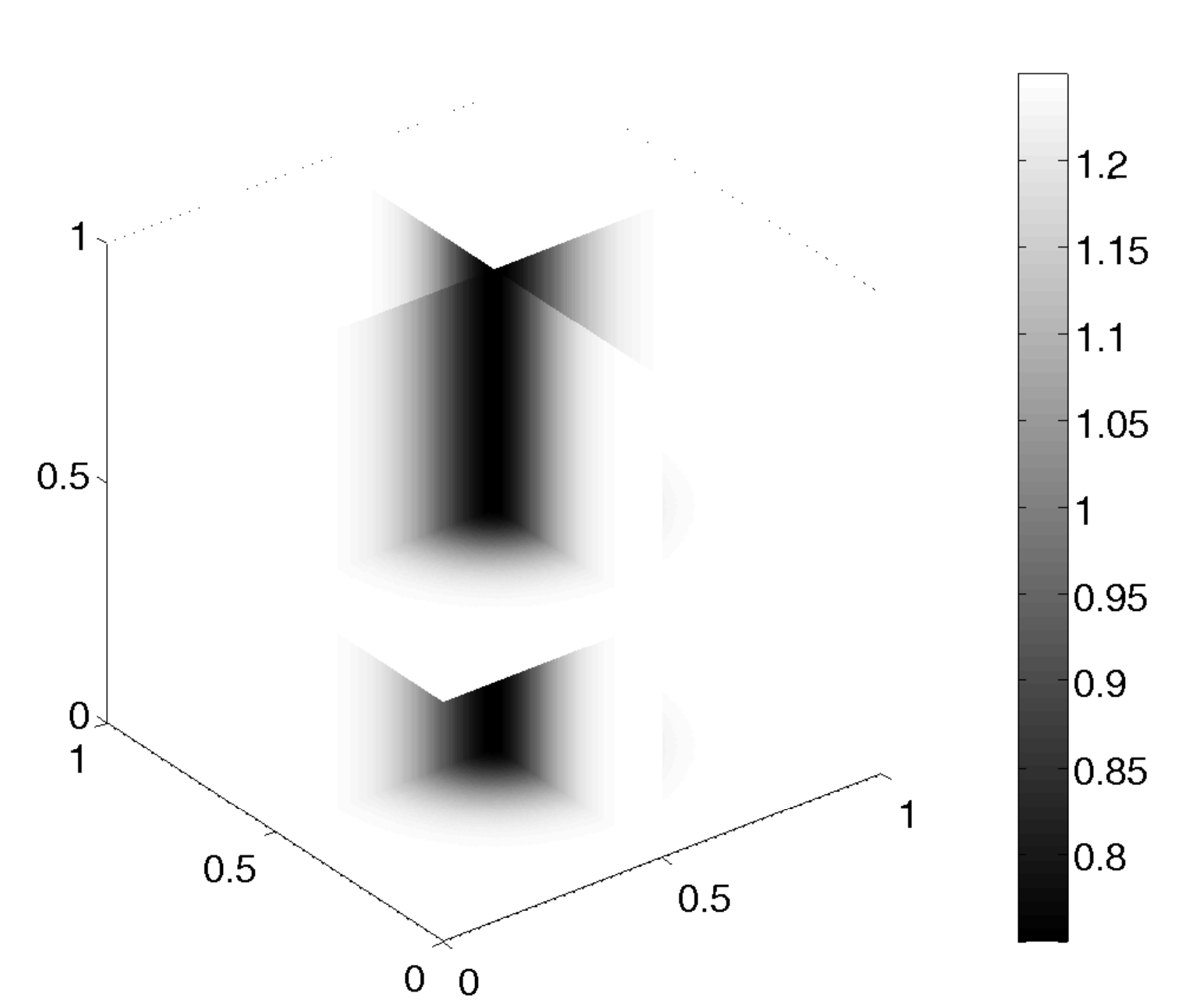}& \includegraphics[height=1.6in]{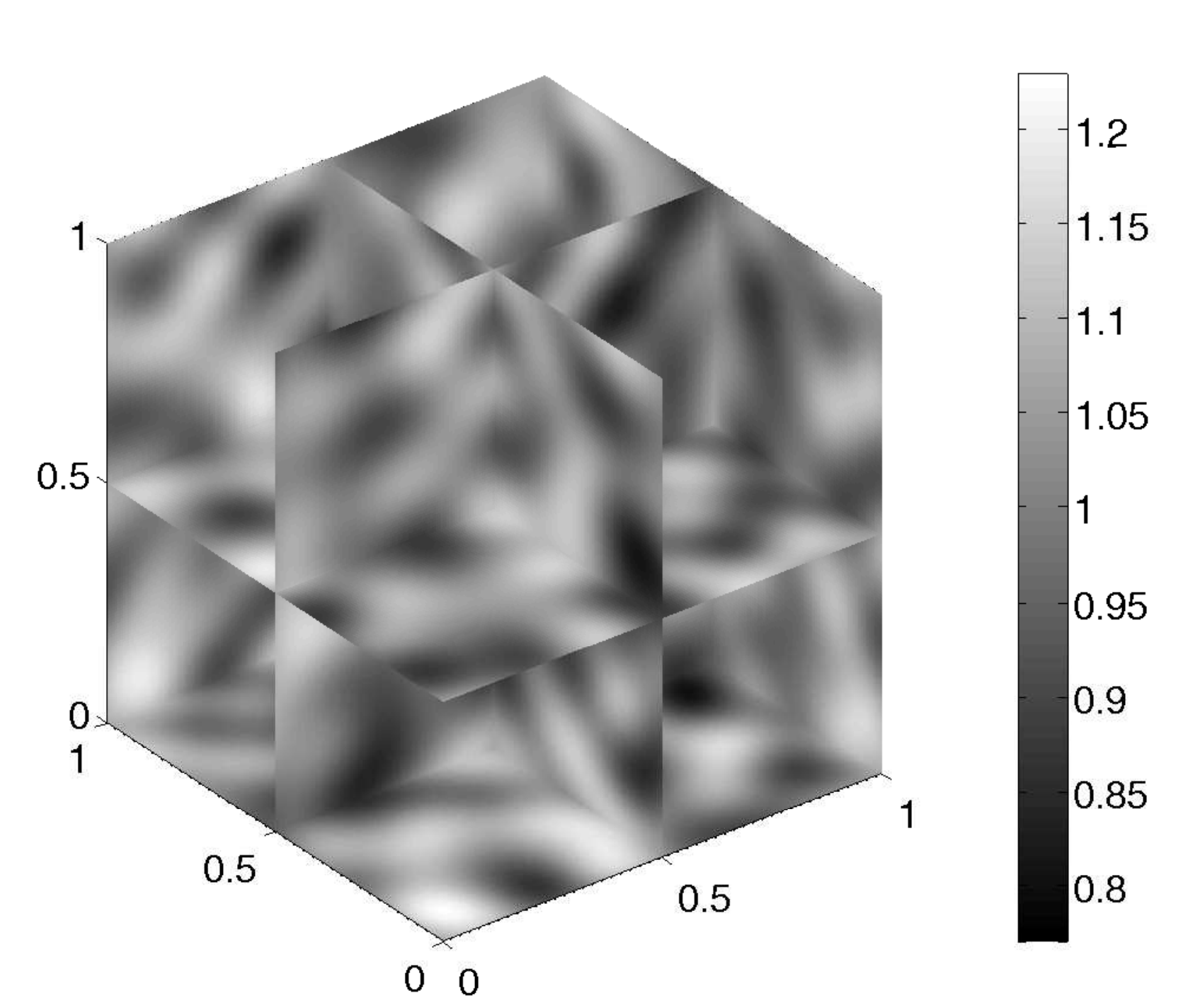}\\
      (a) & (b) & (c)
    \end{tabular}
  \end{center}
  \caption{Test velocity fields. For each velocity field, the cross
    sections at $x_1=0.5$, $x_2=0.5$, and $x_3=0.5$ are shown.}
  \label{fig:3Dnumspeed}
\end{figure}

For each problem, we test with two external forces $f(x)$.
\begin{enumerate}
\item The first external force $f(x)$ is a Gaussian point source
  located at $(x_1,x_2,x_3) = (0.5, 0.5, 0.25)$. The response of this
  forcing term generates spherical waves propagating at all
  directions. Due to the variations of the velocity field, the
  circular waves should bend and form caustics.
\item The second external force $f(x)$ is a Gaussian wave packet whose
  wavelength is comparable to the typical wavelength of the
  domain. This packet centers at $(x_1,x_2,x_3) = (0.5, 0.25, 0.25)$
  and points to the $(0,1,1)$ direction. The response of this forcing
  term generates a Gaussian beam initially pointing towards the
  $(0,1,1)$ direction.
\end{enumerate}

For each velocity field, we perform tests for $\frac{\omega}{2\pi}$
equal to $5,10,20$. In these tests, we discretize with $q=8$ points
per wavelength. Hence, the number of points in each dimension is
$n=40,80,160$. Recall that $R$ is the rank of the factorized form of
the hierarchical matrix representation. It is clear from the
discussion of Section \ref{sec:3Dprehier} that the value of $R$ should
grow with $\omega$ (and $n$). Here, we choose $R=2, 3, 4$ for
$\omega=5,10,20$, respectively. The sweeping direction is bottom-up
from $x_3=0$ to $x_3=1$.

The results of the first velocity field are reported in Table
\ref{tbl:3DPML1}. The two plots show the solutions of the two external
forces on a plane near $x_1 = 1/2$. $T_{\text{setup}}$ is the time
used to construct the preconditioner in seconds. $N_{\text{iter}}$ is
the number of iterations of the preconditioned GMRES solver and
$T_{\text{solve}}$ is the solution time. The analysis in Section
\ref{sec:3Dprehier} shows that the setup time scales like $O(R^2 n^3
\log^2n) = O(R^2 N \log^2 N)$. When $\omega$ grows from 5 to 20, since
$R$ increases from $2$ to $4$, $T_{\text{setup}}$ increases by a
factor of $20$ times each time $\omega$ doubles. Though the setup cost
grows significantly faster than the linear scaling $O(N)$, it is still
much better than the $O(N^2)$ scaling of the multifrontal method. A
nice feature of the sweeping preconditioner is that the number of
iterations is extremely small. In fact, in all cases, the GMRES solver
converges in at most 7 iterations. Finally, we would like to point out
that our algorithm is quite efficient: for the case with
$\omega/(2\pi)=20$ with more than four million unknowns, the solution
time is only about 3 minutes.

The results of the second and the third velocity fields are reported
in Tables \ref{tbl:3DPML2} and \ref{tbl:3DPML3}, respectively. In all
cases, the GMRES solver converges in at most 5 iterations when
combined with the sweeping preconditioner.

\begin{table}[h!]
  \begin{center}
    \includegraphics[height=2.3in]{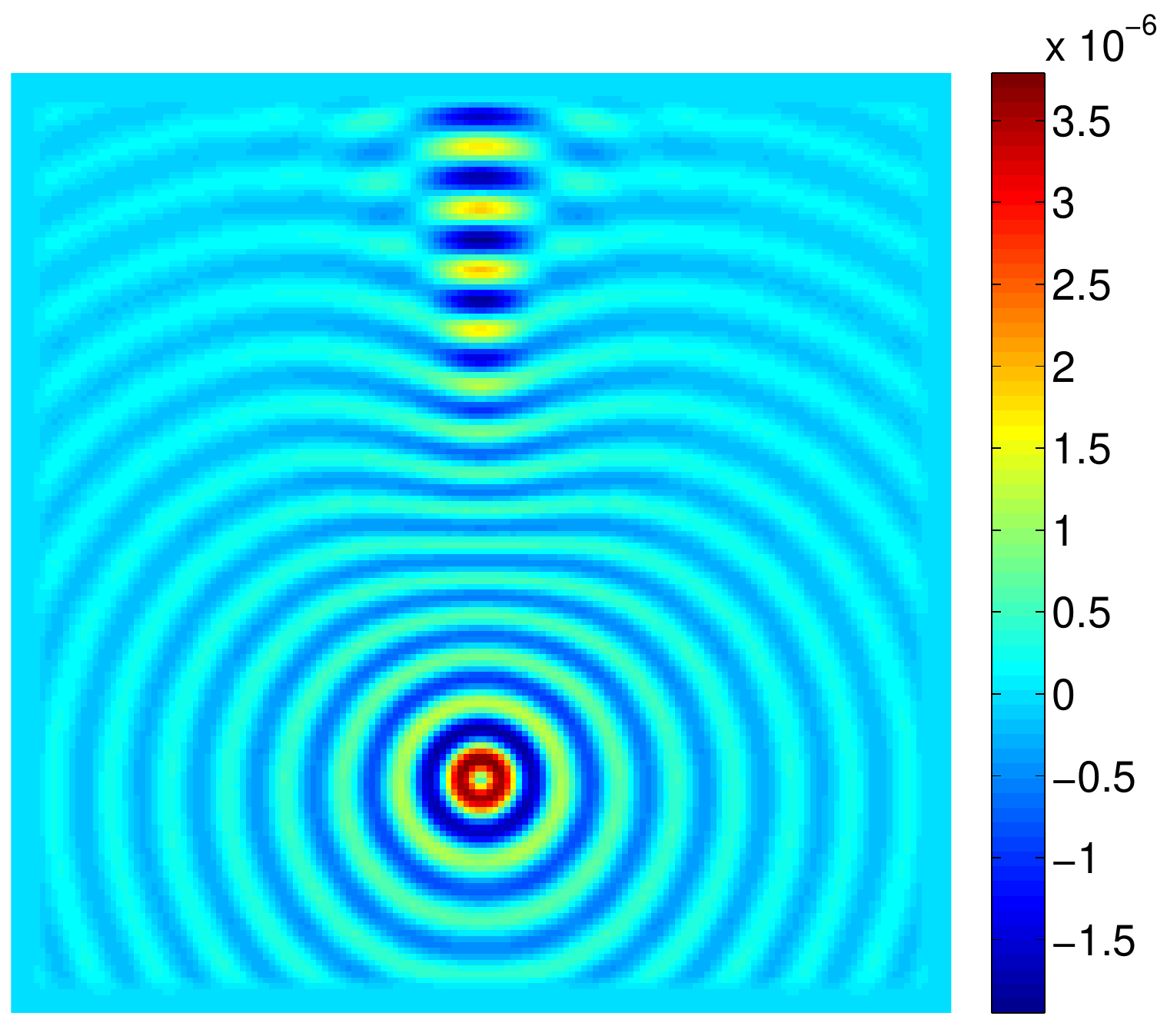}
    \includegraphics[height=2.3in]{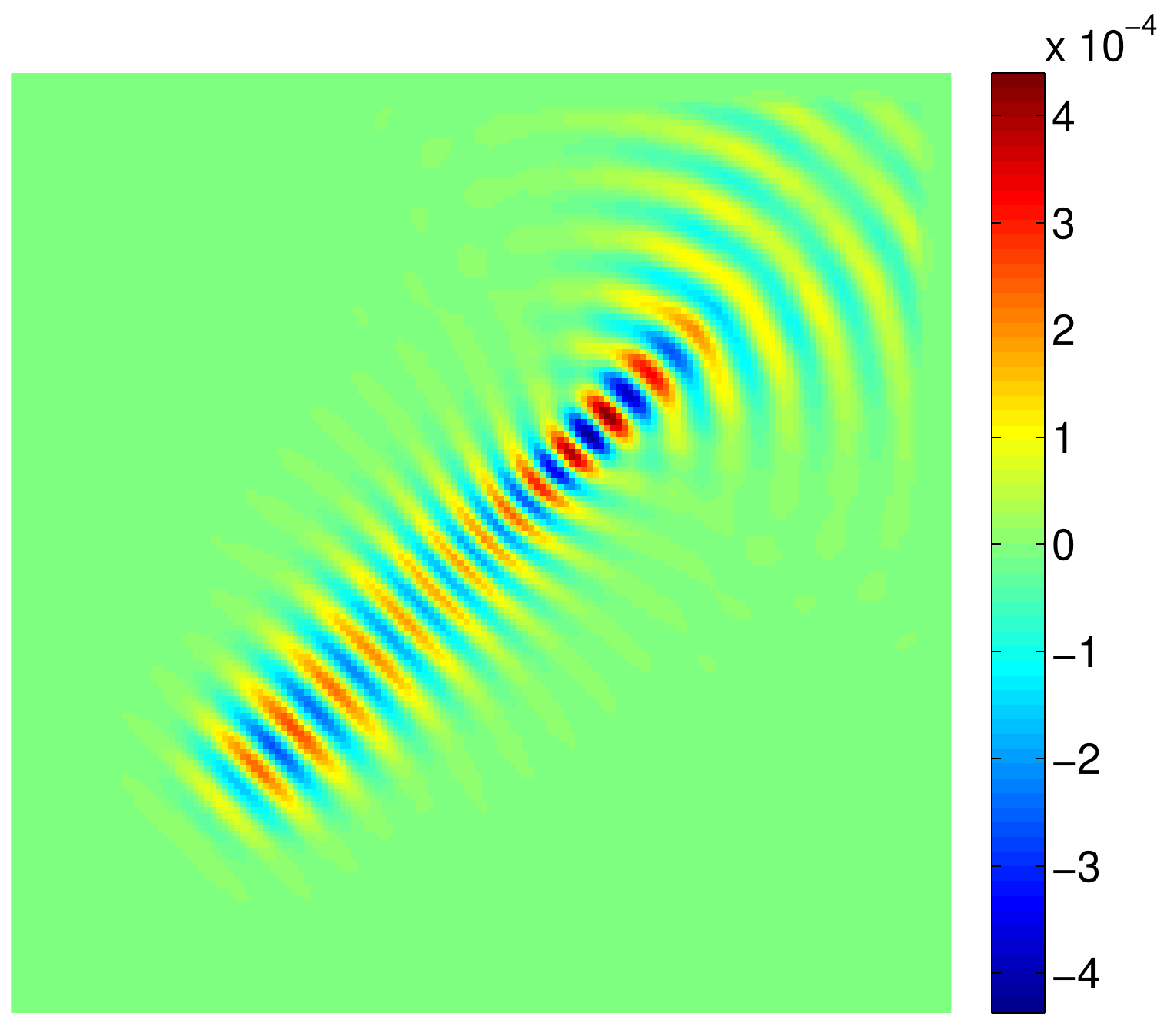}\\
    \begin{tabular}{|ccccc|cc|cc|}
      \hline
      \multicolumn{5}{|c|}{} & \multicolumn{2}{|c|}{Test 1} & \multicolumn{2}{|c|}{Test 2}\\
      \hline
      $\omega/(2\pi)$ & $q$ & $N=n^3$ & $R$ & $T_{\text{setup}}$ & $N_{\text{iter}}$ & $T_{\text{solve}}$ & $N_{\text{iter}}$ & $T_{\text{solve}}$\\
      \hline
      5  & 8 & $40^3$  & 2 & 8.99e+01 & 3 & 7.10e-01 & 3 & 7.20e-01 \\
      10 & 8 & $80^3$  & 3 & 2.30e+03 & 7 & 1.87e+01 & 5 & 1.40e+01 \\
      20 & 8 & $160^3$ & 4 & 4.73e+04 & 6 & 1.90e+02 & 5 & 1.61e+02 \\
      \hline
    \end{tabular}
  \end{center}
  \caption{Results of velocity field 1 for different $\omega$. 
    Top: Solutions for two external forces with $\omega/(2\pi)=20$ on a plane near $x_1=0.5$.
    Bottom: Results for different $\omega$.  }
  \label{tbl:3DPML1}
\end{table}

\begin{table}[h!]
  \begin{center}
    \includegraphics[height=2.3in]{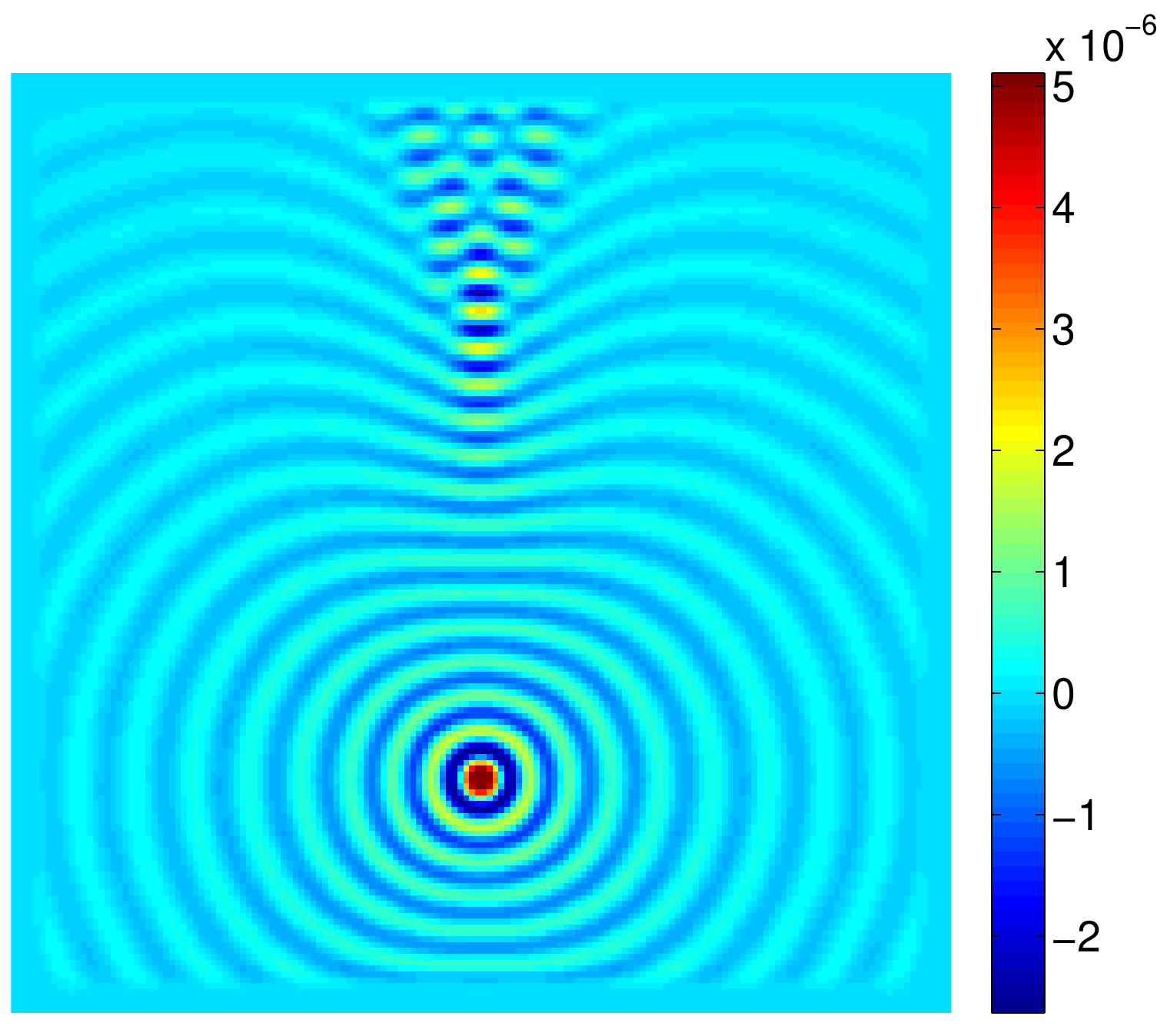}
    \includegraphics[height=2.3in]{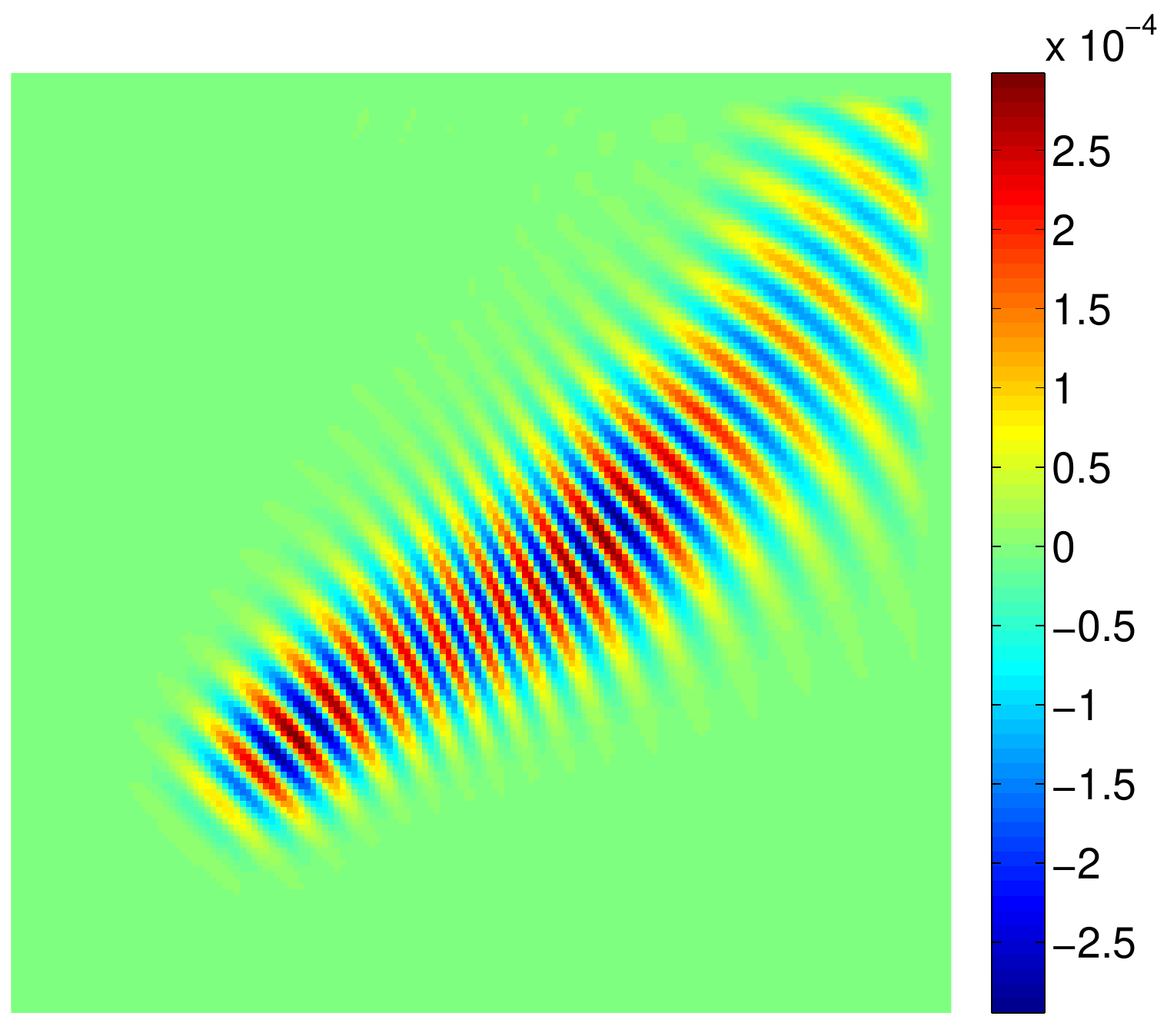}\\
    \begin{tabular}{|ccccc|cc|cc|}
      \hline
      \multicolumn{5}{|c|}{} & \multicolumn{2}{|c|}{Test 1} & \multicolumn{2}{|c|}{Test 2}\\
      \hline
      $\omega/(2\pi)$ & $q$ & $N=n^3$ & $R$ & $T_{\text{setup}}$ & $N_{\text{iter}}$ & $T_{\text{solve}}$ & $N_{\text{iter}}$ & $T_{\text{solve}}$\\
      \hline
      5  & 8 & $40^3$  & 2 & 8.95e+01 & 3 & 7.10e-01 & 3 & 7.00e-01 \\
      10 & 8 & $80^3$  & 3 & 2.35e+03 & 5 & 1.40e+01 & 3 & 9.38e+00 \\
      20 & 8 & $160^3$ & 4 & 4.73e+04 & 4 & 1.38e+02 & 4 & 1.34e+02 \\
      \hline
    \end{tabular}
  \end{center}
  \caption{Results of velocity field 2 for different $\omega$. 
    Top: Solutions for two external forces with $\omega/(2\pi)=20$ on a plane near $x_1=0.5$.
    Bottom: Results for different $\omega$.  }
  \label{tbl:3DPML2}
\end{table}

\begin{table}[h!]
  \begin{center}
    \includegraphics[height=2.3in]{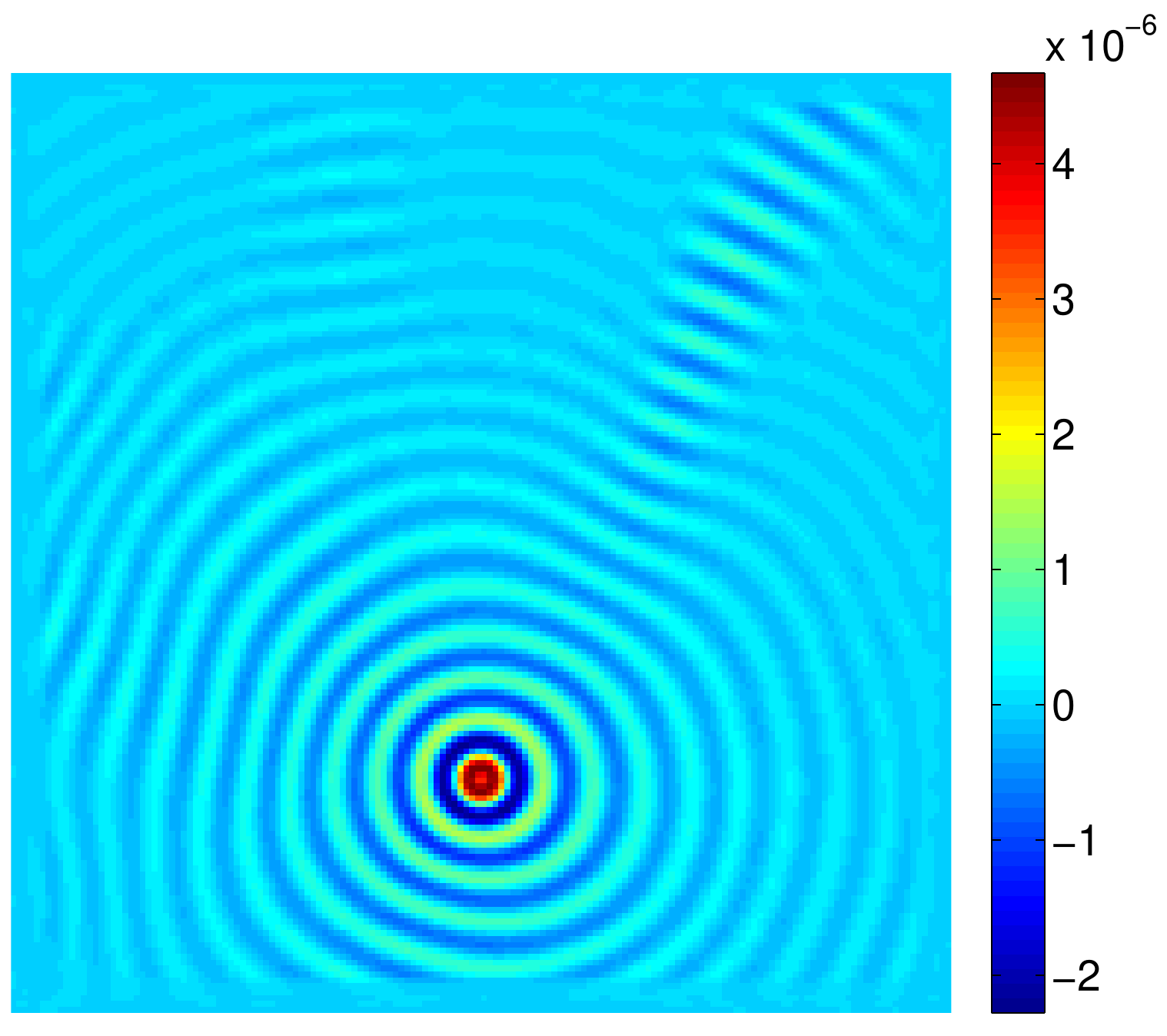}
    \includegraphics[height=2.3in]{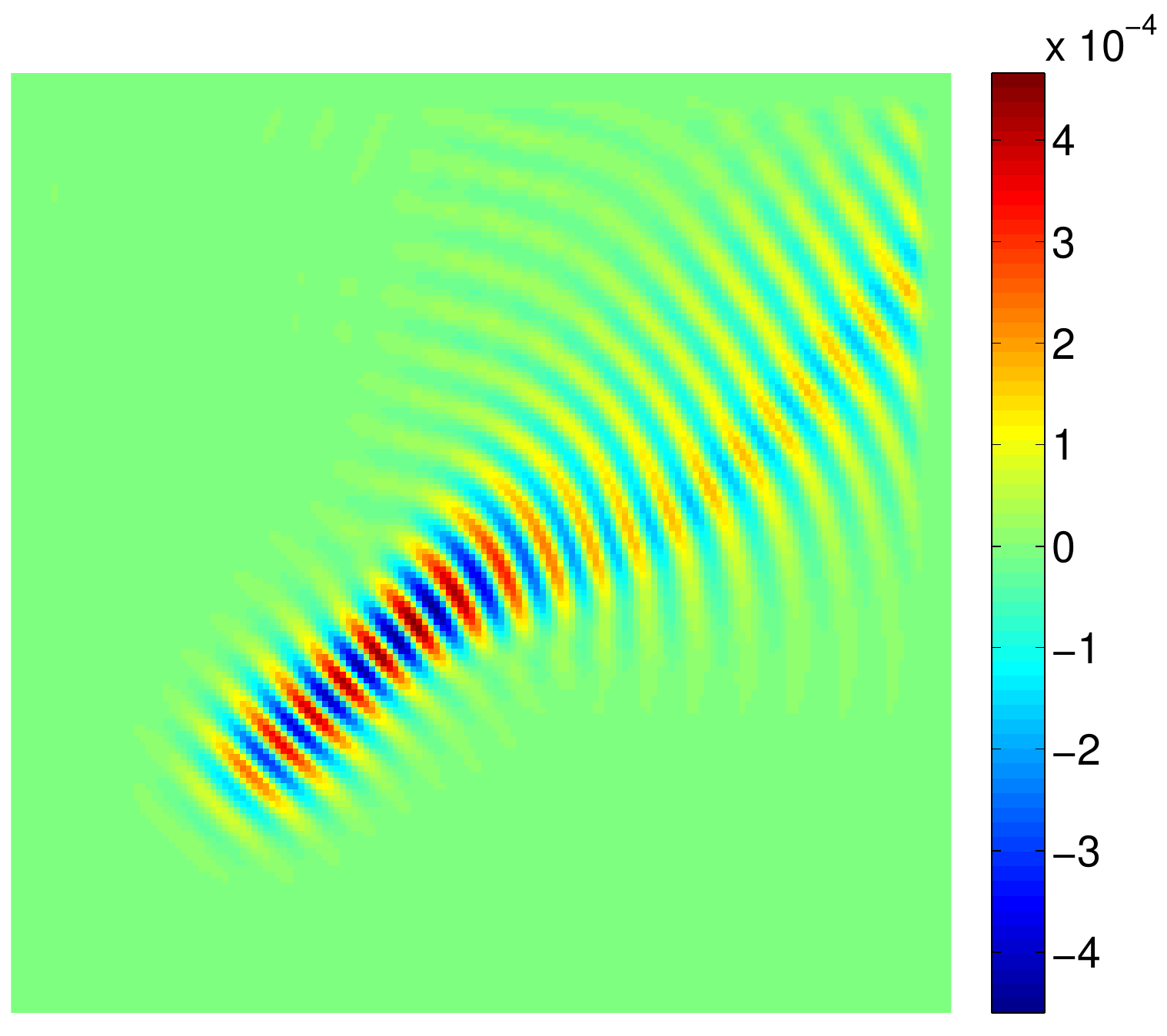}\\
    \begin{tabular}{|ccccc|cc|cc|}
      \hline
      \multicolumn{5}{|c|}{} & \multicolumn{2}{|c|}{Test 1} & \multicolumn{2}{|c|}{Test 2}\\
      \hline
      $\omega/(2\pi)$ & $q$ & $N=n^3$ & $R$ & $T_{\text{setup}}$ & $N_{\text{iter}}$ & $T_{\text{solve}}$ & $N_{\text{iter}}$ & $T_{\text{solve}}$\\
      \hline
      5  & 8 & $40^3$  & 2 & 9.00e+01 & 3 & 7.20e-01 & 3 & 7.20e-01 \\
      10 & 8 & $80^3$  & 3 & 2.37e+03 & 4 & 1.22e+01 & 3 & 9.90e+00 \\
      20 & 8 & $160^3$ & 4 & 4.74e+04 & 4 & 1.37e+02 & 3 & 1.07e+02 \\
      \hline
    \end{tabular}
  \end{center}
  \caption{Results of velocity field 3 for different $\omega$. 
    Top: Solutions for two external forces with $\omega/(2\pi)=20$ on a plane near $x_1=0.5$.
    Bottom: Results for different $\omega$.  }
  \label{tbl:3DPML3}
\end{table}

\section{Conclusion and Future Work}
\label{sec:conc}

In this paper, we have proposed a sweeping preconditioner for the
iterative solution of variable coefficient Helmholtz equations in two
and three dimensions. The construction of the preconditioner is based
on an approximate block $LDL^t$ factorization that eliminates the
unknowns layer by layer starting from an absorbing layer. By
representing and manipulating the intermediate Schur complement
matrices in the hierarchical matrix framework, we have obtained
preconditioners with almost linear cost. Numerical examples
demonstrate that, when combined with standard iterative solvers, these
new preconditioners result almost $\omega$-independent iteration
numbers.

Some questions remain open. First, in the 2D case, we have proved the
compressibility result under the constant coefficient case. A natural
question is to what extent this is still true for a general velocity
field.

The hierarchical matrix representation may not be very accurate for
the Schur complement matrices in 3D, since some high-rank off-diagonal
blocks are stored in a low-rank factorized form. Yet our algorithm
works well with very small iteration numbers. It is important to
understand why this is the case and also to investigate whether other
matrix representations would be able to provide more accurate
approximations for $T_m$.

The memory space required by the sweeping preconditioners is linear
with respect to the number of unknowns. However, the prefactor is
higher compared to the shifted Laplacian preconditioners and the ILU
preconditioners. Most of the memory space is in fact used to store the
diagonal part of $T_m$, which corresponds to the local part of the
half-space Green's function. One improvement is to use the asymptotic
formula of the Green's function to represent the local part
analytically and this can eliminate the need of storing the diagonal
part of the hierarchical matrices.

The matrix representation used here is often referred as the the $\H^1$
form of the hierarchical matrix algebra. More efficient and
sophisticated versions are the uniform $\H^1$ form and the $\H^2$
form. For our problem, Algorithm \ref{alg:setup} requires the matrices
to be represented in the $\H^1$ form since it uses the matrix
inversion procedure. However, Algorithm \ref{alg:solve} of applying
the sweeping preconditioner can potentially speed up dramatically when
the $\H^2$ form is used.

We have chosen the PML for the numerical implementation of the
Sommerfeld condition. Many other boundary conditions are available and
commonly used. The sweeping approach should work for these boundary
conditions, as we have briefly demonstrated for the second order
ABC. The design and implementation of these other boundary conditions
should minimize non-physical reflections in order for the sweeping
preconditioner to do well. 

The second order central difference scheme is used to discretize the
Helmholtz equation in this paper. We would like to investigate other
more accurate stencils and other types of discretizations such as
$h/p$ finite elements, spectral elements, and discontinuous Galerkin
methods.

Since high frequency fields typically oscillate rapidly on a similar
scale throughout the computational domain, uniform grids are very
common. There are however situation where unstructured grids would be
natural. The sweeping approach and more general hierarchical matrix
representations can also be used in this context. The challenge here
is to maintain compatibility between the matrix representation and the
geometry as one sweeps through the computational domain. In a second
paper \cite{EngquistYing:10b} another variant of sweeping
preconditioning is presented, which is more flexible with respect to
unstructured and adaptive grids.

The sequential nature of the sweeping approach complicates
parallelization of the algorithm. One possibility is to use parallel
hierarchical matrix representation for each layer. This would
parallelize an inner part of the algorithm. Another technique would
leverage the idea of domain decomposition and use the sweeping
preconditioner within each subdomain. The subdomains should then be
coupled with absorbing boundary conditions.

The Helmholtz equation is only the simplest example of time-harmonic
wave equations. Other cases include elasticity equation and Maxwell
equations. For these more complicated systems, multiple wave numbers
coexist even for the constant coefficient case. The basic idea of the
sweeping preconditioner should apply but the details need to be worked
out.


\bibliographystyle{abbrv}
\bibliography{ref}

\end{document}